\documentclass[9 pt,reqno,a4paper,oneside]{report}
\usepackage{setspace}
\usepackage{extsizes}
\usepackage{graphicx}
\usepackage[dvipsnames]{xcolor}
\usepackage{mathrsfs}
\usepackage{amsmath,amssymb,enumerate}
\usepackage{epsfig,fancyhdr,color}
\usepackage{epstopdf}
\usepackage{amssymb}
\usepackage{amsmath,amsthm}
\usepackage{latexsym}
\usepackage{amscd}
\usepackage{psfrag}
\usepackage{url}
\usepackage{stfloats}

\usepackage{epsf}
\usepackage[latin1]{inputenc}
\usepackage[all]{xy}
\usepackage{tikz}
\usepackage{prettyref}
\usepackage{subcaption}
\usepackage{float}
\usepackage{color}
\usepackage{array}
\usepackage{cite}
\usepackage[totalwidth=15cm,totalheight=21.5cm]{geometry}

\newcommand{\II}{I\hspace{-0.1cm}I}
\newcommand{\III}{I\hspace{-0.1cm}I\hspace{-0.1cm}I}

\newcommand{\C}{{\mathbb C}}

\newcommand{\R}{{\mathbb R}}

\makeatletter
\newcommand{\myitem}[1]{%
	\item[#1]\protected@edef\@currentlabel{#1}%
}
\makeatother

\newcommand{\PSL}{\rm{PSL}}

\newcommand{\Isom}{\rm{Isom}}

\newcommand{\f}{\mathsf{F}}
\newcommand{\g}{\mathsf{G}}
\newcommand{\fp}{\mathsf{F_{+}}}

\newcommand{\fm}{\mathsf{F_{-}}}

\newcommand{\fpm}{\mathsf{F_{\pm}}}

\newcommand{\ddt}{\frac{d}{dt}\Bigr|_{\substack{t=0}}}
\newcommand{\ddtt}{\frac{d}{ds}\Bigr|_{\substack{s=0}}}

\newcommand{\red}[1]{\textcolor{black}{#1}}
\newcommand{\deter}{\mathsf{det}}
\newcommand{\T}{\mathcal{T}(S)}
\newcommand{\QF}{\mathcal{QF}(S)}
\newcommand{\MF}{\mathcal{MF}(S)}
\newcommand{\FMF}{\mathcal{FMF}(S)}
\newcommand{\AF}{\mathcal{AF}(S)}
\newcommand{\F}{\mathcal{F}(S)}
\newcommand{\Wfp}{\mathcal{W}^{+}_{\mathsf{F}_{+}}}
\newcommand{\Wfm}{\mathcal{W}^{-}_{\mathsf{F}_{-}}}
\newcommand{\Wf}{\mathcal{W}_{\f}}
\newcommand{\Sig}{S}

\newcommand{\ML}{\mathcal{ML}(S)}

\newcommand{\p}{\mathsf{p}}
\newcommand{\pp}{\mathsf{P}}

\newcommand{\ext}{\mathsf{ext}}
\newcommand{\CP}{\C\mathrm{P}}
\newcommand{\D}{{\mathbb D}}
\newcommand{\Hyp}{\mathbb{H}}
\newcommand{\Eps}{\mathrm{Eps}}
\newcommand{\Q}{{\mathcal{Q}}}

\DeclareMathOperator{\trace}{tr}

\DeclareMathOperator{\hess}{Hess}

\newtheorem{theorem}{\rm\bf Theorem}[section]
\newtheorem{proposition}[theorem]{\rm\bf Proposition}
\newtheorem{lemma}[theorem]{\rm\bf Lemma}
\newtheorem{corollary}[theorem]{\rm\bf Corollary}
\newtheorem{definition}[theorem]{\rm\bf Definition}
\newtheorem{remark}[theorem]{\rm\bf Remark}

\newtheorem{question}[theorem]{\rm\bf Question}

\def\interieur#1{\mathord{\mathop{\kern 0pt #1}\limits^\circ}}

\usepackage{hyperref}
\hypersetup{hidelinks}

\begin{document}
	\date{}
		\title{Measured foliations at infinity and CMC foliations of quasi-Fuchsian manifolds close to the Fuchsian locus }
	\author{Diptaishik Choudhury }
	\maketitle

	\begin{abstract}
		The main subject of this thesis are a certain class of hyperbolic $3$-manifolds called quasi-Fuchsian manifold. Given an orientied, closed hyperbolic surface $S$, these manifolds are homeomorphic to $S\times \mathbb{R}$. We study two questions regarding them: one is on \textit{measured foliations at infinity} and the other is on \textit{foliation by constant mean curvature surfaces}. \\\\
	Measured foliations at infinity of quasi-Fuchsian manifolds are a natural analog at infinity to the measured bending laminations on the boundary of its convex core. Given a pair of measured foliations $(\f_{+},\f_{-})$ which fill a closed hyperbolic surface $S$ and are {arational}, we prove that for $t>0$ sufficiently small $t\fp$ and $t\fm$ can be uniquely realised as the {measured foliations at infinity} of a quasi-Fuchsian manifold homeomorphic to $S\times \mathbb{R}$, which is sufficiently close to the Fuchsian locus. The proof is based on that of Bonahon in \cite{bonahon05} which shows that a quasi-Fuchsian manifold close to the Fuchsian locus can be uniquely determined by the data of filling measured bending laminations on the boundary of its convex core. Finally, we interpret the result in {half-pipe} geometry.\\\\For the second part of the thesis we deal with a conjecture due to Thurston asks if almost-Fuchsian manifolds admit a foliation by CMC surfaces. Here, almost-Fuchsian manifolds are defined as quasi-Fuchsian manifolds which contain a unique minimal surface with principal curvatures in $(-1,1)$ and it is known that in general, quasi-Fuchsian manifolds are not foliated by surfaces of constant mean curvature (CMC) although their ends are. However, we prove that almost-Fuchsian manifolds which are sufficiently close to being Fuchsian are indeed monotoni- cally foliated by surfaces of constant mean curvature.
	This work is in collaboration with Filippo Mazzoli and Andrea Seppi.
	
	\end{abstract}
	\tableofcontents
	
\chapter*{Acknowledgements}

First of all, I will like to express my gratitude to Jean-Marc Schlenker. For appointing me as a Ph.D. student and giving me the privilege to be a doctoral candidate at the University of Luxembourg, then providing me with a good thesis topic, and then patiently guiding me there onwards in finishing the Ph.D. through a lot of difficult situations. He also 
introduced me to Andrea Seppi to whom also I owe a lot. He has also been very patient, and attentive to every minute detail. Half of my thesis is directly due to working with him. I also thank them as they have given me enough opportunities to present my work at various conferences too.
Also, Greg McShane helping me to get a Ph.D. position after my M2 and then monitoring my work, is needless to say something I am grateful for.\\
There have been a few other people who have played an important role in my education, namely Pierre
Dehornoy, Erwan Lanneau, and Thierry Barbot. They, especially Pierre, have been extremely kind in
inviting me through internships and guiding me in my initial years of studies in Europe. Also thanks to
Erwan, Sarah Sherotzke and Thierry for agreeing to be part of my defence committee and François Fillastre and Thierry
for being referees for my thesis. I also take this chance to thank the professors I have had in Grenoble and
Luxembourg for their lectures.\\
Thanks are due to all the professors, researchers and students who have invited me to give a talk in their respective seminars. The conversations I had with each and every one of them are something I am grateful for.\\
Filippo Mazzoli has not only been an academic brother but I have been very fortunate to have his plethora
of knowledge and expertise to learn from. Without mentioning names, I also want to mention the people
in the departments of Grenoble and Luxembourg with whom I have passed many good moments.\\
Finally, I will like to mention a lot of people with whom I had positive exchanges like all my flat/roommates
at different times and also the landlord/ladies; all the people I know from various university events or people I have met elsewhere and also all the secretaries and staff. The people who have given me the space to
discuss math, complain, dine, nag, etc have all been very generous. I feel very foolish for being annoying
to some of them at times too.\\
Lastly, nothing would have been possible without the support of my parents, so I want to mention Dipankar
and Jayasri for their kindness, care and for teaching me to be hopeful.
	
\chapter{Introduction}	
\hspace*{.07 cm}  	\hspace*{.07 cm}  Let $S$ be a closed, oriented surface with genus $g \geq 2$ and $M$ a $3$-manifold homeomorphic to $S\times \mathbb{R}$. Call the space of isotopy classes of Fuchsian metrics on $M$ as the {Fuchsian locus} $\F$ and note that it can also be identified with the Teichm{\"u}ller space $\T$ (see \S \S  \ref{teichspace}). Let $T^{*}_{[c]}\T$ be its cotangent space at a point $[c]\in \T$ and again identify it with the space of holomorphic quadratic differentials $Q(S,[c])$ on $(S,[c])$ (see \S \S \ref{HQDintro}). Now, consider quasi-Fuchsian hyperbolic metrics on $M$ and let $\QF$ denote the space of isotopy classes of quasi-Fuchsian metrics on $M$ (see \S\S \ref{introtofuchsian}). Denote the connected components of the boundary at infinity of $M$ as $\partial^{+}_{\infty}M$ and $\partial^{-}_{\infty}M$ (both being homeomorphic to $S$) and let $([c_{+}],[c_{-}])\in \mathcal{T}(\partial^{+}_{\infty}M)\times \mathcal{T}({\partial^{-}_{\infty}M})$ be the respective conformal classes (see Theorem \ref{bersthm}).
\section{Measured foliations at infinity} 
There are unique holomorphic maps, well-defined up to right composition by M{\"o}bius transformations, from the universal covers $\widetilde{\partial^{+}_{\infty}M}, \widetilde{\partial^{-}_{\infty}M}\subset \partial_{\infty}\mathbb{H}^{3}\cong \mathbb{C}P^{1}$ to the unit disc $\Delta\subset \mathbb{C}$ that we obtain by uniformising the respective complex structures (see \S \S \ref{introductionary}). Let the Schwarzians at infinity $\sigma_{+}\in Q(\partial^{+}_{\infty}M,[c_{+}])$ and $\sigma_{-}\in Q(\partial^{-}_{\infty}M,[c_{-}])$ be the holomorphic quadratic differentials obtained by taking the Schwarzian derivative of these maps respectively and passing to quotients. We define the measured foliations at infinity $\fp,\fm$ of a quasi-Fuchsian manifold $M$ as the horizontal measured foliations of $\sigma_{+},\sigma_{-}$ on $(\partial^{+}_{\infty}M,[c_{+}]),(\partial^{-}_{\infty}M,[c_{-}])$ respectively. These measured foliations at infinity can be seen as a natural analog at infinity to the bending lamination on the boundary of the convex core of a quasi-Fuchsian manifold (see \S\S \ref{schwatinf}, Lemma \ref{graft}). \\
\hspace*{.5 cm} Let $\MF$ denote the space of equivalence classes of measured foliations on $S$  (see  \S \S\ref{defmf}, \S \S \ref{horifoli},\cite{AST_1979__66-67_}) and $\fp,\fm \in \MF$. Further, given a pair of measured foliations $(\f,\g)\in \MF\times \MF$ we have the notion of them being a pair which {fills} $\Sig$. That is to say, any other measured foliations $\mathsf{H}$ has non-zero intersection with both $\f$ or $\g$ (see Definition \red{\ref{fillupdef}} and \red{\S\S\ref{gardinermasur}}).
So we ask (see Question $7.4$ in\cite{schlenker20}) whether is it possible to determine a quasi-Fuchsian manifolds uniquely by its measured foliations at infinity? \\
\hspace*{.5 cm}Now let $\mathcal{MF}_{0}(S)\subset \MF$ be the subspace of measured foliations which are arational, i.e, all the prongs are of order $3$ and there are no leaves joining the prongs (see Definition \ref{arationaldef} and Lemma \ref{arationaliff}); $\FMF$ the space of all pairs of measured foliations that fill $S$ and $\mathcal{FMF}_{0}(S)$ be the subspace of such pair which are arational. If the pair $(\fp,\fm)$ belongs to $ \mathcal{FMF}_{0}(S)$, then so do the pair $(t\fp,t\fm)$, for all $t>0$ (see \S \S \ref{defmf}). Note also that for a metric $g\in \F$ the Schwarzians and the measured foliations at infinity  are zero (see \S \S \ref{introductionary}). The result of principal interest that answers the above question partially for quasi-Fuchsian manifolds near the Fuchsian locus is:
{\theorem \label{thm1.1}
	For every pair of measured foliations $(\fp,\fm)$ which are arational and fill $S$, there exists an $\epsilon_{\fpm}>0$ such that for $\forall t \in (0,\epsilon_{\fpm})$ there exists an unique quasi-Fuchsian metric $g\in \QF$ on $M$ sufficiently close to the Fuchsian locus, whose measured foliations at infinity are given by  $t\fp$ and $t\fm$ .}\\\\
That is, given the map $\mathfrak{F}: \QF \rightarrow \mathcal{MF}(\partial^{+}_{\infty}M) \times \mathcal{MF}({\partial^{-}_{\infty}M)}$ sending a quasi-Fuchsian metric to the measured foliations at infinity at the positive and negative end respectively; we have a unique solution $g\in \QF$ to the equation $\mathfrak{F}(g)=(t\fp,t\fm)$ when restricted to $(\fp,\fm)\in \mathcal{FMF}_{0}(S)$ for $t>0$ small enough. Now let $q^{\mathsf{H}}_{[c]}$ be the unique holomorphic quadratic differentials realising  $\mathsf{H}\in \MF$ as its horizontal measured foliation on $(S,[c])$ (see \S \S \ref{HMsection}). An immediate consequence along the lines of McMullen's quasi-Fuchsian reciprocity (see\cite{mcmullen,Krasnov2009,filippoksurface}) which helps in describing the Schwarzians at infinity is that if $g$ be a quasi-Fuchsian metric on $M$ such that the measured foliations at infinity are given as $(t\fp,t\fm)$ for some filling arational pair $(\fp,\fm)$, $t>0$ small enough; then the Schwarzians at infinity of $(M,g)$ are $t^{2}q^{\fp}_{[c_{+}]}\in Q(\partial^{+}_{\infty}M,[c_{+}])$ and $t^{2}q^{\fm}_{[c_{-}]}\in Q(\partial^{-}_{\infty}M,[c_{-}])$ respectively.\\\\ We then consider the case of {quasi-Fuchsian half-pipe} manifolds (see Definition \red{\ref{quasifuchsianhalfpipe}}, also\cite{Danciger2013,barbot,Fillastre2019}). These are intermediary geometric structures that arise naturally when we consider smooth transitions between hyperbolic and {anti-de Sitter} structures on $M$ via the Fuchsian locus; the bending laminations on the convex core boundary of the latter being a well studied topic as well. We define an analogous notion for {half-pipe Schwarzians} in this situation (see Definition \red{\ref{definitionofhpsch}}) and prove:
{\theorem\label{thm1.3} Given any pair of filling measured foliations $\fp,\fm$, there exists a unique quasi-Fuchsian half-pipe manifold such that the horizontal measured foliations of its positive and negative half-pipe Schwarzians are given by $\fp$ and $\fm$ respectively.}

\subsection{Analogy between bending laminations and measured foliations at infinity}\label{schwatinf}\hfill \break 
\hspace*{.3 cm} There are a few points of analogies between the data on the boundary at infinity and that on the boundary of the convex core of a quasi-Fuchsian manifold which makes Theorem \ref{thm1.1} really interesting. We denote the {convex core} of $M$, as $\mathcal{CC}(M)$, as the smallest non-empty convex compact subset contained in $M$ and it is homeomorphic to $S\times [-1,1]$. Call $\partial^{+}\mathcal{CC}(M)$ and $\partial^{-}\mathcal{CC}(M)$ (see \S \ref{introtofuchsian}) as the two boundary components and let the induced metric be called $m_{+}$ and $m_{-}$ respectively. There is a conjecture of Thurston regarding parametrization of quasi-Fuchsian metrics on $M$ uniquely by the data $(m_{+},m_{-})$ (see\cite{Canary2011,sullivan,Labourie1992}). The components $\partial^{\pm}\mathcal{CC}(M)$ moreover carry two {measured geodesic laminations} $\lambda_{+}$ and $\lambda_{-}\in\ML$ where, $\mathcal{ML}(\Sig)$ is the space of measured geodesic laminations on $S$ up to  equivalence (see\cite{bonahon05}). These are called the bending laminations and $\partial^{\pm}\mathcal{CC}(M)$ are bent along leaves of $\lambda_{\pm}$ respectively with the bending angle being given by the transverse measures associated to $\lambda_{\pm}$.\\\hspace*{.5cm} The similarity between the variational formulae for the dual volume $V^{*}_{C}(M)$ of $\mathcal{CC}(M)$ (see \cite{Krasnov2009}) and the renormalised volume $V_{R}$ of $M$ (see \cite{Krasnov2008}) makes Theorem \ref{thm1.1} really interesting as well. Suppose for $0\leq t<\epsilon$, we have a differentiable path of quasi-Fuchsian metrics on $M$ given by $t\mapsto M_{t}$, then the formula for the first-order variation of the renormalised volume is given by (\cite{schlenker20}): 
\begin{align}\label{renormazlised volume}
	\ddt V_{R}(M_{t}) = -\frac{1}{2}d(\ext(\fp))(\ddt [c^{t}_{+}])	
\end{align} where $[c^{t}_{+}]$ denotes the variation of the complex structure (up to equivalence) on $\partial^{+}_{\infty}M_{t}$ and for a measured foliation $\f\in \MF$ we have the function $\ext(\f): \T \rightarrow \mathbb{R}$ sending a conformal class $[c]\in \T$ to the extremal length $\ext_{[c]}(\f)$ of the foliation in that conformal class (see \S \S \ref{extremal}). On the other hand, the first order variation of the dual volume, via an application of the Bonahon-Schl{\"a}fli formula is expressed as (\cite{Krasnov2009,Mazzoli2021}):
\begin{align}\label{dualvolume}
	\ddt V^{*}_{C}(M_{t})=-\frac{1}{2}d(l(\lambda_{+}))(\ddt m^{t}_{+})
\end{align} where for a measured geodesic lamination $\lambda\in \ML$ we have the function $l(\lambda): \T \rightarrow \mathbb{R}$ sending a hyperbolic metric $m\in \T$ to the length of $\lambda$, denoted as $l_{m}(\lambda)$, measured with respect to this metric and $m^{t}_{+}$ denotes the variation of the induced metric on the convex core boundary under the variation of the quasi-Fuchsian structure. Here, we note that for a given measured foliation $\f$ and measured lamination $\lambda$ the derivatives $d(\ext(\fp)),d(l(\lambda_{+})):T\T \rightarrow \mathbb{R}$ are considered as elements in the cotangent space $T^{*}\T$. Moreover, we also have the upper bound from\red{\cite{Bridgeman2019}} that  $l_{m_{\pm}}(\lambda_{\pm})\leq 6\pi|\chi(\Sig)|$ whereas, from\red{\cite{schlenker20}} we have similar upper bounds on the extremal length $\ext_{[c_{\pm}]}(\f_{\pm})\leq3\pi|\chi(\Sig)|$, where $\chi(S)$ is the Euler characterisitic of $S$. \\ \hspace*{.5 cm}Further, there is a well-studied conjecture of Thurston which asks if the map $
\mathfrak{B}:\QF \rightarrow \mathcal{ML}(\Sig) \times \mathcal{ML}(\Sig)
$ sending a quasi-Fuchsian metric $g\in \QF$ to the data $\mathfrak{B}(g):=(\lambda_{+},\lambda_{-})$ of measured bending laminations on the boundary of its convex core, is a homeomorphism onto its image? That is to say, whether quasi-Fuchsian metrics on $M$ can be parametrized by the data of measured bending laminations $(\lambda_{+},\lambda_{-})$ on the boundary of its convex core.
Although the problem remains open in full generality (see also \cite{Bonahon2004,Lecuire2005,Series,geuritaud} and \cite{adsbending} for the anti-de Sitter case) it can be seen from rather elementary arguments that the image of the map $\beta(\QF)$ is contained in $\mathcal{FML}(\Sig)$, the space of pairs of {filling} measured geodesic laminations on $\Sig$, i.e,  $\lambda_{+}$ and $\lambda_{-}$ always fill $S$ for any quasi-Fuchsian manifold. Using this property Bonahon proves the following theorem  to which we claim our Theorem \red{\ref{thm1.1}} is an analogue of when restricted to the case of measured foliations which are arational: 
{\theorem \red{\cite{bonahon05}}\label{bonahon}
	There exists an open neighbourhood $V$ of $\mathcal{F}(S)$ in $\mathcal{QF}(S)$, such that $\mathfrak{B}:\QF\rightarrow \ML\times\ML$ is a homeomorphism between $V\setminus \mathcal{F}(S)$ and its image.  Moreover, $V$ can  be chosen so that, $\mathfrak{B}(V \setminus \F)=U$ is an open subset of $\mathcal{FML}(S)$ which intersects
	each ray $(0,\infty)(\lambda_{+},\lambda_{-})$ in an interval $(0,\epsilon_{\lambda_{\pm}})(\lambda_{+},\lambda_{-})$.}\\\\
\hspace*{.5 cm} A consequence of the theorem above is that the image $\mathfrak{B}(U\setminus\F)$ are pairs of filling measured geodesic laminations $(t\lambda_{+},t\lambda_{-})$, for $t>0$ small enough and clearly, this inspires Theorem \ref{thm1.1}. measured foliations at infinity of $M$ can be thus thought of as a new invariant that provide coordinates for $\QF$ near the Fuchsian locus in a fashion similar to that of measured bending lamination on the boundary of the convex core $\mathcal{CC}(M)$ and we summarise the preceding discussion as Table \ref{tb:1}. We conjecture that our current result can be extended to any pair $(t\fp,t\fm)\in \FMF$ for $t$ small enough.
\begin{table}[H]\label{table}
	
	\begin{tabular}{|c|c|}
		\hline
		
		On the convex core & On the boundary at infinity \\
		\hline\hline
		Thurston's conjecture on $(m_{+},m_{-})$ & Bers' Simultaneous Uniformisation Theorem \\ \hline 
		Hyperbolic length $l_{m_{\pm}}(\lambda_{\pm})$ & Extremal length  $\ext_{[c_{\pm}]}(\f_{\pm})$
		\\\hline
		$l_{m_{\pm}}(\lambda_{\pm})\leq 6\pi|\chi(S)|$ &   $\ext_{[c_{\pm}]}(\f_{\pm})\leq 3\pi|\chi(S)|$\\ \hline 
		
		Variational formula (\ref{dualvolume}) for $V^{*}_{C}$ & Variational formula (\ref{renormazlised volume}) for $V_{R}$\\  \hline
		
		Theorem \ref{bonahon} & Theorem \red{\ref{thm1.1}} \\
		\hline
	\end{tabular}
	\caption{}
	\label{tb:1}
\end{table} 

\subsection{Outline}\label{2}
We prove Theorem \ref{thm1.1} by showing the existence of unique paths in $\QF$ starting from the Fuchsian locus whose measured foliations at infinity are given by $(t\fp,t\fm)\in \mathcal{FMF}_{0}(S)$ for $t>0$ is small enough. Following\cite{bonahon05} this is done essentially by applying an inverse function theorem to the function $\mathfrak{F}: \QF \rightarrow \MF \times \MF$ at the Fuchsian locus and to remove the non-degeneracy of $\mathfrak{F}$ at $\F$, we pass to the blow-up space $\widetilde{\QF}$. To methodize, in \red{\S \ref{minsec}} we establish a necessary condition that infinitesimal deformations of quasi-Fuchsian metrics starting from the Fuchsian locus should satisfy if they have any pair of filling measured foliations $(t\fp,t\fm)$ appearing as their foliation at infinity at first order at $\F$ (Proposition \ref{neccesary}). In \red{\S \ref{pathexists}} we then use this condition to construct small paths $g_{t}$ of quasi-Fuchsian metrics starting from the Fuchsian locus which satisfies $\mathfrak{F}(g_{t})=(t\fp,t\fm)\in \mathcal{FMF}_{0}(S)$ for $0<t<\epsilon_{\fpm}$ where $\epsilon_{\fpm}$ depends on $(\fp,\fm)$ (we don't know how the $\epsilon_{\fpm}$ depends on the pair though). For this we study the sections $q^{\f},q^{-\g}:\T \rightarrow T^{*}\T$ for an arational filling pair $(\f,\g)$. An important step in the proof is to identify the intersection of $[q^{\f}]$ and $[q^{-\g}]$ in the quotient unit bundle $UT^{*}\T$ with a Teichm{\"u}ller geodesic given by the critical point of the function $\ext(t \f)+\ext(\g):\T \rightarrow \mathbb{R}$; this is done in  \red{\S \ref{horifoli}}. In \red{\S\ref{halfpipe}} we define the notion of half-pipe Schwarzians (see \red{\S \S \ref{hpschwarzandfoli}}) and use the results in \S \ref{minsec} once more by to prove Theorem \red{\ref{thm1.3}}. Chapter \ref{prelim} contains the necessary preliminaries. \\

\section{CMC surfaces}
 In this Chapter (written with Filippo Mazzoli and Andrea Seppi) we continue the analytic study of quasi-Fuchsian manifolds, and in particular of \emph{foliations} whose leaves are surfaces of \emph{constant mean curvature} (CMC), as in the following definition:

{\definition \label{defi foliation}
	A Riemannian three-manifold  $M$ homeomorphic to $\Sigma\times\R$ is (smoothly) monotonically foliated by CMC surfaces with mean curvature ranging in the interval $(a,b)$ if there exists a diffeomorphism between $\Sigma\times(a,b)$ and $M$ which, for every $H\in (a,b)$, is an embedding of constant mean curvature $H$ when restricted to $\Sigma\times\{H\}$.}

\vspace{0.1cm}

 It is known that there exist quasi-Fuchsian manifolds containing several closed minimal surfaces homotopic to $\Sigma\times\{*\}$, see \cite{zbMATH03876145} and \cite{zbMATH06460565}. In particular, this implies that there exists quasi-Fuchsian manifolds $M$ that do not admit a global monotone CMC foliation. Indeed if $M\cong \Sigma\times\R$ admits a monotone CMC foliation (as in Definition \ref{defi foliation}), then by a simple application of the geometric maximum principle, the closed embedded minimal surface in $M$ homotopic to $\Sigma\times\{*\}$ would be unique. 

Concerning uniqueness of minimal surfaces, the work of Uhlenbeck \cite{Uhlenbeck1984} highlighted the importance of a class of quasi-Fuchsian manifolds, which has been later called \emph{almost-Fuchsian} in \cite{Krasnov2007}, defined by the existence of a closed minimal surface with principal curvatures in $(-1,1)$. This condition actually implies that the minimal surface is unique, and that the equidistant surfaces from the minimal surface provide a global foliation of $M$. However, the leaves of this equidistant foliation do not have constant mean curvature, except in the trivial case where $M$ is {Fuchsian}.

Thurston conjectured that every almost-Fuchsian manifold is foliated by CMC surfaces. However, to the best of our knowledge, Fuchsian manifolds are so far the only known examples of quasi-Fuchsian manifolds that are (monotonically) foliated by CMC surfaces.

Before stating our result, let us turn our attention to some positive results in this direction. By a special case of the results of Mazzeo and Pacard in \cite{MP}, each end of any quasi-Fuchsian manifold (namely, each connected component of the complement of a compact set homeomorphic to $\Sigma\times I$ for $I$ a closed interval) is smoothly monotonically foliated by CMC surfaces, with mean curvature ranging in $(-1,-1+\epsilon)$ and  $(1-\epsilon,1)$. This result has been reproved by Quinn in \cite{quinn}, using an alternative approach which is extremely relevant for the present work. Moreover, the recent work of Guaraco-Lima-Pallete \cite{GLP} showed that every quasi-Fuchsian manifold admits a global foliation in which every leaf has constant sign of the mean curvature, meaning that it is either minimal or the mean curvature is nowhere vanishing on the entire leaf.

We also remark that existence results for CMC surfaces in the hyperbolic three-space with a given boundary curve at infinity, and in quasi-Fuchsian manifolds, have been obtained in \cite{zbMATH06553438,zbMATH06759193,zbMATH06993267}. So the main result of this part of the thesis is:

{\theorem \label{thm:foliation}
	Let $\Sigma$ be a closed oriented surface of genus $\geq 2$. Then there exists a neighbourhood $U$ of the Fuchsian locus in quasi-Fuchsian space $\QF(\Sigma)$ such that every quasi-Fuchsian manifold in $U$ is smoothly monotonically foliated by CMC surfaces, with mean curvature ranging in $(-1,1)$.}\\

\subsection{Method and outline}

The monotone CMC foliation of a quasi-Fuchsian manifold $M\cong\Sigma\times\R$, when it exists, is automatically unique by a standard application of the geometric maximum principle. More precisely, the leaf of the foliation with mean curvature $H$ is the unique closed surface homotopic to $\Sigma\times\{*\}$ in $M$ having mean curvature identically equal to $H$.

Observe that, if a quasi-Fuchsian manifold admits a monotone CMC foliation, then the mean curvature necessarily ranges in $(-1,1)$. Indeed, any leaf of the foliation must necessarily have mean curvature in $(-1,1)$, see \cite[Lemma 2.2]{zbMATH05046844}. Moreover, by the aforementioned result of Mazzeo-Pacard, the mean curvature converges to $-1$ and $1$ as the foliations approaches the ends.

We remark that the methods of our proof, which we outline below, also provide a direct proof of the \emph{existence} of closed embedded CMC surfaces of  
mean curvature $H\in (-1,1)$ in the quasi-Fuchsian manifolds $M$ within the neighbourhood $U$. (See Theorem \ref{thm:existence}.) Our proof is independent of previous result in the literature, and does not rely on geometric measure theory techniques.

The main idea of the proof of Theorem \ref{thm:foliation}  is to combine the foliations of the ends, which have been provided in the works of Mazzeo-Pacard and Quinn for every quasi-Fuchsian manifold, with foliations of the compact part that we obtain by a ``deformation'' from Fuchsian manifolds. .The main steps are 
\begin{itemize}
	\item For the foliations of the ends, we adapt the proof given by Quinn in \cite{quinn}, which relies on the Epstein map construction (\cite{epstein,Dumas2017}), that associates to a conformal metric defined  in (a subset of) the boundary at infinity of $\Hyp^3$ an immersed surface in $\Hyp^3$ by ``envelope of horospheres''. One can then translate the condition of constant mean curvature into a PDE on the conformal factor, to which we apply an implicit function theorem method in an infinite-dimensional setting. The fact that the obtained solutions provide a smooth monotone foliation of the complement of a large compact set in the quasi-Fuchsian manifold $M$ follows from another application of the implicit function theorem. The main difference with respect to Quinn's proof is that we refine his method in order to achieve the existence of monotone foliations by CMC surfaces of mean curvature $(-1,-1+\epsilon)\cup (1-\epsilon,1)$ for \emph{any} quasi-Fuchsian manifold in a neighbourhood $U_M$ of a given  $M\in\QF(\Sigma)$, where the constant $\epsilon$ is \emph{uniform} over $U_M$ (Theorem \ref{thm:foliation_ends}).
	
\item	For the compact part, we again obtain the existence of CMC surfaces, for $H\in (-1,1)$, with an implicit function theorem method in infinite-dimensional spaces, using the Epstein construction. In this case, however, the initial solution to which we apply the implicit function theorem is not ``at infinity''; it is instead the umbilical CMC surface in a Fuchsian manifold. In other words, we ``deform'' CMC surfaces in a Fuchsian manifold $M'$ to nearby quasi-Fuchsian manifolds in a neighbourhood $U_{M'}$. Similarly as above, the main technical difficulty is to have a uniform control of the constants, which must not depend on the quasi-Fuchsian manifold as long as we remain in the neighbourhood $U_{M'}$. See Theorem \ref{thm:existence_compact}.
	
\item 	The proof of Theorem \ref{thm:foliation} is then concluded by showing that these surfaces patch together to  a global smooth monotone foliation (Section \ref{sec:finish}), by means of a combination of a careful analysis of the constructed open sets in $\QF(\Sigma)$ and of several geometric arguments, for instance applications of the geometric maximum principle, relying on the observation that the CMC surfaces obtained as deformations from the Fuchsian locus can be assumed, up to restricting to smaller neighbourhoods, to have principal curvatures in $(-1,1)$.
	
\end{itemize}
\chapter{Preliminaries}\label{prelim}
 \section{Hyperbolic space} \hfill \break 
For describing the upper half space model of the $n$-dimensional hyperbolic space we let\\ $\left\lbrace \mathbb{H}^{n}:= {(x_{1},x_{2},\dots,x_{n})\in \mathbb{R}^{n}}| x_{n}>0\right\rbrace $ to be the upper-half plane in $\mathbb{R}^{n}$ with the metric\\ $\frac{dx_{1}^{2}+dx_{2}^{2}+\dots + dx_{n}^{2}}{x_{n}^{2}}$. For all the thesis we will restrict to the case when $n=2,3$.  \\
The boundary at infinity $\partial_{\infty}\mathbb{H}^{n}$ in this model identified with $\mathbb{R}^{n-1}\cup \infty$ which is nothing but the $(n-1)$-sphere $S^{n-1}$.\\ 
We call $\Isom^{+}(\mathbb{H}^{n})$ as the group of orientation preserving isometries of $\mathbb{H}^{n}$ and in the cases $n=2,3$ it is identified witht the group $PSL_{2}(\mathbb{R})$ and $PSL_{2}(\mathbb{C})$ respectively. The action extends uniquely to $\partial_{\infty}\mathbb{H}^{n}$ via conformal diffeomorphisms of $S^{n-1}$ for $n=2,3$.
Isometries of $\mathbb{H}^{n}$ are self-maps which preserve the metric. \\
We will also focus on the fact that in lower dimensions there is an interesting interplay between hyperbolic geometry and complex analysis in $1$-dimension. In dimension $2$ this can be seen by writing down the upper half plane model as $H=\left\lbrace z\in \mathbb{C}| z= x + \textbf{\textit{i}}y , y>0 \right\rbrace $ with the metric given by $\frac{|dz|^{2}}{y^{2}}$. Here again we see that $\Isom^{+}(\mathbb{H}^{2})$ is identified with the group of biholomorphism of $H$.\\
For dimension $3$ we see that every isometry of $\mathbb{H}^{3}$ is uniquely determined by its conformal action of the sphere at infinity $\partial_{\infty}\mathbb{H}^{3}$, the latter being the complex projective space $\mathbb{C}P^{1}$. This is one way of seeing that $\Isom^{+}(\mathbb{H}^{3})$ is identified with the group of M{\"o}bius transformations of the disc, which is $PSL_{2}(\mathbb{C})$. 

\section{Hyperbolic surfaces and $3$-manifolds} \hfill \break 
We will consider $S$ to be a closed surface of genus $g\geq 2$. We call $S$ to be a hyperbolic surface if we have an atlas $(U_{i},\phi_{i})$ on $S$ where $\phi_{i}: U_{i}\rightarrow \mathbb{H}^{2}$ are charts such that at each intersection $U_{i}\cap U_{j}$, the composition $\phi_{i} \circ \phi^{-}_{j}$ are locally restrictions of elements of $PSL_{2}(\mathbb{R})$. An alternative defition can be to say that $S$ is a closed hyperbolic surface if it carries a complete Riemannian metric of constant sectional curvature $-1$. It follows from Gauss-Bonnet theorem that $S$ can carry such a metric only when $g>1$, thus leading to our assumptions. In such a case, one can also state that $S$ is isometric to the quotient of $\mathbb{H}^{2}$ by $\Gamma$ where $\Gamma$ is a discrete subgroup of $PSL_{2}(\mathbb{R})$.\\

On the other hand, a complex structure $c$ on $\Sig$ consists of an atlas $\left\lbrace U_{\alpha},\phi_{\alpha}\right\rbrace $ on $\Sig$ where $\phi_{\alpha}:U_{\alpha}\rightarrow \mathbb{C}$ are holomorphic maps and the transition functions $\phi_{i}\circ \phi^{-1}_{j}$ are biholomorphic maps on $\phi_{i}(U_{i}\cap U_{j})$. Given a complex structure $c$, we consider its equivalence class under diffeomorphisms of $S$ isotopic to the identity and denote it as $[c]$. \\
Consider now a Riemannian metric $g$ on $S$, we can define:
{\definition A conformal class on a surface $\Sig$ is an equivalence class of Riemannian metrics $[g]$, where \begin{align*}
		[[g]]=\left\lbrace e^{2u}g| u: S \rightarrow \mathbb{R}\right\rbrace 
	\end{align*} and $u$ is smooth.} \\ \\
When $S$ is oriented there is a one-to-one correspondence between equivalence classes of complex structures on $S$ under diffeomorphisms isotopic to the identity and conformal classes on $S$  again up to diffeomorphisms isotopic to the identity. Owing to this we can also view Teichm{\"u}ller space as the space of conformal classes on $S$ (see \cite{bookoftromba}). We will denote both a conformal and complex structure on $S$ as $c$. A Riemannian metric on $S$ in a conformal class has the local expression $g(z)=\rho(z)dzd\bar{z}$ where $\rho(z)\geq 0$ is a smooth function on $\Sig\rightarrow \mathbb{R}_{> 0}$. 
\\\hspace*{.5 cm}We will recall now an important lemma concerning the change of {Gaussian} or intrinsic curvature $K_{g}$, associated to a Riemannian metric $g$ under change of conformal factor in the same conformal class. See\red{\cite{Krasnov2007}} among others for a reference:
{\lemma\label{formula} Let $g$ and $g'$ be two Riemannian metrics on $\Sig$ in the same conformal class and let  $u: \Sig \rightarrow \mathbb{R}$ be a function such that $g'=e^{2u}g$. Let $K_{g}$ and $K_{g'}$ be the Gauss curvatures associated to $g$ and $g'$ respectively. Then $K_{g'}=e^{-2u}(-\Delta_{g} u+ K_{g})$, where $\Delta_{g}u$ is the Laplace-Beltrami operator for the metric $g$.}\\\\
Here we use the convention that $\Delta_{g}$ is negative of the usual analysts Laplacian. If we consider $g$ to be a conformal metric, then the hyperbolic metric $m$ in the conformal class of $g$ is given by $m=e^{2u}g$ where $u$ solves:
{\begin{align*}
		-1=e^{-2u}(-\Delta_{g} u+K_{g})
\end{align*}}
So in dimension $2$, corresponding to every conformal class on $S$, one has a unique hyperbolic metric and also an equivalence class of complex structures. \\ 
The definitions for hyperbolic structures on surfaces extend to that for $3$-manifolds where we will call a $3$-manifold to be hyperbolic if it carries a complete Riemannian metric with sectional curvature being $-1$. Alternatively, we can identify it with the quotient $\mathbb{H}^{3}$ with a discrete subgroup of $PSL_{2}(\mathbb{C})$. We will note that such subgroups of $PSL_{2}(\mathbb{R})$ and $PSL_{2}(\mathbb{C})$ are called Fuchsian and Kleinian respectively. 

\section{Fuchsian, almost-Fuchsian and quasi-Fuchsian $3$-manifolds} \label{introtofuchsian}\hfill \break 
We will consider $3$ manifolds $M$ here which are quotient $\mathbb{H}^{3}/\Gamma$ where $\Gamma$ is a discrete subgroup of $PSL_{2}(\mathbb{C})$. Associated to the action of $\Gamma\hookrightarrow\mathbb{H}^{3}$ one has the limit set $\Lambda_{\Gamma}$ which is the set of accumulation points of orbit of $\Gamma$, and it can so shown that it is a subset of $\partial_{\infty}\mathbb{H}^{3}$. When $\Gamma$ is a discrete subgroup of $PSL_{2}(\mathbb{R})\subset PSL_{2}(\mathbb{C})$, we call $M$ to be a Fuchsian manifold. In this case $\Lambda_{\Gamma}$ is a circle on $\partial_{\infty}\mathbb{H}^{3}$. This can be seen as the boundary of the totally geodesic copy of $\mathbb{H}^{2}\subset \mathbb{H}^{3}$ preserved by the action. \\

We call a $3$ manifold $M$ to be quasi-Fuchsian if the $\Gamma < PSL_{2}(\mathbb{C})$ is such that $\Lambda_{\Gamma}$ is a quasi-circle. To define concretely, one has a quasi-conformal map $\phi: \mathbb{C}P^{1}\rightarrow \mathbb{C}P^{1}$ and a Fuchsian subgroup $\Gamma_{0}$ such that $\Gamma := \phi^{-1}\circ \Gamma_{0} \circ \phi$. To simplify, we can consider $\Gamma$ to be quasi-Fuchsian if $\Lambda_{\Gamma}$ is a Jordan curve on $\partial_{\infty}\mathbb{H}^{3}$. \\ 
One more subspace of quasi-Fuchsian space we want to consider is that of almost-Fuchsian. This is defined as quasi-Fuchsian manifolds containing a closed minimal surface with principal curvatures in $(-1,1)$. 
We will also recall here that a surface is minimal if its mean curvature is $0$. \\
One more charateristic difference between Fuchsian and quasi-Fuchsian manifolds is by the different geometry of their convex hulls. Given the limit set $\Lambda_{\Gamma}$ as the setting above, we can consider its convex hull in $\mathbb{H}^{3}$. In the Fuchsian case we recover the totally geodesic copy of $\mathbb{H}^{2}$ preserved by the action as $\Lambda_{\Gamma}$ is a circle. In the quasi-Fuchsian case we have that the convex hull is a closed, convex region which upon taking quotient becomes the convex core $\mathcal{CC}(M)$. The convex core is thus the smallest-non empty convex submanifold contained in the quasi-Fuchsian manifold and it comes with two boundary components since it is $S\times [-1,1]$. \\
Moreover, we have the boundary at infinity of Fuchsian or almost or quasi-Fuchsian manifolds. The action of $\Gamma$ is free and properly discontinuous in the complement $(\partial_{\infty}\mathbb{H}^{3}\setminus \Lambda_{\Gamma})$, also called the domain of discontinuity, which has two connected components. The boundary at infinity of $M$, denoted as $\partial^{+}_{\infty}M, \partial^{-}_{\infty}M$, are the respective quotients of components of $\partial_{\infty}\mathbb{H}^{3}\setminus \Lambda_{\Gamma}$, each being homeomorphic to $S$.
\section{Teichm{\"u}ller space} \label{teichspace}\hfill \break 
We will now briefly introduce the Teichm{\"u}ller space $\T$ associated to a closed surface of genus $g\geq 2$.
{\definition The Teichm{\"u}ller space $\T$ is the space equivalence classes of complex structures on $S$ under diffeomorphisms isotopic to the identity.}\\\\
Owing to the presence of an unique hyperbolic metric in every conformal class or class of complex structure we can alternately define $\T$ as the space of equivalence classes of hyperbolic metrics on $S$ up to diffeomorphisms isotopic to the identity.\\
Now recall that the boundary at infinity $\partial^{\pm}_{\infty}M$ are two copies of $S$ which carry their respective complex structures. It so happens that when $M$ is Fuchian, the complex structures/hyperbolic metrics are identical and thus if one defines $\F$ as the equivalence class of Fuchsian metrics on $M$ up to diffeomorphism isotopic to the identity, then $\F\cong \T$. \\
Likewise, call $\QF$ to be the equivalence class of quasi-Fuchsian metrics on $M$ up to diffeomorphism isotopic to the identity. By virtue of Bers' simultaneous Uniformization theorem one has that $\QF\cong \mathcal{T}(\partial^{+}_{\infty}M)\times \mathcal{T}(\partial^{-}_{\infty}M)$. \\
Now, given two conformal structures $c,c'$ we can define the notion of quasi-conformal map between them.
{\definition Let $\Omega,\Omega'\subset\mathbb{C}$ be two domains and $\phi: \Omega\rightarrow \Omega'$ be a homeomorphism with continuous partial derivatives with respect to $z$ and $\bar{z}$. We denote $\phi$ to be $k$-quasiconformal if 
	\begin{align*}
		k_{\phi}(z)=\frac{|\frac{\partial \phi}{\partial \bar{z}}(z)|+|\frac{\partial \phi}{\partial z}(z)|}{|\frac{\partial \phi}{\partial \bar{z}}(z)|-|\frac{\partial \phi}{\partial z}(z)|}\leq k
	\end{align*}
	for almost every $z\in \Omega$.}\\\\ 
We will denote the $\frac{\partial \phi}{\partial w}$ as $f_{w}$. The {Beltrami differential} $\mu$ associated to the map $\phi$ is defined as the ratio 
\begin{align*}
	\mu_{\phi} = \frac{\phi_{\bar{z}}}{\phi_{z}}
\end{align*}
which is defined almost everywhere, is measurable and satisfies $||\mu||_{\infty}<1$. Equivalently $k_{\phi}$ has the expression $\frac{1+|\mu|}{1-|\mu|}$, which is bounded above by $k=\frac{1+||\mu||_{\infty}}{1-||\mu||_{\infty}} $ and $k_{\phi}$ is called the eccentricity coefficient of $\phi$. We also note that a Beltrami differential on $(S,[c])$ is a tensor of type $(-1,1)$. We will use this to define the notion of holomorphic quadratic differentials on $S$ and how it relates to the tangent and cotangent space of $\T$.
\section{Holomorphic quadratic differentials}\label{HQDintro} \hfill \break 
A {holomorphic quadratic differential} $q$ on $(S, c)$ is a tensor of type $(2,0)$ which in local coordinates can be written as $f(z)dz^{2}$, where $f$ is a holomorphic function. The space of holomorphic quadratic differential denoted as $Q(S)$ forms a bundle over $\mathcal{T}(S)$, where $\T$ is seen as the space of complex structures on $\Sig$ up to  diffeomorphisms isotopic to the identity. The fiber over an equivalence class $[c] \in \T$, which is denoted as $Q(S, [c])$ is a vector space which can be shown to have real dimension $6g-6$ (by, for example the {Riemann-Roch formula}). Moreover, a holomorphic quadratic differential $q$ has zeroes on $\Sig$ the degree of which is defined in terms of the degree of the zero of the local Taylor expansion of $q$. To be precise, if $q$ has a zero of order $k$ at a point $p\in S$ then this means that for all chart centred at $p$ on $S$, $q$ has the local expression $f(z)z^{k}dz^{2}$ for some holomorphic function $f(z)$ such that $f(0)\neq 0$.. Moreover, it follows from, for example the Riemann-Roch theorem, that the sum of the degrees of all zeroes of $q$ on $\Sig$ is $4g-4$.  The space $Q(S)$ further carries a natural stratification depending on the order of the zeroes of $q$. Please consult, for example \cite{Hubbard}, \cite{fredbook} for references on this topic. 
{\definition Let $\overline{k}$ be a $n$-tuple of integers $(k_{1},k_{2},\dots,k_{n})$ such that $\sum_{i}^{n}k_{i}=4g-4$. The stratum $Q^{\overline{k}}(S)$ is the set of holomorphic quadratic differentials $q$ such that the degrees of the zeroes of $q$ are given by the $k_{i}$. We say that $q$ is generic, if $k_{i}=1$ for all $i$.} \\\\
Holomorphic quadratic differentials with only $4g-4$ simple zeroes are termed as generic quadratic differentials and it is known that they form a dense open subset of $Q(S)$ which will be denoted as $Q_{0}(S)$. (see \red{\cite{Douady1975}}).\\\\
Notice that the product of a Beltrami differential and a holomorphic quadratic differential gives us a $(1,1)$ tensor. Let $B(S,c)$ denote the vector space of measurable Beltrami differentials on $(S,c)$ where an element is expressed locally as $\mu = b(z)\frac{d\bar z}{d z}$ . From here we have a natural complex pairing between $\mu \in B(S,c)$ and $q \in Q(S,c)$ as: 
\begin{align*}
	\left\langle q,\mu \right\rangle = \int_{(S,c)} q\mu dz d\bar{z} 
\end{align*}
It follows as a consequence, see \cite{Hubbard}, that: 
{\proposition \label{keylemma}There is an isomorphism of vector spaces between 
	\begin{align*}
		T_{[c]}\T \cong B(S,c)/Q(S,c)^{\perp} {\quad and \quad} T^{*}_{[c]}\T \cong Q(S,c)
	\end{align*}
	where $Q(S,c)^{\perp}= \left\lbrace \mu \in B(S,c)|\left\langle \mu,q\right\rangle  =0, \forall q \in Q(S,c)\right\rbrace$.}\\\\
This allows us to define the Weil-Petersson metric on $T^{*}_{[c]}\T$ as: 
\begin{align*}
	\left\langle q_{1},q_{2}\right\rangle_{WP} = \int_{S} \frac{f_{1}(z)\overline{f_{2}(z)}}{\rho(z)}dz \overline{dz}
\end{align*}
where the hyperbolic metric in the class $[c]$ has the expression $\rho(z)|dz|^{2}$ and $q_{i}=f_{i}(z)dz^{2}$ for $i=1,2$ are two holomorphic quadratic differentials in $Q(S,c)$. This also induces an inner product on the tangent space $T_{[c]}\T$ by duality. The Weil-Petersson metric gives $\T$ with the structure of negatively curved Riemannian manifold. On the other hand, the $L^{1}$-norm on the cotangent space $T^{*}_{[c]}\T$ is defined as:
\begin{align*}
	||q||_{1} = \int_{(S,c)}|f(z)| dz \wedge \overline{dz}.
\end{align*}
This induces a Teichm{\"u}ller norm on $T\T$ via the duality between Beltrami differentials and holomorphic quadratic differentials. One way to express the associated metric, called the Teichm{\"u}ller metric $d_{\T}$, is:
\begin{align}\label{teichmullergeo}
	d_{\T}([c],[c']):=\frac{1}{2}\inf\left\lbrace\log k_{\phi}| \text{ }\phi:(\Sig,[c])\rightarrow (\Sig,[c'])\text{ quasiconformal isotopic to the identity} \right\rbrace 
\end{align} One can consult, for example \cite{carlos}, for further details in this topic.

		\section{Measured foliations on $S$ and the space $\MF$}\label{defmf}\hfill\break 
	Following\red{\cite{hubbardmasur}} we define:
	{\definition\label{deffoliations} A smooth measured foliation $\f$ on $\Sig$ with singularities $\left\lbrace p_{1},\dots,p_{n}\right\rbrace $ of
		order $\left\lbrace k_{1},\dots,k_{n}\right\rbrace $ (respectively) is given by an open covering ${U_{i}}$ of $\Sig\setminus\left\lbrace p_{1},\dots,p_{n}\right\rbrace $ and
		open sets $\left\lbrace V_{1},\dots,V_{n}\right\rbrace $ around $\left\lbrace p_{1},\dots,p_{n}\right\rbrace $ (respectively) along with smooth non-vanishing real valued $1$-forms $d\phi_{i}$
		defined on $U_{i}$ such that:
		\begin{itemize}
			\item $d\phi_{i}=\pm d\phi_{j}$ on $U_{i}\cap U_{j}$
			\item around each $p_{l}$ there is an open neighbourhood $V_{l}$ and a chart $(x_{1},x_{2}):V_{l}\rightarrow \mathbb{R}^{2}$ such that $d\phi_{i} = \mathfrak{I}(z^{\frac{k_{l}}{2}+1} dz)$ on $U_{i}\cap V_{j}$ where $z=x_{1}+\textbf{i}x_{2}$.
	\end{itemize}}

	Immersed lines on $\Sig$ along which $d\phi_{i}$ vanish give a
	foliation $\f$ on $\Sig \setminus (p_{1},\dots,p_{n})$ and we have a $(k_{j}+2)$ pronged singularity at $p_{j}$. Given an arc $\gamma$ on $\Sig$ which avoids the zeroes $(p_{1},p_{2},\dots,p_{n})$, a measured foliation $\f$ associates a transverse measure to $\gamma$ defined as $\mu_{\f}(\gamma)= \int_{\gamma}|d\phi|$, where $|d\phi|$ restricted on each $U_{i}$ is given by $|d\phi_{i}|$. \\\hspace*{.5 cm}
	This measure is invariant under isotopies that maintain the same end points of $\gamma$ and the transversality of the intersection of $\gamma$ with the given foliation. That is, if $\gamma'$ is isotopic to $\gamma$ with the same end points and maintaining the transversality at every time, then $\mu_{\f}(\gamma)=\mu_{\f}(\gamma')$. So, given a class $[\gamma]$, we define:
	{\definition {The intersection number} of $\f$ with a isotopy class of closed curves $[\gamma]$ avoiding the singularities of $\f$ is defined $i([\gamma],\f)=\inf_{\gamma \in [\gamma]}i(\gamma,\f)=\inf_{\gamma \in [\gamma]}\mu_{\f}(\gamma)$, where the infimum is taken over all $\gamma \in [\gamma]$. }\\\\
	As we can see that the intersection number defines a function from the set of closed curves up to isotopy on $\Sig$ to $\mathbb{R}_{>0}$ and we define following, for example\red{\cite{AST_1979__66-67_}}:
	{\definition 
		Two measured foliations $\f$ and $\g$ on $\Sig$ are said to be equivalent if they define the same intersection number. The space of equivalence classes of measured foliations on $\Sig$ will be denoted as $\MF$. 
	}\\
	
	The space $\MF$ can also be defined through a topological equivalence of two foliations which comes via {Whitehead moves} but we do not elaborate on that. Also it follows from Proposition {$2.1$} in\red{\cite{hubbardmasur}} that a measured foliation $\f$ can have $4g-4$ pronged-singularities counted up to multiplicity. \\\hspace*{.5 cm}
	 There is also an action of $\mathbb{R}_{>0}$ on $\MF$ defined as $t.\f\mapsto t\f$ where the latter denotes the measured foliation obtained by multiplying $t$ by the $1$-forms $d\phi_{i}$ which give us the measured foliation $\f$ as par Definition \ref{deffoliations}. 
	The space $(\MF\setminus 0)/\mathbb{R}_{>0}$ is called the space of projectivised measured foliations, denoted as $P\MF$ and it is identified with the {Thurston boundary} of $\T$. We will denote by $[\f]$ as the equivalence class of $\f$ in $P\MF$. Interested readers can consult, for example\red{\cite{AST_1979__66-67_}}, for further details in this topic.  \\
	We also introduce the notion of arational measured foliations here as: 
	{\definition \label{arationaldef} A measured foliation is said to be arational if all its singularities have $3$ prongs and if there are no leaves of the foliation joining the singularities.}\\ 
	
	We call leaves of the foliations joining prongs of the singularities as saddle connections.

\section{Minimal and CMC surfaces in hyperbolic $3$ manifolds}\label{minimalsurface}\hfill\break
Let $i: \Sig \rightarrow M$ be an immersion of $\Sig$ into $M$. Associated to this, we have the data of the {first fundamental form } denoted as $I$ on $\Sig$ which is the induced metric on $\Sig$ inherited from the ambient space and the {second fundamental form} denoted a $\II$ which is a symmetric bilinear form on the tangent bundle $T\Sig$. Associated to the couple $(I,\II)$ we have a unique self-adjoint operator $B$, called the {shape operator}, which satisfies the relation \begin{align*}
	\II(x,y)=I(Bx,y)=I(x,By)
\end{align*} where $x,y \in T_{p}S, \forall p \in S$. We can also define another quantity associated to the immersion called the third fundamental form $\III$ defined as $\III(x,y)=I(Bx,By)=I(B^{2}x,y)$ for $x,y,B$ as before. See\cite{Perdigao} for reference.\\\hspace*{.5 cm}
The eigenvalues of $B$ gives us the principal curvatures associated to the immersion. Since minimal immersions are those immersions for which the mean curvature of $\Sig$ is zero we can define as well that: 
{\definition An immersion is minimal if and only if $B$ is traceless. }
\\

On the other hand, assume we are given a smooth Riemannian metric $g$ on $\Sig$ and a symmetric bilinear form $h$ on $T\Sig$. The couple $(g,h)$ will be associated to the data of an immersion  of $\Sig$ if it satisfies the following:

\begin{enumerate}
	
	\item The Codazzi equation, $d^{\nabla}h=0$, where $\nabla$ is the Levi-Civita connection of $g$.
	\item The Gauss equation, $K_{I}=-1+\deter_{g}(h)$, where $K_{g}$ denotes the Gaussian or intrinsic curvature of the metric $g$.
\end{enumerate}
Now we define the notion of an {almost}-Fuchsian manifold as: 

{\definition A quasi-Fuchsian hyperbolic $3$-manifold $M \cong \Sig \times \mathbb{R}$ is called {almost-Fuchsian} if it contains a closed minimal surface homeomorphic to $\Sig$ with principal curvatures in $(-1,1)$.}\\\\
Given an immersed surface in $M$, an outcome of the Codazzi equation is that, the traceless part of the second fundamental form $(\II)_{0}$ is equal to the real part of a holomorphic quadratic differential $q$. 
Now, following Corollary $2.9$ of\red{\cite{Krasnov2007}} we can show that if $M$ is an almost-fuchsian manifold then the closed minimal surface it contains is unique. This allows us to parametrise almost-Fuchsian hyperbolic manifolds by an open subset $\Omega$ in $T^{*}\T$, following \cite{Krasnov2007}, Theorem $2.12$:
{\theorem\label{methodandoutline} There exists an open subset $\Omega \subset T^{*}\T$ such that we have the following bijection:
	\begin{enumerate}
		\item Given a point $([c],q)\in \Omega$, there exists a unique almost-Fuchsian metric on $M$ such that the unique minimal surface has first fundamental form conformal to $[c]$ and the second fundamental form $\II$ is $\mathfrak{R}(q)$.
		\item Given a almost-Fuchsian metric $g \in \AF$ on $M$, the induced metric and  second fundamental form of its unique minimal surface are specified by a point in $\Omega$. 
\end{enumerate}} 
We will now recall another definition which is of a CMC $H$-surface.
\begin{definition}
	An immersed surface $S$ in $M$ is called a CMC $H$-surface if $\trace(B)=H$ for all points on $S$.
\end{definition}
We will recall now a tool that we use in all the chapters of this thesis:
\section{Schwarzians at infinity of quasi-Fuchsian manifolds}\label{introductionary}\hfill \break 
Recall that the boundary at infinity $\partial_{\infty}\mathbb{H}^{3}$ is identified with the complex projective space $\mathbb{C}P^{1}$. The components $\partial^{+}_{\infty}M$ and $\partial^{-}_{\infty}M$ are the quotient of domains in $\mathbb{C}P^{1}$ under the action of $\Gamma < PSL_{2}(\mathbb{C})$ and so they carry canonical $\mathbb{C}P^{1}$-structure which is a $(G,X)$ structure on $S$ with $G=PSL_{2}(\mathbb{C})$ and $X=\mathbb{C}P^{1}$\red{\cite{Dumas}}. That is to say we have an open covering of $\Sig$ by an atlas $(U_{\alpha},\phi_{\alpha})$ such that $\phi_{\alpha}:U_{\alpha}\rightarrow \mathbb{C}P^{1}$ are charts to open domains in $\mathbb{C}P^{1}$ and in the overlap  $U_{i}\cap U_{j}$  of two charts the change of coordinate map $\phi_{i}\circ \phi^{-1}_{j}$ is locally a restriction of a M{\"o}bius transformations. Denote the space of equivalence classes of $\mathbb{C}P^{1}$-structures on $\Sig$ under diffeomorphisms isotopic to the identity as $\mathcal{CP}(\Sig)$.  \\\hspace*{.5 cm}Now, given a $\mathbb{C}P^{1}$-structure on $\Sig$ we have an underlying complex structure as M{\"o}bius transformations are biholomorphisms and for $\partial^{\pm}_{\infty}M$ it is precisely $[c_{\pm}]$ up to equivalence. This gives us a natural forgetful map $\mathcal{CP}(\Sig)\rightarrow \T$ mapping a $\mathbb{C}P^{1}$-structures to the underlying complex one. Now by the  Uniformisation theorem any complex structure $c$ on $\Sig$ arises as the quotient of the action of some discrete subgroup  $\Gamma_{c}$ of $PSL_{2}(\mathbb{R})$ on $\mathbb{H}^{2}$. As $\Gamma_{c}<PSL_{2}(\mathbb{C})$ as well and $\mathbb{H}^{2}$ can be seen as the unit disc $\Delta\subset \mathbb{C}P^{1}$, we have a canonical $\mathbb{C}P^{1}$-structure associated to a complex structure which we call the {standard Fuchsian complex projective structure} and this gives us a continuous section $\T \rightarrow \mathcal{CP}(\Sig)$.\\\hspace*{.5 cm}
The Schwarzian derivative yields a parametrisation of the fibers  of the forgetful map $\mathcal{CP}(\Sig)\rightarrow \T$. In general, given a domain $\Omega\subset \mathbb{C}$, 
the Schwarzian derivative of a locally injective holomorphic map $u:\Omega \rightarrow \mathbb{C}$ is a holomorphic quadratic differential defined as:
\begin{align*}
	\sigma(u)= ((\frac{u''}{u'})^{'}-\frac{1}{2}(\frac{u''}{u'})^{2})dz^{2}
\end{align*} One way to obtain the expression on the right hand side above is to consider the unique M{\"o}bius transformation
$M_{u}$ which matches with $u$ up to second order derivative. The expression above is precisely the difference of
the third order terms in the local Taylor series expansion of $u$ and $M_{u}$ (see Proposition $6.3.3$ of\cite{Hubbard}). Further, they have two remarkable properties:
\begin{itemize}
	\item For two locally injective holomorphic maps $u,v: \Omega \rightarrow\mathbb{C}$ we have $\sigma(u\circ v)= v^{*}\sigma(u)+\sigma(v)$
	\item $\sigma(A)=0$ if and only if $A$ is a M{\"o}bius transformation.
\end{itemize} In particular, there is a unique map defined up to right conjugation by M{\"o}bius transformations between a given complex projective structure on $S$ and the standard
Fuchsian one which is holomorphic with respect to the underlying canonical complex structure and by virtue of
the properties above, the Schwarzian derivative for this holomorphic map can be defined in a chart independent
way. This is called the Schwarzian parametrisation of a complex projective structure on $S$ with respect to
the standard Fuchsian complex projective structure (see \cite{Dumas},\cite{jmnotes}). When $M \in \QF$, the components of $\partial_{\infty}\mathbb{H}^{3}\setminus \Lambda_{\Gamma}$  have non-trivial Schwarzian derivatives
associated to them by construction which descend to two holomorphic quadratic differentials on $\partial^{\pm}_{\infty}M$ upon
taking quotients. So we define: {\definition  The Schwarzians at infinity $\sigma_{+}$ and $\sigma_{-}$ are the holomorphic quadratic differentials obtained on $(\partial^{+}_{\infty}M,[c_{+}])$ and $(\partial^{-}_{\infty}M,[c_{-}])$ by the Schwarzian parametrisation of the $\mathbb{C}P^{1}$ structures on $\partial^{\pm}_{\infty}M$ with respect to the corresponding standard Fuchsian complex projective structure.}\\\\
Also note that when $M\in\F$ due to the second property above the Schwarzians at infinity are zero. What is more important to us from this discussion is that due to this we get another parametrisation of quasi-Fuchsian manifolds by $Q(S)$ by considering the schwarzian derivative and complex structure appearing at one end at infinity. So we can therefore construct a well-defined map
\begin{equation}\label{eq:schwarzian map}
	\mathcal S:\QF\to\Q(S)~.
\end{equation}
Here $\Q(S)$ denotes the bundle of holomorphic quadratic differentials over $\T$, whose fiber over a point $(S,[h])$ coincides with the vector space $H^0((\Sigma,h),K^2)$, where $K$ denotes the canonical divisor of $(S,[h])$. Consequently, the space $\Q(S)$ is a complex manifold of dimension $3g-3$, where $g$ denotes the genus of $S$. In fact, the map $\mathcal S$ turns out to be injective (see also the discussion below on the construction of its inverse) and, being $\QF$ and $\Q(S)$ manifolds of the same real dimension, the invariance of domain theorem implies that its image is an open subset of $\Q(S)$.

\subsection{Constructing the inverse} We will often use the the inverse map of $\mathcal S$, defined on the image of  $\QF$. Hence it will be useful to quickly discuss its explicit construction. In general, given a holomorphic quadratic differential $q$ on a connected open set $\Omega\subset \C$, there exists a locally injective holomorphic map $f_q:\Omega\to \C$ such that $S(f_q)=q$, see \cite{zbMATH03053379} and \cite[Proposition 6.3.7]{zbMATH05042912}. By the fundamental properties discussed above, $f_q$ is unique up to post-composition with a M\"obius transformation. One can also see that $f_q$, suitably normalized, depends smoothly on $q$.\\

To apply this in our setting, we consider a hyperbolic metric $h$ on $\Sigma$ and $\phi\in H^0((\Sigma,h),K^2)$, and realize $(\Sigma,h)$ as the quotient of the Poincar\'e disc $\D$ by a discrete group $\Gamma_h$ of biholomorphisms. We can then lift $\phi$ to a $\Gamma_h$-invariant holomorphic quadratic differential $\tilde \phi$ on $\D$, and find a locally injective holomorphic map $f_{\tilde\phi}:\D\to\C$ whose Schwarzian derivative is equal to $\tilde \phi$. Since $\tilde\phi$ is invariant under the action of $\Gamma_h$, we have
$$S(f\circ \gamma)=\gamma^*S(f)=\gamma^*\tilde\phi=\tilde\phi=S(f)$$ for every $\gamma\in \Gamma_h$. We deduce that for any $\gamma \in \Gamma_h$ there exists a M\"obius tranformation $\zeta=\zeta(\gamma)$ such that $f\circ\gamma=\zeta(\gamma)\circ f$, providing us with a representation $\zeta : \Gamma_h \rightarrow \PSL(2,\C)$. This construction is exactly the inverse of the map $\mathcal S$, in the sense that if $f:\D\to\Omega^+$ is the biholomorphic map associated to a quasi-Fuchsian manifold $\Hyp^3/\Gamma$ as in Section \ref{eq:schwarzian map}, and $h$ and $\phi$ are the induced hyperbolic metric and holomorphic quadratic differential on $\Sigma$, then $f_{\tilde\phi}=f$ and the image of the representation $\zeta$ coincides with the quasi-Fuchsian group $\Gamma$.

\chapter{Measured foliation at infinity of quasi-Fuchsian manifolds}\label{ch3}
We will now focus on the first part of the thesis which is on realising measured foliations at the boundary at infinity of quasi-Fuchsian manifolds close to the Fuchsian locus. 

\section{The bundle $Q(\Sig)$ and measured foliations realised by holomorphic quadratic differentials.}\label{horifoli}\hfill\break
Now, given $q\in Q(S,c)$, away from its zeroes we can always perform a local change of coordinates $z\mapsto w:=\int\sqrt{q}$ on $S$ such that $q=f(z)dz^{2}$ has the local expression $dw^{2}$ with respect to this coordinate. If we write $w=w_{1}+\textbf{i}w_{2}$ then the holomorphic quadratic differential $dw^{2}$ canonically equips $\mathbb{C}$ with two measured foliations: 
\begin{itemize}
	\item The horizontal measured foliation, which are immersed lines given by  $w_{2}=\text{const.}$, i.e the horizontal lines of $\mathbb{C}$. Its transverse measure being given by $|\mathfrak{I}\sqrt{dw^{2}}|=|dw_{2}|$.
	\item The vertical measured foliation, which are immersed lines along which $w_{1}=\text{const.}$, i.e the vertical lines of $\mathbb{C}$.
	Its transverse measure being given by $|\mathfrak{R}\sqrt{dw^{2}}|=|dw_{1}|$.
\end{itemize} 
Moreover, notice that the horizontal measured foliations (resp. vertical measured foliations) of quadratic differential $|dw^{2}|$ gives us all the horizontal lines (resp. vertical lines) on $\mathbb{C}$, thus inspiring the nomenclature. So we define: 
{\definition The horizontal measured foliation  $\mathrm{hor}_{[c]}(q)$ (resp. vertical measured foliation $\mathrm{ver}_{[c]}(q)$) of $q$ on $(S,c)$ is a smooth singular measured foliation, with singularities at the zeroes of $q$, which is obtained locally by pulling back the horizontal measured foliations (resp. vertical measured foliation) of $dw^{2}$ under the change of coordinate $z\mapsto w:= \int \sqrt{q}$ defined above. The transverse measure for the horizontal measured foliation (resp. vertical measured foliation) is given by $|\mathfrak{I}\sqrt{q}|$ (resp. $|\mathfrak{R}\sqrt{q}|$).} \\\\If the measured foliation $\f$ is realised by a holomorphic quadratic differential $q$ then $\f$ has a prong of order $k+2$ at the point where $q$ has a zero of order $k$. Also, if $q$ is expressed as $dw^{2}$ in local coordinates  then $-q$ is nothing but the differential $-dw^{2}$, whose horizontal (resp. vertical) foliations are given by the vertical lines (resp. horizontal lines) on $\mathbb{C}$. We thus have the simple but important remark:

{\remark\label{remarkk} $\mathsf{hor}_{[c]}(q)$ is measure equivalent to $\mathsf{ver}_{[c]}(-q)$ in $\MF$ for all holomorphic quadratic differential $q\in Q(S,c)$. }\\\\
Now we recall a well-known theorem of Teichm{\"u}ller that enables us to interpret quasi-conformal deformations in terms of measured foliations (see for example\red{\cite{carlos}}):
{\theorem Given two conformal classes $[c]$ and $[c']$ in $\T$, there exists an unique quasi-conformal map $\phi:[c]\rightarrow [c']$ with minimal eccentricity coefficient among all quasi-conformal maps from $[c]$ to $[c']$. The associated Beltrami differential $\mu$ is of the form $k\frac{ |q|}{q}$ for some unique holomorphic quadratic differential $q\in T^{*}_{[c]}\T$ with $||q||_{1}=1$ and for some $k\in [0,1)$. The quadratic differential $q$ is denoted as the initial quadratic differential of the map. There is a quadratic differential $q'\in T^{*}_{[c']}\T$ denoted as the terminal quadratic differential with the property that the map $\phi$ takes zeroes of $q$ to zeroes of $q'$ of the same order. In the natural local coordinates $z = x + \textbf{i}y$ of $q$ in the complement of its zeroes, and the natural coordinates $w = x' + \textbf{i}y'$ for $q'$, we have:
	\begin{align*}
		x' = \sqrt{k}x, y' = (1/ \sqrt{k})y
	\end{align*}.}

By virtue of the above, the metric $d_{\T}$ defined in Equation (\ref{teichmullergeo}) is a metric on $\T$. Further, we also have that given $(S,c)$, a quadratic differential $q \in T^{*}_{[c]}\T$ with $||q||_{1}=1$ and $t \geq 0$, there is a conformal class $(S,[c_{t}])$, and a unique extremal map $\phi_{t} : (S,c) \rightarrow (S,[c_{t}])$ such that:
\begin{align*}
	d_{\T}([c],[c_{t}])=\frac{1}{2}\log k_{\phi_{t}}
\end{align*}
Choosing $\log k_{\phi_{t}}=2t$ gives us that the image of the map $\mathbb{R}_{> 0} \rightarrow T^{*}\T\cong Q(S)$ with $t\mapsto [c_{t}]$ is a properly embedded geodesic line in $\T$ with respect to the metric $d_{\T}$ which is called the {Teichm{\"u}ller geodesic} with initial quadratic differential $q$. This also makes the metric complete (see, for example \cite{carlos} for more further details). 

\subsection{The Theorem of Hubbard-Masur and the sections $q^{\f}$ and $q^{-\f}$}\label{HMsection} \hfill\break
Define now a map $\mathsf{hor}:Q(\Sig,c)\rightarrow \MF$ which sends a holomorphic quadratic differential $q$ to its horizontal measured foliation $\mathsf{hor}_{c}(q)$ and we consider its image in $\MF$. Then we have from\red{\cite{hubbardmasur}} (see also \cite{Kerckhoff1980},\cite{Wolf1996}):
{\theorem\label{hubbardmasur} The map $\mathsf{hor}_{c}: Q(S, c) \rightarrow \mathcal{MF}(S)$ is a homeomorphism}.

{\remark Given any measured foliation $\f$ on Riemann surface, we may not find a holomorphic quadratic differential realising it as its horizontal measured foliation, for example notice the example in  \S {$2$} of Chapter {$II$} of \red{\cite{hubbardmasur}}. However as noted in the paper, this issue can be taken care of as according to  Proposition {$2.2$} of\red{\cite{hubbardmasur}} as in the equivalence class of $\f$ in $\MF$ there exists a representative that can be realised by a holomorphic quadratic differential.}\\

We will consider the inverse of this map for our purpose which will provide sections of $T^{*}\T$ for a fixed foliation $\f$. This we define as follows: 
{\definition For a given equivalence class of foliation $\f \in \MF$, define 
	\begin{align*}
		q^{\f}:\T \rightarrow T^{*}\T
	\end{align*} to be the map, which associates to each equivalence class of complex structure $[c]$ on $\Sig$, the unique holomorphic quadratic differential $q^{\f}_{c}$ such that $\mathsf{hor}_{c}(q^{\f}_{c})$ is measure equivalent to $\f$. } \\

We will denote the holomorphic quadratic differential associated to $[c]$ as $q^{\f}_{[c]}$. In fact, the theorem of Hubbard-Masur holds true if we consider vertical measured foliations instead of horizontal ones and thus we can consider the map 
\begin{align*}
	q^{-\f}: \T \rightarrow T^{*}\T
\end{align*} which associates to a complex structure $[c]$ on $\Sig$, the unique holomorphic quadratic differential $q^{-\f}_{[c]}$, such that vertical measured foliation of $q^{-\f}_{[c]}$ on $(\Sig,[c])$ is measure equivalent to $\f$. We can thus reformulate Remark \red{\ref{remarkk}} as: 

{\remark\label{lelele} For a given measured foliation $\f \in \MF$, $q^{-\f}_{[c]}=-q^{\f}_{[c]}$ for any $[c]$ in $\T$.}\\\\
These sections are $C^{0}$, in particular they fail to be $C^{1}$ (\cite{Royden1971}). It also follows from a result of Masur in\red{\cite{Masur1995}} that when $q$ is generic, then the sections $q^{+\f},q^{-\f}: \T \rightarrow T^{*}\T$ are real-analytic. In fact we can state:
{\lemma\label{arationaliff} A measured foliation is arational if and only if the holomorphic quadratic differential realising it as the horizontal measured foliation over each point in $\T$ is generic. Moreover in this case the map $q^{\f}: \T\rightarrow T^{*}\T$ is smooth.}\\	

One implication is obvious as if the quadratic differential is generic then the measured foliation it realises has three prongs at each zero. The other side can be seen easily as if $F$ is arational then any Whitehead equivalent measure foliation is isotopic to $F$ (since by definition there are no saddle connections to collapse). We also denote the subset of arational measured foliations as $\mathcal{MF}_{0}(S)$ and note that this is a dense subset of $\MF$ as well. 
\subsection{Filling measured foliations and the theorem of Gardiner-Masur}\label{gardinermasur}\hfill \break 
First we recall that:
	{\definition\label{fillupdef}
	A pair of measured foliations $(\mathsf{F},\mathsf{G})$ is said to {fill $S$} if for any measured foliation $\mathsf{H}\in \MF $ on $S$ we have, 
	\begin{align*}
	i(\mathsf{H},\f)+i(\mathsf{H},\g)>0
	\end{align*}.
} 

Recall that we denote the space of equivalence classes of pairs of filling foliations as $\FMF$. Notice that the pair $(\mathsf{hor}(q),\mathsf{ver}(q))$ automatically satisfies the topological property of {filling up $S$} by the following Lemma {$5.3$} of\red{\cite{Gardiner1991}} (see \cite{grenobleschool} as well):

 {\lemma\label{2.4} Given a holomorphic quadratic differential $q$ on a Riemann surface $(\Sig,[c])$, the pair $(\mathsf{hor}_{[c]}(q),\mathsf{ver}_{[c]}(q))$ fill $\Sig$.}\\
 
 Given a pair $(\f,\g)\in \FMF$ we can thus ask whether under a fixed complex structure up to equivalence, a pair $(\mathsf{F},\mathsf{G})$ can be realized as the horizontal and vertical measured foliation of the same holomorphic quadratic differential $q$. The answers are affirmative and can be summarized as: 
 
 	{\theorem(\red{\cite{Gardiner1991}},\cite{Wentworth2007})\label{GM}
 	A pair $(\mathsf{F},\mathsf{G})$ of measured foliations on $S$ is filling if and only
 	there is a complex structure $c$ and a holomorphic quadratic differential $q \in (S, c)$ such that $(\mathsf{F},\mathsf{G})$ are respectively measure equivalent to the vertical and horizontal
 	foliations of $q$. Moreover, the class $[c]$ up to  diffeomorphism isotopic to the identity is determined uniquely
 	and for each $c \in [c]$ the quadratic differential $q$ realising the filling pair $(\f,\g)$ is also unique. }\\ \\
 So, for a pair $(\f,\g)$ that fill we have: 
{\corollary\label{gola} The sections $q^{\f}$ and $q^{-\g}$ intersect uniquely in $T^{*}\T$ at the point $([c],q)$ determined by Theorem \red{\ref{GM}}. Moreover $\f=\mathsf{hor}_{[c]}(q)$ and $\g=\mathsf{hor}_{[c]}(-q)$ in $\MF$ where $q=q^{\f}=q^{-\g}=-q^{\g}$.}
\subsection{Extremal lengths of measured foliations}\label{extremal}\hfill\break
Given a simple closed curve $\gamma$ on $(\Sig,c)$ we define its extremal length as \begin{align*}
	\ext_{c}(\gamma)= \sup_{[[g]]}\frac{l^{2}_{g}(\gamma)}{\text{Area}(g)}
\end{align*}
where $l_{g}(\gamma)$ is the length computed with respect to $g$ and the supremum is taken over all Riemannian metrics in the conformal class $c$. This definition of extremal length on closed curve extends to that of a measured foliation (see\red{\cite{Kerckhoff1980}}) where extremal length of a foliation $\f$ defines a continuous function on $\T$ 
\begin{align*}
	\ext(\f):\T\rightarrow \mathbb{R}\\
	[c]\mapsto \ext_{[c]}(\f)
\end{align*}
where $\ext_{[c]}(t\f)=t^{2}\ext_{[c]}(\f)$ for $t>0$. Using the sections $q^{\f}_{[c]}$ we can also express this as (see\red{\cite{hubbardmasur}}):
{\lemma\label{babagola} For $[c]\in \T$, the extremal length of $\f\in\MF$ is given by \begin{align*}
		\ext_{[c]}(\f)=\int_{(S,c)}|q^{\f}_{[c]}|dz\wedge d\bar{z}=||q^{\f}_{[c]}||_{1}=||q^{-\f}_{[c]}||_{1}
\end{align*} }
Here $||q||_{1}=\int_{\Sig}|\phi(z)|dz \wedge d\bar z$ where $q=\phi(z)dz^{2}$. Another simple observation that follows from this is:

{\corollary\label{gorombhat} $\ext_{[c]}(\f)=\ext_{[c]}(\g)$ where $[c]\in \T$ is determined by Theorem \red{\ref{GM}}. }
{\proof As $q=q^{\f}_{[c]}=q^{-\g}_{[c]}$ at $[c]$ where the sections $q^{\f}$ and $q^{-\g}$ intersect (see Remark \red{\ref{gola}}), we get the result using Lemma \red{\ref{babagola}}. \qed}\\\\
We also have a well-known variational formula for extremal lengths originally due to Gardiner (see \cite{gardiner1984}, also\cite{Liu2016}) which states:
{\lemma \label{gardform}Let $[c_{t}]$ for $0\leq t<0$ be a smooth $1$-parameter family of conformal classes and $\f$ be a smooth measured foliation in $\MF$ then 
\begin{align*}
	(d \ext_{[c_{0}]}\f) (\mu)=\mathfrak{R}\left\langle q^{\f}_{[c_{0}]},\mu\right\rangle 
\end{align*} where $\mu\in T_{[c_{0}]}\T$ is the Beltrami differential denoting the derivative $\ddt [c_{t}]$.  }

{\remark It is known from \cite{Royden1971} that the extremal length function is not $C^{2}$ in general. The fact that the Hubbard-Masur map is not $C^{1}$ can be now seen from the above formula.}
 \subsection{Intersection of $q^{\f}$ and $q^{-\g}$ in $T^{*}\T$}\label{sectionintersectionofsection}\hfill\break
Using the tools developed thus far we can state:	{\proposition \label{traninter}
 	Let $(\mathsf{F},\mathsf{G})\in\mathcal{FMF}(\Sig)$ be a pair of  measured foliations that fill $S$ and $q^{\mathsf{F}},q^{-\mathsf{G}}: \mathcal{T}(S) \rightarrow T^{*}\mathcal{T}(S)$ be the associated sections defined before. Then their images in $T^{*}\T$ intersect uniquely and the projection of the intersection into $\T$ is the unique critical point of the function $\ext(\f)+\ext(\g):\T \rightarrow \mathbb{R}$. Moreover when $(\f,\g)\in \mathcal{FMF}_{0}(S)$ then the sections intersect transversely. }
 {\proof 
 		Given, $(\mathsf{F},\mathsf{G}) \in \mathcal{FMF}(S)$, the sections $q^{\mathsf{F}}_{[c_{0}]},q^{-\mathsf{G}}_{[c_{0}]}$ intersect if and only if $q^{\f}_{[c_{0}]}=q^{-\g}_{[c_{0}]}=-q^{\g}_{[c_{0}]}$ from Remark \red{\ref{lelele}}. If $\mu$ be the Beltrami differential denoting $\ddt [c_{t}]$, then we have from Lemma \red{\ref{gardform}} that:
 		
 	\begin{align*}
 		\ddt\ext_{[c_{t}]}(\mathsf{F})=  \mathfrak{R}(\langle q^{\mathsf{F}}_{[c_{0}]}, \mu \rangle)=\mathfrak{R}(\langle q^{-\mathsf{G}}_{[c_{0}]}, \mu \rangle),  \mu \in T_{[c]}\mathcal{T}(S).
 		\end{align*}
  We can also consider $\g$ to be a measured foliation realised as the horizontal measured foliation and consider
 	\begin{align*}
 	\ddt\ext_{[c_{t}]}(\mathsf{G})= \mathfrak{R}(\langle q^{\mathsf{G}}_{[c_{0}]}, \mu \rangle),  \mu \in T_{[c]}\mathcal{T}(S).
 	\end{align*}
 	
 	Hence $[c_{0}]$ is a critical point of the function $\ext(\mathsf{F})+\ext(\mathsf{G})$ if and only if $q^{\f}_{[c_{0}]}= -q^{\g}_{[c_{0}]}$. The existence and uniqueness of the critical point follows from Theorem \ref{GM} and Remark \ref{gola}. \\
   Now assume $(\f,\g)\in \mathcal{FMF}_{0}(S)$. If $dq^{\f},dq^{-\g}: T\T \rightarrow T T^{*}\T$ be the respective differentials then for transversality of intersection we need to show that if $dq^{\f}(\nu)=dq^{-\g}(\nu)$ for $\nu \in T\T$ then $\nu = 0$. Recalling the definitions of the sections $q^{\f},q^{-\g}$ this amounts to showing:
   {\lemma Consider a deformation of the type $(c_{t},q_{t}),t\geq 0$ with $(c_{0},q_{0})$ being the point of intersection of $q^{\f},q^{-\g}$ with $(\f,\g)\in \mathcal{FMF}_{0}(S)$. Let $\f_{t},\g_{t}$ be the horizontal and vertical measured foliations realised by $(c_{t},q_{t})$ and assume $\ddt \f_{t} = \ddt \g_{t}=0$. Then the deformation is trivial. }
   
   {\proof Recall that $q^{\f},q^{-\g}:\T \rightarrow T^{*}\T$ is a smooth map when $(\f,\g)$ are arational. For this, we consider $\pi:\widehat{S}\rightarrow S$ to be the canonical double cover branched over the zeroes of $q$  (see\red{\cite{Lanneau2003}}, Construction {$1.2$} and also\red{\cite{Dumas2015}}) such that $\pi^{*}(q)=\omega_{q}^{2}$ where $\omega_{q}$ is a holomorphic $1$ form on $(\widehat{S},[\widehat{c}])$ where $[\widehat{c}]=\pi^{*}[c]$.	 It follows from Lemma $2$ of\red{\cite{Lanneau2003}} that $([c],q)\mapsto ([\widehat{c}],\omega_{q}^{2})$ is a local embedding, so a deformation $([c_{t}],q_{t})$ in the generic stratum induces a deformation $([\widehat{c_{t}}],\omega_{q_{t}}^{2})\in Q(\widehat{\Sig})$ maintaining the same strata. Let $\widehat{\f_{t}}$ nd $\widehat{\g_{t}}$ be the horizontal and vertical foliations realised by $\omega_{q_{t}}^{2}$ on $(\widehat{\Sig},[\widehat{c_{t}}])$ with $\widehat{\f_{0}}=\widehat{\f}$ (resp. $\widehat{\g_{0}}=\widehat{\g}$) being the lift of $\f$ (resp. $\g$) in the double cover. Consider now  $\gamma$ to be a cycle in the  relative homology group $H^{-}_{1}(\widehat{\Sig},V_{\omega_{q}};\mathbb{C})$ where the latter is the eigenspace of $H_{1}(\widehat{\Sig},V_{\omega_{q}};\mathbb{C})$ consisting of cycles invariant under the involution of $\widehat{S}$ and the set $V_{\omega_{q}}$ denotes the set of zeroes of $\omega_{q}$. The real and imaginary part of the holonomy $\int_{\gamma}\omega_{q}$ are precisely the intersection number with the horizontal and vertical foliations of $\omega_{q}^{2}$. This gives us
    Period coordinates \begin{align*}
   	\mathsf{per} &: Q_{0}(S)\hookrightarrow H^{1}_{-}(\widehat{S},V_{\omega_{q}};\mathbb{C})\cong \mathbb{R}^{12g-12}\\
   & \mathsf{per}(q) \mapsto \int_{\gamma} \omega_{q} \end{align*} which is an immersion. Our assumption then translates to $\ddt i(\gamma,\widehat{\f_{t}})=0$ and $\ddt i(\gamma,\widehat{\g_{t}})=0$. This gives $\ddt \mathsf{per}(q_{t})(\gamma)=\ddt i(\gamma,\widehat{\g_{t}})+\textbf{i}i(\gamma,\widehat{\f_{t}})=0$, where we can assume $\gamma$ is fixed when one restricts to deformations maintaining strata. Since the period map is an immersion, it follows that this deformation is necessarily trivial. \qed \\\\ }
 	Define now : {\definition For a pair $(\f,\g)$ that fill $\Sig$, we denote $\p(\f,\g)$ to be the critical point in $\T$ of the function $\ext(\f)+\ext(\g):\T \rightarrow \mathbb{R}$}.\\\\
 	It is a simple observation from the definition that if the transverse measure of a foliation $\f$ is given by $|\mathfrak{R}\sqrt{q^{\f}}|$, then the corresponding holomorphic quadratic differential realising the measured foliation $t\f$ over the same Riemann surface structure on $\Sig$ is nothing but $t^{2}q^{\f}$ since then the transverse measure is given by $|\mathfrak{R}\sqrt{t^{2}q^{\f}}|$ which is equal to $t|\mathfrak{R}\sqrt{q^{\f}}|$. In the notation of the critical point $\p(\f,\g)$ this implies that:
 	{\lemma\label{scalediff} If $(\f,\g)$ fill $S$, then $\p(t\f,t\g)=\p(\f,\g)$ where $t>0$.}
 	{\proof Observe that $\ext(t\f)+\ext(t\g)= t^{2}(\ext(\f)+\ext(\g))$ from the definition of extremal length function and hence $\ext(t\f)+\ext(t\g)$ and $\ext(\f)+\ext(\g)$ have the same critical points. \qed} \\\\
 	Also we have the observation that this point is uniquely determined by the second coordinate. That is:
 	{\lemma\label{oneone} If $\mathsf{p}(\mathsf{F},\mathsf{G})=\mathsf{p}(\mathsf{F'},\mathsf{G}) $, then $\mathsf{F}= \mathsf{F'}$ in $\MF$.}
 	{\proof Let $\p(\f,\g)=[c]$ be the unique point in $\T$ and $q\in Q(\Sig,c)$ be the unique holomorphic quadratic differential realising $(\f,\g)$ as its horizontal and vertical measured foliations respectively. For the pair $(\f',\g)$ we have that $\p(\f',\g)=p(\f,\g)=[c]$. Since on $Q(S,c)$ the choice of $q'$ realising $\g$ as its vertical measured foliation is unique from the theorem of Hubbard-Masur, we have that $q'=q$. But by definition $\f'$ is measure equivalent to $\mathsf{hor}_{[c]}(q')=\mathsf{hor}_{[c]}(q)=\f$.\qed}
 	
	\subsection{Quotient of $Q(\Sig)$ under the action of $\mathbb{R}_{>0}$ and intersection of $[q^{\f}]$ and $[q^{-\g}]$ }\label{sectionquotientinter}\hfill\break 
There is a natural action of $\mathbb{R}_{>0}$ on $(Q(S, [c])-{0}) $ which sends every non-zero $q \in Q(S, [c]) $ to $ t^{2}q$,$ \forall t \in (0, \infty)$. We can thus define $Q^{1}(S,c)$ to be quotient $(Q(S, [c])-{0})/\mathbb{R}_{>0}$ under this action. Clearly $Q^{1}(S,c)$ is isomorphic to $ UT^{*}_{[c]}\mathcal{T}(S)$ from Proposition \ref{keylemma}, where the latter denotes the unit cotangent space at a point $[c] \in \T$. The next proposition is a similar result for the sections $[q^{\f}]$, which are the images of $q^{\mathsf{F}}$ under the quotient map. We can now address the main proposition of this section involving the intersection of the equivalence classes $[q^{\f}]$ and $[q^{-\g}]$ in $UT^{*}\T$ for a filling pair $(\f,\g)$: 
{\proposition\label{line}
	Let $(\mathsf{F},\mathsf{G})\in \FMF$ be a pair filling measured foliations on $S$, then the projection of the intersection of the sections $[q^{\f}]$,$[q^{-\g}] $ in $ UT^{*}\mathcal{T}(S)$ onto $\mathcal{T}(S)$ is a Teichm{\"u}ller geodesic line given by $t\mapsto \p(\sqrt{t}\f,\frac{1}{\sqrt{t}}\g)\in \T$ for $t>0$. Moreover, when $(\f,\g)\in \mathcal{FMF}_{0}(S)$ then the sections intersect transversely in $UT^{*}\T$.
}

{\proof
	
	We first note that if a pair $(\f,\g)$ fill $\Sig$ then so do the pairs $(t\f,\g)$,$(t\f,\frac{1}{t}\g)$ and $(\f,t\g)$ for any $t>0$. \\
	Let $[c_{t}]\in \T$ be an equivalence class of complex structures such that the two sections $q^{t\f},q^{-\g}$ meet over $[c_{t}]$. Then by definition we have $ t^{2}q^{\mathsf{F}}_{[c_{t}]}=q^{-\mathsf{G}}_{[c_{t}]}$ which is equivalent to $tq^{\f}_{[c_{t}]}=\frac{1}{t}q^{-\g}_{[c_{t}]}$. Since the foliation $t\f$ is realised by $t^{2}q^{\f}_{[c_{t}]}$ on the same complex structure, we have that $ q^{\sqrt{t}\mathsf{F}}_{[c_{t}]}=q^{-\frac{1}{\sqrt{t}}\mathsf{G}}_{[c_{t}]}$ for some $t>0$ at the point $[c_{t}]$. \\This is equivalent to the fact that $[c_{t}]$ is the unique critical point of the function $\ext(\sqrt{t}\mathsf{F})+\ext(\frac{1}{\sqrt{t}}\mathsf{G})$ since $(\sqrt{t}\mathsf{F},\frac{1}{\sqrt{t}}\mathsf{G})$ fill $S$. As $[c_{t}]$ is identified with $\mathsf{p}(\sqrt{t} \f, \frac{1}{\sqrt{t}}\g) $, the projection of the intersections is along 
	\begin{align*}
		\mathbb{R}_{>0} \rightarrow \T \\
		t \mapsto [c_{t}]=\p(\sqrt{t}\f,\frac{1}{\sqrt{t}}\g)
	\end{align*}
	So it now suffices to show that the path $t\rightarrow \p(t\f,\g)$ is a geodesic for the Teichm{\"u}ller metric $d_{\T}$ on $\T$. We first note that $[c]$ being the critical point $\ext(\f)+\ext(\g)$ for a filling pair $(\f,\g)$ also implies that $[c]$ is the critical point for the function
	 \begin{align*}
	\ext(\f)\ext(\g): \T \rightarrow \mathbb{R}
		\end{align*} 
		where we use the fact that $\ext_{[c]}(\f)=\ext_{[c]}(\g)$ from corollary \ref{gorombhat}. So $\p(\sqrt{t}\f,\frac{1}{\sqrt{t}}\g)$ is a critical point for $\ext(\sqrt{t}\f)\ext(\frac{1}{\sqrt{t}}\g)$ and since $\ext_{[c]}(t\f)=t^{2}\ext_{[c]}(\f)$ we also have as a consequence that the point $\p(\sqrt{t}\f,\frac{1}{\sqrt{t}}\g)$ is a critical point for the function $\ext(\f)\ext(\g)$. Now it has been shown in\cite{Gardiner1991} that the set of critical points for the function $\ext(\f)\ext(\g)$ is a Teichm{\"u}ller geodesic line in $\T$ when $(\f,\g)$ fill $S$. Moreover, from Lemma \ref{oneone} the map $t\mapsto \p(t\f,\g)\in \T$ is injective. Finally, we observe that every critical point of $\ext(\f)\ext(\g)$ is also a critical point for $\ext(\alpha\f)+\ext(\beta \g)$ for some $\alpha, \beta >0$ and hence the the image of the map $t\rightarrow \p(t\f,\g)$ is the entire Teichm{\"u}ller geodesic. \\
		For transversality we can use Proposition \ref{traninter} as the pairs $q^{t\f}$ and $q^{-\g}$ intersect transversely, i.e, \begin{align*}
			T_{([c_{t}],q_{t})} T^{*}\T = T_{([c_{t}],q_{t})}(q^{t\f}(\T)) \bigoplus T_{([c_{t}],q_{t})} (q^{-\g} (\T))
		\end{align*} is true for all $t\geq 0$ and $(\f,\g)\in \mathcal{FMF}_{0}(S)$. The result follows when we take quotient.\qed   \\\\
			For a given pair $(\f,\g)\in \FMF$ we call  $t \mapsto \p(\sqrt{t}\f,\frac{1}{\sqrt{t}}\g)\in \T$ for $t>0$ as $\pp(\f,\g)$. A further corollary is that:
		{\corollary The line $\pp(\f,\g)$ is a Teichm{\"u}ller geodesic.}\\\\
		For what will follow later we also want to conclude more about the regularity of the intersection of the sections $q^{\f}$ and $q^{-\g}$ when $(\f,\g)$ are a filling arational pair.

	\section{Necessary condition for paths with small filling measured foliations at infinity}\label{minsec}
	
The goal of this section is to establish a necessary conditions that small differentiable  paths in $\QF$ starting from $\F$ should satisfy if the measured foliations at infinity are given by a filling pair $(t\fp,t\fm)\in \FMF$ at first order at $\F$. For this reason following\cite{Uhlenbeck1984}, we will study the curve $\beta_{([c],q)}(t^{2})\in \QF$, for $t>0$ small enough, which is parametrised by the data of the unique minimal surface it contains, i.e, the first fundamental form $I$ is in the conformal class $[c]\in \T$ and the second fundamental form $\II$ is given by $t^{2}\mathfrak{R}(q)$ for some $q\in T^{*}_{[c]}\T$. We will compute first-order estimates for Schwarzians at infinity for this path and determine that if the measured foliations at infinity for this path is indeed $(t\fp,t\fm)$ at first-order at $\F$ then $[c]$ is indeed the unique critical point for the functions $\ext(\fp)+\ext(\fm):\T\rightarrow \mathbb{R}$ and $q$ is the unique holomorphic differential we obtain from the theorem of Gardiner-Masur that realise $(\fp,\fm)$ on $[c]$ . To this end, we will begin by recalling the definition of the Schwarzians at infinity and then introduce the tools required to make the computations regarding their first-order estimations.
\subsection{Schwarzians at infinity of quasi-Fuchsian manifolds}\label{introductionary}\hfill \break 
Recall that the boundary at infinity $\partial_{\infty}\mathbb{H}^{3}$ is identified with the complex projective space $\mathbb{C}P^{1}$. The components $\partial^{+}_{\infty}M$ and $\partial^{-}_{\infty}M$ are the quotient of domains in $\mathbb{C}P^{1}$ under the action of $\Gamma < PSL_{2}(\mathbb{C})$ and so they carry canonical $\mathbb{C}P^{1}$-structure which is a $(G,X)$ structure on $S$ with $G=PSL_{2}(\mathbb{C})$ and $X=\mathbb{C}P^{1}$\red{\cite{Dumas}}. That is to say we have an open covering of $\Sig$ by an atlas $(U_{\alpha},\phi_{\alpha})$ such that $\phi_{\alpha}:U_{\alpha}\rightarrow \mathbb{C}P^{1}$ are charts to open domains in $\mathbb{C}P^{1}$ and in the overlap  $U_{i}\cap U_{j}$  of two charts the change of coordinate map $\phi_{i}\circ \phi^{-1}_{j}$ is locally a restriction of a M{\"o}bius transformations. Denote the space of equivalence classes of $\mathbb{C}P^{1}$-structures on $\Sig$ under diffeomorphisms isotopic to the identity as $\mathcal{CP}(\Sig)$.  \\\hspace*{.5 cm}Now, given a $\mathbb{C}P^{1}$-structure on $\Sig$ we have an underlying complex structure as M{\"o}bius transformations are biholomorphisms and for $\partial^{\pm}_{\infty}M$ it is precisely $[c_{\pm}]$ up to equivalence. This gives us a natural forgetful map $\mathcal{CP}(\Sig)\rightarrow \T$ mapping a $\mathbb{C}P^{1}$-structures to the underlying complex one. Now by the  Uniformisation theorem any complex structure $c$ on $\Sig$ arises as the quotient of the action of some discrete subgroup  $\Gamma_{c}$ of $PSL_{2}(\mathbb{R})$ on $\mathbb{H}^{2}$. As $\Gamma_{c}<PSL_{2}(\mathbb{C})$ as well and $\mathbb{H}^{2}$ can be seen as the unit disc $\Delta\subset \mathbb{C}P^{1}$, we have a canonical $\mathbb{C}P^{1}$-structure associated to a complex structure which we call the {standard Fuchsian complex projective structure} and this gives us a continuous section $\T \rightarrow \mathcal{CP}(\Sig)$.\\\hspace*{.5 cm}
The Schwarzian derivative yields a parametrisation of the fibers  of the forgetful map $\mathcal{CP}(\Sig)\rightarrow \T$. In general, given a domain $\Omega\subset \mathbb{C}$, 
the Schwarzian derivative of a locally injective holomorphic map $u:\Omega \rightarrow \mathbb{C}$ is a holomorphic quadratic differential defined as:
\begin{align*}
\sigma(u)= ((\frac{u''}{u'})^{'}-\frac{1}{2}(\frac{u''}{u'})^{2})dz^{2}
\end{align*} One way to obtain the expression on the right hand side above is to consider the unique M{\"o}bius transformation
$M_{u}$ which matches with $u$ up to second order derivative. The expression above is precisely the difference of
the third order terms in the local Taylor series expansion of $u$ and $M_{u}$ (see Proposition $6.3.3$ of\cite{Hubbard}). Further, they have two remarkable properties:
\begin{itemize}
	\item For two locally injective holomorphic maps $u,v: \Omega \rightarrow\mathbb{C}$ we have $\sigma(u\circ v)= v^{*}\sigma(u)+\sigma(v)$
	\item $\sigma(A)=0$ if and only if $A$ is a M{\"o}bius transformation.
\end{itemize} In particular, there is a unique map defined up to right conjugation by M{\"o}bius transformations between a given complex projective structure on $S$ and the standard
Fuchsian one which is holomorphic with respect to the underlying canonical complex structure and by virtue of
the properties above, the Schwarzian derivative for this holomorphic map can be defined in a chart independent
way. This is called the Schwarzian parametrisation of a complex projective structure on $S$ with respect to
the standard Fuchsian complex projective structure (see \cite{Dumas},\cite{jmnotes}). When $M \in \QF$, the components of $\partial_{\infty}\mathbb{H}^{3}\setminus \Lambda_{\Gamma}$  have non-trivial Schwarzian derivatives
associated to them by construction which descend to two holomorphic quadratic differentials on $\partial^{\pm}_{\infty}M$ upon
taking quotients. So we define: {\definition  The Schwarzians at infinity $\sigma_{+}$ and $\sigma_{-}$ are the holomorphic quadratic differentials obtained on $(\partial^{+}_{\infty}M,[c_{+}])$ and $(\partial^{-}_{\infty}M,[c_{-}])$ by the Schwarzian parametrisation of the $\mathbb{C}P^{1}$ structures on $\partial^{\pm}_{\infty}M$ with respect to the corresponding standard Fuchsian complex projective structure.}\\\\
Also note that when $M\in\F$ due to the second property above the Schwarzians at infinity are zero.

		\subsection{Minimal surfaces in hyperbolic $3$ manifolds}\label{minimalsurface}\hfill\break
	 Let $i: \Sig \rightarrow M$ be an immersion of $\Sig$ into $M$. Associated to this, we have the data of the {first fundamental form } denoted as $I$ on $\Sig$ which is the induced metric on $\Sig$ inherited from the ambient space and the {second fundamental form} denoted as $\II$ which is a symmetric bilinear form on the tangent bundle $T\Sig$. Associated to the couple $(I,\II)$ we have a unique self-adjoint operator $B$, called the {shape operator}, which satisfies the relation \begin{align*}
	 	\II(x,y)=I(Bx,y)=I(x,By)
	 \end{align*} where $x,y \in T_{p}S, \forall p \in S$. We can also define another quantity associated to the immersion called the third fundamental form $\III$ defined as $\III(x,y)=I(Bx,By)=I(B^{2}x,y)$ for $x,y,B$ as before. See\cite{Perdigao} for reference.\\\hspace*{.5 cm}
	 The eigenvalues of $B$ give us the principal curvatures associated to the immersion. Since minimal immersions are those immersions for which the mean curvature of $\Sig$ is zero we can define as well that: 
	{\definition An immersion is minimal if and only if $B$ is traceless. }
\\

On the other hand, assume we are given a smooth Riemannian metric $g$ on $\Sig$ and a symmetric bilinear form $h$ on $T\Sig$. The couple $(g,h)$ will be associated to the data of an immersion  of $\Sig$ if it satisfies the following:

\begin{enumerate}

	\item The Codazzi equation, $d^{\nabla}h=0$, where $\nabla$ is the Levi-Civita connection of $g$.
	\item The Gauss equation, $K_{I}=-1+\deter_{g}(h)$, where $K_{g}$ denotes the Gaussian or intrinsic curvature of the metric $g$.
	\end{enumerate}}
 Now we define the notion of an {almost}-Fuchsian manifold as: 

{\definition A quasi-Fuchsian hyperbolic $3$-manifold $M \cong \Sig \times \mathbb{R}$ is called {almost-Fuchsian} if it contains a closed minimal surface homeomorphic to $\Sig$ with principal curvatures in $(-1,1)$.}\\\\
		Now, following Corollary $2.9$ of\red{\cite{Krasnov2007}} we can show that if $M$ is an almost-fuchsian manifold then the closed minimal surface it contains is unique. This allows us to parametrise almost-Fuchsian hyperbolic manifolds by an open subset $\Omega$ in $T^{*}\T$, following \cite{Krasnov2007}, Theorem $2.12$:
	{\theorem\label{methodandoutline} There exists an open neighbourhood $\Omega \subset T^{*}\T$ of the zero-section such that we have the following bijection:
	\begin{enumerate}
		\item Given a point $([c],q)\in \Omega$, there exists a unique $g\in \AF$ such that the unique minimal surface in $(M,g)$ has first fundamental form conformal to $[c]$ and the second fundamental form $\II$ is $\mathfrak{R}(q)$.
		\item Given a almost-Fuchsian metric $g \in \AF$ on $M$, the induced metric and  second fundamental form of its unique minimal surface are specified by a point in $\Omega$. 
\end{enumerate}} 
Using the minimal immersion data we will now use foliations by surfaces equidistant from it to approach the boundary at infinity of $M$.

 \subsection{Fundamental forms at infinity}\label{sectionfundadorms}\hfill\break
 Given a minimal surface in an almost-Fuchsian manifold $M$, we can consider the surfaces equidistant from it in $M$ at an oriented distance. These surfaces foliate the almost-Fuchsian manifold and we can then compute the associated first and second fundamental forms for these surfaces in terms of the data associated to the minimal embedding. We thus can formulate the following\red{\cite{Krasnov2008}} (see also  \cite{jeanmarcpaper}):
 	{\lemma
 	Let $S$ be a complete, oriented, smooth surface with principal curvatures in $(-1,1)$ immersed minimally into an almost-Fuchsian manifold homeomorphic to $S \times (-\infty,\infty)$ and let $(I,\II,B)$ be the associated data of the immersion. Then $\forall r \in \mathbb{R}$ the set of point $S_{r}$ at an oriented distance $r$ from $S$ is a smooth embedded surface with data $(I_{r},\II_{r},B_{r})$ where :
 	\begin{enumerate}
 		\item $I_{r}(x,y)=I((cosh(r)E+sinh(r)B)x,(cosh(r)E+sinh(r)B)y)$
 		\item $\II_{r}=\frac{1}{2}\frac{dI_{r}}{dr}$
 		\item $B_{r}=(cosh(r)E+sinh(r)B)^{-1}(sinh(r)E+cosh(r)B)$
 	\end{enumerate}
 	where $E$ is the identity operator and $S_{r}$ is identified to $S$ through the closest point projection.}\\

The {fundamental forms at infinity} denoted as $I^{*},\II^{*}$ and introduced in\red{\cite{Krasnov2008}}, quantify the asymptotic behaviour of the quantities described above as $r \rightarrow \infty$. In particular, it estimates the data at the conformal class at infinity of an almost-Fuchsian manifold $M$ with respect to the, unique minimal surface with principal curvature in $(-1,1)$, it contains. \\\\
Formally, $I^{*}=\lim_{r \rightarrow \infty}2e^{-2r}I_{r}$ and $\II^{*}=\lim_{r \rightarrow \infty}(I_{r}-\III_{r})$. However the lemma above gives us explicit formulae to express the same in terms of $(I,\II,\III)$ and we use that to define:
{\definition
	Adhering to the notations introduced above, the first fundamental form at infinity is given by the expression $I^{*}=\frac{1}{2}(I+2\II+\III)$ and 
	the second fundamental form at infinity is given by $\II^{*}=\frac{1}{2}(I-\III)$.
}\\\\
The pair $(I^{*},\II^{*})$ satisfy a modified version of Gauss equation at infinity (see\cite{jmnotes},\cite{Krasnov2007}), i.e, $\trace(B^{*})=-K^{*}$ where $B^{*}$ is the shape operator associated to $I^{*}$ and $\II^{*}$. The Codazzi equation on the other hand, holds as it is by considering the Levi-Civita connection $\nabla^{*}$ compatible with $I^{*}$. The thing for importance to us is the expression for curvature associated to $I^{*}$ which we call $K^{*}$.\red{\cite{Krasnov2008}} further provide us with an expression for it using the data of the immersed minimal surface:
{\lemma \label{formatinfinity}
	With the notation as above, 
	\begin{align*}
	K^{*}= \frac{2K}{\deter(E+B)}=\frac{-1+\deter(B)}{1+\deter(B)}
	\end{align*} 
	where $K$ is the Gaussian curvature of the minimal immersion of $S$. }\\
{\remark
	The second equality follows from the fact that $Tr(B)=0$ the immersion being minimal and $(I,\II)$ satisfy the Gauss-Codazzi equations.}\\

In general $I^{*}$ need not be a hyperbolic metric. In fact,\red{\cite{Krasnov2008}} note that when multiplied by the correct conformal factor to take $I^{*}$ to the unique hyperbolic metric in its conformal class, the corresponding change in $\II^{*}$ is closely related to the Schwarzian derivative $\sigma$ associated to that end. So we have the following accounting for the change in $(\II^{*})_{0}$ when we apply a conformal change to $I^{*}$: 
{\lemma\label{conformalchange} Let $I^{*}_{1}$ and $I^{*}_{2}$ be two metrics in the same conformal class at infinity such that $I^{*}_{2}=e^{2f}I^{*}_{1}$ for some smooth function $f$, then the traceless parts $(\II^{*}_{1})_{0}$ and $(\II^{*}_{2})_{0}$ are related as:
\begin{align*}
	(\II^{*}_{2})_{0}-(\II^{*}_{1})_{0}=\hess_{I^{*}_{1}}(f)-df\otimes df + \frac{1}{2}||df||I^{*}_{1}-\frac{1}{2}(\Delta f)I^{*}_{1}
\end{align*}} In fact if we consider a holomorphic map $u:\Omega \rightarrow \mathbb{C}$ where $\Omega\subset \mathbb{C}$ then $\mathfrak{R}(\sigma(u))$ is precisely the term on the right hand side of the above equation when we consider $2f=\log(\frac{u'}{u})$. We thus have the following: (a geometric proof of which can also be found in Appendix $A$ of\red{\cite{Krasnov2008}})}:

{\theorem\label{schlenker}
	If $I^{*}$ is hyperbolic, then $(\II^{*})_{0}=-\mathfrak{R}(\sigma)$, where $(\II^{*})_{0}$ denotes the traceless part of the second fundamental form at infinity and $\sigma$ is the Schwarzian at infinity.
}\\\\
In the following sections we will use the parametrisation of almost-Fuchsian metric in terms of the data of $I$ and $\II$ of its unique minimal surface and compute $I^{*}$ and $\II^{*}$ at the two ends. For this, we will use the curve introduced by Uhlenbeck in\red{\cite{Uhlenbeck1984}} to prove that quasi Fuchsian metrics close enough to $\F$ admit a minimal surface with data given by a point $([c],sq)\in T^{*}\T$, for $s>0$ sufficiently small.

\subsection{The curve $\beta_{([c],q)}(s)$ in $\QF$}\label{sectioncurveinQF}\hfill\break\\
	Let $([c],q)$ be a point in $ T^{*}\T$. As discussed in\red{\cite{Uhlenbeck1984}} we consider a smooth $1$-parameter curve $\beta_{([c],q)}(s)$ , $s \in [0,\epsilon)$  of almost-Fuchsian metrics starting from the Fuchsian locus which are given by the data 
	\begin{align*}
		\beta_{([c],q)}:\mathbb{R}_{>0}\rightarrow T^{*}\T\supset &\Omega\cong\AF\subset \QF\\
	s\mapsto ([c],s\mathfrak{R}(q))
	\end{align*}
	 of the unique minimal surface such that $I$ is $e^{2u_{s}}h$ for some function $u_{s}:\Sig \rightarrow \mathbb{R}$, where $h$ denotes the unique hyperbolic metric in the conformal class $c$, and $\II=s\mathfrak{R}(q)$. At $s=0$, we have $u_{s}=0$ and $\beta_{([c],q)}(0)\in \F$. \\\\
	By Gauss equation, the pair $(e^{2u_{s}}h,s\mathfrak{R}(q))$ is the data of the minimal immersion if and only if $u_{s}$ is a solution for the following equation:
		\begin{align}\label{gausscodazzo}
	e^{-2u_{s}}(-\Delta_{h} u_{s}-1)= -1 + e^{-4u_{s}}s^{2}\deter_{h}(\mathfrak{R}(q))
	\end{align}

	{\remark This is a reformulation of the Gauss equation $K_{g}=-1+\deter(B)$ for the pair $(e^{2u_{s}}h,s\mathfrak{R}(q))$. The left hand side comes from Lemma \red{\ref{formula}}. The right hand side comes by the formulae for change of basis for determinants.}\\\\
		It is then known from\cite{Uhlenbeck1984} (see also\cite{Trautwein2019}) that a unique solution exists for Equation \ref{gausscodazzo} which in terms of almost-Fuchsian metrics can be formulated as:
	
	{\proposition\label{citethisprop} 	For $0\leq s < \epsilon$, $\exists$ a unique almost-Fuchsian metric with a unique minimal surface whose $(I,\II)$ is given by the pair $([c],s\mathfrak{R}(q))$.} 
	
	\subsection{First order estimations of measured foliations at infinity for the path $\beta_{([c],q)}(t^{2})$}\label{sectionfoliationsatinfinity}\hfill\break
	We will in fact do all the computations for the path $\beta_{([c],q)}(s)$ and perform a change of variable of $s$ to $t^{2}$ later on. This is done in order to account for the correct factor of the measured foliations at infinity at first order that we will compute eventually. Let us fix some notations: For a fixed $s>0$, the data of the minimal surface $S$ embedded into an almost-Fuchsian manifold $M$ can be expressed as 
	$I_{s}=e^{2u_{s}}h$ and $\II_{s}=s\mathfrak{R}(q)$.
		Since $B_{s}=I_{s}^{-1}\II_{s}$ and $\III_{s}(x,y)=I_{s}(B_{s}^{2}x,y)$, a simple computation in local orthonormal coordinates for $I_{s}$ show that $\III_{s}$ is equal to $-s^{2}e^{-2u_{s}}(\deter_{I_{s}}(\mathfrak{R}(q)))h$.
	 Let the associated fundamental forms at infinity for this manifold be $I^{*}_{s},\II^{*}_{s}$ and the curvature at infinity be $K^{*}_{s}$. Further, let the Schwarzian at infinity associated to the two ends of $\beta_{([c],q)}(s)$ be called $\sigma^{s}_{+}$ and $\sigma^{s}_{-}$.\\\hspace*{.5 cm}
	Our goal first is to say that $I^{*}_{s}$ is hyperbolic at first order at $s=0$, so that we can apply a first order version of Theorem \red{\ref{schlenker}} relating the traceless part of $\II^{*}_{s}$ with the real part of $\sigma^{s}_{+}$.

{\lemma\label{firstorderi}
	$I^{*}_{s}$ is hyperbolic at first order at $\F$ i.e the derivative of the curvature $K^{*}_{s}$ with respect to $s$ vanishes at $\F$ and $K^{*}_{0}=-1$ at $s=0$. Moreover, for this path $\ddtt  (\II^{*}_{s})_{0}=-\mathfrak{R}(q)$  .}
{\proof
	First we note that at $s=0$ we are at the Fuchsian locus and from Lemma \ref{formatinfinity} we have $K^{*}_{0}=-1$.  Now observe that \begin{align*}
		\frac{d}{ds}\Bigr|_{\substack{s=0}} (\deter(B_{s})) = \frac{d}{ds}\Bigr|_{\substack{s=0}} s^{2}e^{-2u_{s}}\mathfrak{R}(q) = 0
	\end{align*}Therefore using Lemma \ref{formatinfinity} \begin{align*}
	\frac{d}{ds}\Bigr|_{\substack{s=0}}	K^{*}_{s}= \frac{d}{ds}\Bigr|_{\substack{s=0}} \frac{-1+\deter(B_{s})}{1+\deter(B_{s})}=0
	\end{align*}
	For the next part we first see that $u_{s}$ solves:
	
	\begin{align}\label{step1}
	e^{-2u_{s}}(-\Delta_{h} u_{s}-1)= -1 + e^{-4u_{s}}s^{2}\deter_{h}(\mathfrak{R}(q))
	\end{align}
	
	We define the non-linear map:
	\begin{align}
	F: W^{(2,2)}(S) \times [0,\infty)  \rightarrow   L^{2}(S) &&\\
   	F(u_{s},s)=-\Delta_{h} u_{s} -1\label{step2} +e^{2u_{s}}-e^{-2u_{s}}s^{2}\deter_{h}(\mathfrak{R}(q)) 
	\end{align}
	where $W^{(2,2)}$ is the classical Sobolev space. The Fr{\'e}chet derivative is given by:
\begin{align*}
	dF_{(u_{s},s)}(\dot u_{s},\dot s)=-\Delta \dot u_{s} +2\dot u_{s} e^{2u_{s}}+2\dot u_{s} e^{-2u_{s}}s^{2}\deter_{h}(\mathfrak{R}(q))-2e^{-2u_{s}}s \deter_{h}(\mathfrak{R}(q))
\end{align*}It is clear that $u=u_{s}$ solves Equation (\ref{step1}) if and only if $(u_{s},s)$ is a solution for Equation (\ref{step2}). We now see the linearised operator with respect to $u$ of the function $F(u,s)$ which has the expression: 
	\begin{align*}
	L_{u_{s}}(\dot u_{s})=-\Delta_{h}\dot u_{s} +2\dot u_{s}(e^{2u_{s}}+e^{-2u_{s}}s^{2}\deter_{h}(\mathfrak{R}(q))
	\end{align*}

	So, at the solution $(u_{s},s)=(0,0)$ we have that
	\begin{align*}
		L_{u_{0}}:W^{(2,2)}(S) \rightarrow L^{2}(S)\\ \dot u_{s} \mapsto -\Delta_{h} \dot u_{s}+ 2\dot u_{s}
	\end{align*} is a linear isomorphism of vector spaces see \cite{cmcdifian}. So  
we can apply Implicit Function Theorem to get the solution curve $\gamma : [0,\epsilon)\rightarrow W^{(2,2)}(S)\times [0,+\infty)$ where $\gamma(s):=(u_{s},s)$  satisfies $F(u_s,s)=0,\forall s \in [0,\epsilon)$ (see also Theorem $5.13$ of\red{\cite{Trautwein2019}}). Now \begin{align*}
	dF_{(u_{s},s)}(\dot u_{s},\dot s)=L_{u_{s}}(\dot u_{s})-2e^{-2u_{s}}s\deter_{h}(\mathfrak{R}(q))
\end{align*}
We have that $dF(\dot u_{s},\dot s)=0$ for this path so,
\begin{align*} 
\dot u_{s}= L_{u_{s}}^{-1}(2se^{-2u_{s}}\deter_{h}(\mathfrak{R}(q))\\
\implies \ddtt u_{s}=0
\end{align*}
Now recall that $I^{*}_{s}=\frac{1}{2}(I_{s}+2\II_{s}+\III_{s})=\frac{1}{2}(e^{2u_{s}}h+2s\mathfrak{R}(q)-s^{2}e^{-2u_{s}}\deter_{I_{s}}(\mathfrak{R}(q))h)$. So taking derivative at $s=0$ gives us: 
\begin{align}\label{japarishlekh}
\frac{d}{ds}\Bigr|_{\substack{s=0}} I^{*}_{s}= \frac{1}{2}(2\dot u_{0}h+2\mathfrak{R}(q)+0)=\mathfrak{R}(q).
\end{align}

Now, by same computation observe that \begin{align*}
	\ddtt \II^{*}_{s}= \ddtt(I_{s}-\III_{s})=0.\end{align*} \\
	Now,\red{\cite{Krasnov2008}} further shows that that the mean curvature at infinity is expressed as $H^{*}_{s}=-K^{*}_{s}$. 
	Writing \begin{align*}
	&	\II^{*}_{s}=(\II^{*}_{s})_{0}+H^{*}_{s}I^{*}_{s}\\
		\implies \ddtt (-\II^{*}_{s})_{0} &= H^{*}_{0}\ddtt I^{*}_{s}+I^{*}_{0}\ddtt H^{*}_{s} \\
\implies \ddtt(\II^{*}_{s})_{0}&= -\ddtt I^{*}_{s}
	\end{align*} 
	From Equation \ref{japarishlekh} we have our claim.  \qed}\\ \hfill\break
Now from Theorem \ref{schlenker} we know that if $I^{*}_{s}$ is hyperbolic then $(\II^{*}_{s})_{0}$ is equal to $-\mathfrak{R}(\sigma^{s}_{+})$ where $\sigma^{s}_{+}$ is the Schwarzian at the positive end at infinity. Moreover if we parametrise the quasi-Fuchsian space by the data of hyperbolic metric and Schwarzian at infinity at one end at infinity:
\begin{align*}
 \mathfrak{A}:	\QF \rightarrow T^{*}\mathcal{T}(\partial^{+}_{\infty}M)\\
	g\mapsto (I^{*}_{h},(\II^{*}_{h})_{0})
	\end{align*}
then at the point $[c]\in \F$ of the Fuchsian locus we have a canonical decomposition of the tangent space $T_{[c]}(\QF)=T_{[c]}\T \bigoplus T^{*}_{[c]}\T $ where the first factor is the tangent to the Fuchsian locus denoting the derivative of the hyperbolic metric and the second factor is the derivative of the schwarzian at infinity at the Fuchsian locus. When considering the path $\beta_{([c],q)}(s)$ we have that $\ddtt \mathfrak{A}( \beta_{([c],q)}(s))=(\mathfrak{R}(q),-\mathfrak{R}(q))$. So, 

{\lemma\label{firstorderschwarz}
	For the path $\beta_{([c],q)}(s)$, $\ddtt \sigma^{s}_{+}=q$ 
}\\

Note that we have done all the computation at one boundary component at infinity of $M$, which is almost-Fuchsian. However, recall that $M$ admits a foliation by surfaces "parallel" to the minimal surface, and the corresponding computation for the other component will differ by a sign. To be precise, $I^{*}_{s}=\frac{1}{2}(I_{s}-2\II_{s}+\III_{s})$ when we consider the component at the boundary at the other end at infinity (see \cite{Krasnov2007}). The rest of the computation follows as it is. Keeping this in mind we have:
{\proposition\label{fluckit}
	For the path $\beta_{([c],q)}(s)$, $\ddtt \sigma^{s}_{\pm}=\pm q$  .}\\

Upon a change of variable from $s$ to $t^{2}$, we will now show that the path $\beta_{([c],q)}(t^{2})$ is indeed a candidate for a path of almost-Fuchsian metrics with measured foliations at infinity given by the pair $(t\fp,t\fm)$ at first order. Denote the measured foliations at infinity for a metric in this path to be $\mathfrak{F}(\beta_{([c],q)})(t^{2})=(\f^{t}_{+},\f^{t}_{-})$. Here again a 1-parameter family of foliations $\f^{t}$ is said to be equivalent to a foliation $\f$ at first order, if for any given closed curve $\gamma$ on $\Sig$ 
\begin{align*}
	\ddt i(\gamma,\f^{t} )=i(\gamma,\f) \\ 
	\implies i(\gamma,\f^{t})= ti(\gamma,\f)+o(t)
\end{align*}
	 Note by Proposition \red{\ref{fluckit}} $\sigma^{t^{2}}_{\pm}=\pm t^{2}q+o(t^{2})$ at first order at $t^{2}=0$ (or at $\F$) for the path $\beta_{([c],q)}(t^{2})$. So we need to show:

		{\lemma\label{firtorderfoliations} For any isotopy class of simple closed curve $\gamma$  on $S$ we have:
		\begin{align*}
		i(\gamma,\f^{t})-i(\gamma,\mathsf{hor}_{[c]}(t^{2}q))= o(t)
		\end{align*}	}
		{\proof 
		We proceed with the computation :
		\begin{align*}
		i(\gamma,\f^{t})-i(\gamma,\mathsf{hor}_{[c]}(t^{2}q))&= \inf_{\gamma}\int_{\gamma} | \mathfrak{I}\sqrt{\sigma_{t^{2}}}|- \inf_{\gamma}\int_{\gamma} |\mathfrak{I}\sqrt{t^{2}q} | \\
		&=\inf_{\gamma}\int_{\gamma} |\mathfrak{I} \sqrt{t^{2}q+O(t^{2})}|- \inf_{\gamma}\int_{\gamma} |\mathfrak{I}\sqrt{t^{2}q} | \\
		&=\inf_{\gamma}\int_{\gamma}|\mathfrak{I}\frac{\sigma_{t^{2}}-t^{2}q}{\sqrt{\sigma_{t^{2}}}+\sqrt{t^{2}q}}|
		=o(t)
		\end{align*}
		and we have our claim. \qed \\
\subsection{Necessary conditions for paths with given small filling measured foliations at infinity at first order}\label{sectioncondition}\hfill\break		
		So we see that for $0< t <\epsilon$ metrics in the path $\beta_{([c],q)}(t^{2})$ have that the measured foliations at infinity $(\f^{t}_{+},\f^{t}_{-})$ which at first order at the Fuchsian locus is given by the filling pair $(t\fp,t\fm)$. Secondly, notice that the point $([c],q)$ is the unique point associated to the filling pair $(\fp,\fm)$ via the Gardiner-Masur Theorem and $\beta_{([c],q)}(0)=[c]=\p(\fp,\fm)$ is the unique critical point for the functions $\ext(\fp)+\ext(\fm)$ by Proposition \red{\ref{traninter}}. These two points will precisely help us to formulate the condition we want paths with given first order behaviour of measured foliations at infinity to satisfy. 
	 {\proposition \label{neccesary}  Let $(\fp,\fm)$ be a pair of measured foliations that fill $S$. Then there exists differentiable curve of quasi-Fuchsian metrics $t \mapsto \beta_{([c],q)}(t^{2})$, for $t\in[0,\epsilon)$, starting from the Fuchsian locus such that the image $\mathfrak{F}(\beta_{([c],q)}(t^{2}))\in \FMF$ is measure equivalent to $(t\fp,t\fm)$ at first order at $\F$. Moreover $[c]\in \T$ is the unique critical point of the function $\ext(\fp)+\ext(\fm): \T \rightarrow \mathbb{R}$ and $q\in T^{*}_{[c]}\T$ is the unique holomorphic quadratic differential realising $(\fp,\fm)$.}
{\proof From  Proposition \ref{fluckit} and Lemma \ref{firtorderfoliations} we have that measured foliations at infinity for this path are given by the pair $(t\mathsf{hor}_{[c]}(q),t\mathsf{hor}_{[c]}(-q))$ at first order at $t=0$. The proposition is then a consequence of Proposition \ref{traninter}.\qed}

\section{Uniqueness of Paths with Small Filling Foliations}\label{pathexists}	

The goal of this section is to construct differentiable paths realising small pairs of measured foliations at infinity which are arational and filling, utilising the condition proved in Proposition \red{\ref{neccesary}} that they should satisfy. To do that first we will introduce the {blow-up} space $\widetilde{\QF}$ which we obtain by replacing $\F \subset \QF$ with its "unit normal bundle" $UN\mathcal{F}(\Sig)$. Following the strategy of\red{\cite{bonahon05}} 
we then consider subsets of $\QF$ called $\Wfp$ (and $\Wfm$), defined as: {\definition For $\f\in \mathcal{MF}(S)$, define $\Wf^{+}\subset \QF$ (resp. $\Wf^{-}$) to be the set of quasi-Fuchsian metrics $g$ such that the foliation at the end at $+\infty$ (resp. $-\infty$) is $t\f$ for all $t\geq 0$.}\\\\Call $\widetilde{\mathcal{W}^{\pm}_{\f}}$ the image of $\Wf^{\pm}$ under the lift $\QF\rightarrow \widetilde{\QF}$. For $(\fp,\fm)\in \mathcal{FMF}_{0}(S)$ we will then show in the blow-up space $\widetilde{\Wfp}$ and $\widetilde{\Wfm}$ are submanifolds of $\widetilde{\QF}$ and that their boundaries $\partial \widetilde{\Wfp}$ and $\partial\widetilde{\Wfm}$ contained and intersecting in $\partial \widetilde{\QF}$ where $\partial \widetilde{\Wfp}\cap \partial \widetilde{\Wfm}$ intersect transversely and project onto the Teichm{\"u}ller geodesic line $\mathsf{P}(\fp,\fm)\in \T$ as defined in Proposition \red{\ref{line}}. 
We then consider the map ${\pi}:{\Wfp}\cap{\Wfm}\rightarrow {\mathbb{R}^{2}}$, sending $g \in \widetilde{\Wfp}\cap\widetilde{\Wfm}$ to $(a,b)$ where $\mathfrak{F}(g)=(a\fp,b\fm)$ for some $a,b\geq 0$ by definition and lift the setting to the blow-up $\widetilde{\pi}: \widetilde{{\Wfp}}\cap \widetilde{\Wfm} \rightarrow \widetilde{\mathbb{R}^{2}}$ where the latter denotes the blow-up of $\mathbb{R}^{2}$ at the origin. The existence of paths with given small foliations then follows as we show that $\widetilde{\pi}$ at $\partial \widetilde{\Wfp} \cap \partial \widetilde{\Wfp}$ is a local diffeomorphism.

	\subsection{The normal bundle $N\F$ to $\F$ }\label{sectionnormal}\hfill\break\\
First, let us recall that the Weil-Petersson metric endows $\T$ with a symplectic form $\omega_{WP}$ which is defined on the cotangent space as 
\begin{align*}
	\omega_{WP}(.,.)= -\mathfrak{I}\left\langle .,.\right\rangle _{WP}
\end{align*} Moreover, $\T$ is endowed with an almost complex structure $J_{WP}$ such that $\left\langle q_{1},q_{2}\right\rangle_{WP} = \omega_{WP}(q_{1},J_{WP}(q_{2}))$ defined by $J_{WP}(q_{2})=\textbf{i}q_{2}$. Further, recall the notion of the {character variety} $\chi_{PSL_{2}(\mathbb{C})}$  which is an irreducible affine variety of complex dimension $6g-6$ (\cite{Goldman1988}) and can be expressed as the GIT quotient:
\begin{align*}
\chi_{PSL_{2}(\mathbb{C})}:= \mathsf{Hom}(\pi_{1}(S),PSL_{2}(\mathbb{C}))//PSL_{2}(\mathbb{C})
\end{align*} As each hyperbolic structure on $S$ is uniquely determined by the holonomy representation of $\pi_{1}(S)$ in to the group of orientation preserving isometries of $\mathbb{H}^{2}$, identified with $ PSL_{2}(\mathbb{R})$, the Fuchsian locus $\F(\cong \T)$ can be identified with a connected component of the set of {real points} in $\chi_{PSL_{2}(\mathbb{C})}$(\cite{Goldman1988}). Now, the group of orientation preserving isometries of $\mathbb{H}^{3}$ is identified with $PSL_{2}(\mathbb{C})$ and so the space $\QF$ is also identified with an open neighbourhood of $\F$ in $\chi_{PSL_{2}(\mathbb{C})}$ via discrete faithful representations from $\pi_{1}(S)\rightarrow PSL_{2}(\mathbb{C})$ which we can associate to a quasi-Fuchsian metric. This provides $\QF$ with a complex structure $J_{PSL_{2}(\mathbb{C})}^{2}=-1$ which also gives a decomposition of the tangent space at a point $[c]\in\F$ as $T_{[c]}\QF=T_{[c]}\F\bigoplus J_{PSL_{2}(\mathbb{C})}T_{[c]}\F
$. This enables use to recall {Bers' Simultaneous Uniformisation Theorem}: 
{\theorem[\red{\cite{bers1965moduli}}]\label{bersthm}
	The map $B:\QF \rightarrow \T \times \mathcal{T}(\overline{S})$ mapping a quasi-Fuchsian metric $ g \in \QF $ to the pair $B(g):=([c_{+}],[c_{-}])$ is biholomorphic with respect to the complex structure $J_{PSL_{2}(\mathbb{C})}$ of $\QF$ coming from the character variety and the complex structure $J_{WP}$ on $\T$.} \\\\
  It is also clear that $\F$ is the pre-image of the diagonal.  If $v\in T_{[c]}\QF$ is the tangent vector to the path $t\rightarrow g_{t}$ of quasi-Fuchsian metrics for $0\leq t <\epsilon$ at $t=0$ such that $g_{0}\in \F$ then the derivative of the Bers map at a point $[g_{0}]\in \T$  is given by  
$
d_{[g_{0}]} B(v) := (q_{1},q_{2})$
 where $q_{1},q_{2}$ are two holomorphic quadratic differentials in $Q(S,[g_{0}])$ denoting tangent vectors to $\T$ associated to the variation of the complex structures at two ends at infinity corresponding to the vector $v$. We thus have that $
d_{[g_{0}]}B(J_{PSL_{2}(\mathbb{C})}v)=(J_{WP}(q_{1}),-J_{WP}(q_{2}))=(\textit{\textbf{i}}q_{1},-\textit{\textbf{i}}q_{2})
$ where $\textit{\textbf{i}}^{2}=-1$ and the minus sign in the second factor is simply due to the opposite orientation of $S$.\\\hspace*{.5 cm}As $\F$ is identified with $\T$ by considering the unique hyperbolic metric $m$ in each conformal class $[c]$, when we consider a deformation of hyperbolic structures on $S$ the tangent vector $T_{m}\F\cong T_{m}\T$ is given by $\mathfrak{R}(q)$ for some $q\in Q(S,c)$ (see\cite{Fischer}). If one considers the variation of the hyperbolic metrics in the conformal classes associated to the two ends at infinity, then $d_{[c]}B(J_{PSL_{2}(\mathbb{C})}v)=(\mathfrak{R}(\textit{\textbf{i}}q_{1}),\mathfrak{R}(-\textit{\textbf{i}}q_{2}))\in T_{[c]}\F\times T_{[c]}\F$. We can thus define:
\begin{definition}
	The {normal bundle} $N\F\rightarrow \F$ is the bundle whose fiber $N_{[c]}\F$ over each conformal class $[c]\in\T$ is the vector space isomorphic to the quotient $T_{[c]}\QF/T_{[c]}\F$.
\end{definition}
The fibers $N_{[c]}\F$ are $J_{PSL_{2}(\mathbb{C})}T_{[c]}\F$ where $J_{PSL_{2}(\mathbb{C})}$ is the almost complex structure of $\QF$. So $N_{[c]}\F$ is the set of tangent vectors $v_{([c],q)}$ such that $dB(J_{PSL_{2}(\mathbb{C})}v_{([c],q)})=(\mathfrak{R}(\textit{\textbf{i}}q), \mathfrak{R}(\textit{\textbf{i}}q))\in T\T \times T\T $ for $q\in Q(S,c)$.
Now, let $v_{([c],q)}$ be the vector tangent to the path $\beta_{([c],q)}: [0,\epsilon) \rightarrow \QF$ at $\F$. That is to say:
\begin{align*}
	v_{([c],q)}=\ddtt \beta_{([c],q)}(s)\in T_{[c]}\QF
\end{align*} So we formulate:
{\proposition $v_{([c],q)}$ is an element of $N_{[c]}\F$. }
{\proof As the first fundamental form at infinity $I^{*}_{s}$ remains hyperbolic at first order at $s=0$ it follows from Lemma \red{\ref{firstorderi}} that \begin{align*}
	dB(v_{([c],q)})=(\mathfrak{R}(q),-\mathfrak{R}(q))
	\end{align*} Then, \begin{align*}
		dB(J_{PSL_{2}(\mathbb{C})}v_{([c],q)})=(\mathfrak{R}(\textit{\textbf{i}}q),\mathfrak{R}(\textit{\textbf{i}}q))
	\end{align*}  So $J_{PSL_{2}(\mathbb{C})}v_{([c],q)}\in T_{[c]}\F$, i.e, tangent to the Fuchsian locus. The decomposition of the tangent space $T_{[c]}\QF= T_{[c]}\F \bigoplus J T_{[c]}\F $ at the Fuchsian locus then implies that $v_{([c],q)}\in N_{[c]}\F$.\qed\\

\subsection{The blow-up $\widetilde{\QF}$ of $\QF$ at $\F$}\label{sectionblowup}\hfill\break
For constructing the blow-up $\widetilde{\QF}$ consider again the bundle $N\F$ defined over $\F$. So we take the quotient of $N_{[c]}\F\setminus{0}$ by the action of $\mathbb{R}_{>0}$ called the unit normal bundle $UN_{[c]}\F$ and let $\overline{v_{([c],q)}}$ be the image of $v_{([c],q)}\in N_{[c]}\F$.  Consider now $\eta(N\F)\rightarrow UN\F$ to be the canonical differentiable line bundle and also we have a canonical linear map $\eta ({N\F)}\rightarrow \F $. We can show now that $\eta(N\F )\setminus {(\text 0-section)}\cong N\F\setminus \F$ is a diffeomorphism. Note that the zero section of $\eta(N\F)$ is again $UN\F$.\\\hspace*{.5 cm}
Now let $\tau$ be a tubular map for $\F$ in $\QF$ and $\theta: \eta(N\F)\rightarrow N\F$ be the canonical map. The blow up $\widetilde{\QF}$ is the unique differentiable structure on $(\QF\setminus \F)\cup UN\F$ for which the inclusion map $\QF \setminus \F \subset \widetilde{\QF}$ and the map:
\begin{align*}
	&\eta(N\F) \rightarrow \widetilde{\QF} \\
	&	v \mapsto \tau (\theta(v)) \hspace{.3cm}{\text{when}\hspace{.1cm} v\in \eta(N\F)\setminus{\text{(0-section)}}}\\
	&v \mapsto v \hspace{.3cm}\text{otherwise}
\end{align*} are embeddings (see, for example\cite{Brocker1982}, pg. $128$). \\

 Moreover, it is a manifold with boundary $\partial \widetilde{\QF}$ which is $UN\F$. We observe that the natural inclusion $\QF\setminus\F\hookrightarrow \widetilde{\QF}$ lifts to $\QF\rightarrow \widetilde{\QF}$ by sending $[c]\in\F$ to $UN_{[c]}\F\in \partial \widetilde{\QF}$.\\\hspace*{.5 cm}
Recall now the spaces of $\QF$, $\Wf^{+}$ and $\Wf^{-}$ for some $\f \in \mathcal{MF}(S)$.
Since $\F \subset \Wf^{\pm}$, we have the natural inclusion $\Wf^{\pm}\setminus\F\hookrightarrow\QF\setminus\F$ which again lifts to a unique embedding $\Wf^{\pm} \rightarrow \widetilde{\QF}$ that we obtain by replacing a point $[c]\in \F \subset \Wf^{\pm} $ by the unique normal vector in $\overline{v_{([c],q)}}\in UN_{[c]}\F$ tangent to $\Wf^{\pm}$ at $[c]$. So by construction of the blow-up, for $t$ small enough $(\overline{v_{([c],q)}},t)\mapsto([c],t^{2}\mathfrak{R}(q))$ has $\f^{t}_{\pm}$ given by $t\f$ at first order at $\F$. By Lemma \red{\ref{firstorderschwarz}} and Theorem \red{\ref{hubbardmasur}},  $\overline{v_{([c],q^{\pm\f}_{[c]})}}$ is indeed that vector.
We thus define: 
{\definition $\widetilde{ \Wf^{+}}$ and $\widetilde{ \Wf^{-}}$ are the respective lifts of $\Wf^{+}$ and $\Wf^{-}$ into $\widetilde{\QF}$.}\\\\
Having removed the Fuchsian locus which carry trivial Schwarzians from $\QF$, we will now parametrise elements in $\widetilde{\QF}$ by the data of the holomorphic quadratic differential being realised as the Schwarzian derivatives at the boundaries at infinity to show that $\widetilde{\Wf^{+}}$ and $\widetilde{\Wf^{-}}$ are submanifolds with boundary of $\widetilde{\QF}$. For this we first recall that the Schwarzians at infinity parametrise the $\mathbb{C}P^{1}$-structures on $\partial^{+}_{\infty}M$ and $\partial^{-}_{\infty}M$ (see \cite{Dumas}). More generally, if we denote the space of equivalence classes of $\mathbb{C}P^{1}$-structures on $\Sig$ under diffeomorphisms isotopic to the identity as $\mathcal{CP}(\Sig)$, then the Schwarzian derivative provides us parametrisation of the fibers of the forgetful map $\mathcal{CP}(S)\rightarrow \T$(\cite{Dumas}). So we formulate:
{\lemma\label{graft} A quasi-Fuchsian metric on $M\cong S \times \mathbb{R}$ can be uniquely parametrised by the Schwarzian parameterisation of the $\mathbb{C}P^{1}$-structures on one component of the boundary at infinity.}   

{\proof Denote the induced metric on the boundary of the convex core $\mathcal{CC}(M)$ as $m$. The lemma then follows in two steps. First, a quasi-Fuchsian metric on $M$ can be uniquely determined by the data of induced metric and measured bending lamination in the boundary of the convex core from\red{\cite{AFST_1996_6_5_2_233_0}} where we have that the map from $\QF \rightarrow\T \times \ML $ that associates the data of the unique pleated surface to the data of the quasi-Fuchsian metric is biholomorphic and so, smooth. Consider now the data $(m_{\pm},d_{m_{\pm}}(l(\lambda_{\pm})))$ which gives us a point in $T^{*}\T$, $d_{m_{\pm}}(l(\lambda_{\pm}))$ being the derivative of the length function for the measured lamination $\lambda_{\pm}$ computed at $m_{\pm}$. The claim is then a consequence of Theorem $4.1$ of\red{\cite{Dumas}}, originally due to Thurston and the main theorem in \cite{Krasnov2009}, which together state that the smooth \text{Grafting map} sending the data of the induced metric and measured bending lamination $(m_{\pm},\lambda_{\pm})\in \T \times \ML \cong T^{*}\T \rightarrow  \mathcal{CP}(S)$  on the boundary of the convex core $\partial^{\pm}_{\infty}\mathcal{CC}(M)$ to the data $([c_{\pm}],\sigma_{\pm})\in T^{*}\mathcal{T}(\partial^{\pm}_{\infty}M)$ at the boundary at infinity is a homeomorphism and $C^{1}$.\qed
 	
\subsection{Submanifolds $\widetilde{\Wfp}$, $\widetilde{\Wfm}$ and the intersection  $\partial \widetilde{ \Wfp}\cap\partial\widetilde{ \Wfm}$}\label{sectionblowupintersect}\hfill\break
 As Lemma \red{\ref{graft}} allows us to parametrise quasi-Fuchsian structures uniquely by the data of Schwarzian derivatives at the boundaries $\partial^{\pm}_{\infty}M$ we can thus proceed to discuss the following:

{\proposition\label{5.2} For $\f \in \mathcal{MF}_{0}(S)$, the set ${\Wf^{+}}$ (resp. ${\Wf^{-}}$) is a smooth submanifold with boundary of ${\QF}$ of dimension $\text{dim}(\T)+1$ with boundary $\F$. In the blow-up $\widetilde{\QF}$, the lifts $\widetilde{\Wfp}$ and $\widetilde{\Wfm}$ are again smooth submanifolds with the boundary $\partial \widetilde{\Wf^{+}}$ (resp. $\partial \widetilde{\Wf^{-}}$) contained in $\partial \widetilde{\QF}$.}
{\proof We just treat the case of $\widetilde{ \Wf^{+}}$ as the same proof holds for $\widetilde{ \Wf^{-}}$ by symmetry. First we will show that $\Wf^{+}\setminus \F$ is a submanifold of $\QF\setminus\F$. Note that:
	\begin{align*}
		\QF \rightarrow \mathcal{CP}(\partial^{+}_{\infty}M)\hookrightarrow Q_{0}(\partial^{+}_{\infty}M)
	\end{align*}
	gives us an embedding of $\QF$ into an open subset of $Q_{0}(\partial^{+}_{\infty}M)$ by Lemma \ref{graft}.
For a given $[c_{+}]$ we have an unique $\sigma^{\f}_{[c_{+}]}\in T^{*}\mathcal{T}(\partial^{+}_{\infty}M)$ realising $\f$ as the horizontal measured foliation at $\partial^{+}_{\infty}M$. This gives us the map $\sigma^{\f}:\mathcal{T}(\partial^{+}_{\infty}M) \rightarrow T^{*}\mathcal{T}(\partial^{+}_{\infty}M)$ which is identified with the map $q^{\f}$. Recall now that over the same complex structure $[c]$, $t\f$ is realised by $t^{2}\sigma^{\f}_{[c]}$. We see that $\Wf^{+}\setminus \F$ is locally embedded as $\mathbb{R}_{>0}\times \sigma^{\f}(\mathcal{T}(\partial^{\pm}_{\infty}M))$ in $\mathbb{R}^{12g-12}$ via the period coordinates of $\sigma^{\f}$. In other words, it is the image of the embedding:
\begin{align*}
	\mathbb{R}_{\geq 0} \times \mathcal{T}(\partial^{+}_{\infty}M)\rightarrow Q_{0}(\partial^{+}_{\infty}M)\hookrightarrow \mathbb{R}^{12g-12}
\end{align*} where the last inclusion is via the period coordinates associated to the dense stratum which gives us coordinate charts into $\mathbb{R}^{12g-12}$. Notice that the Fuchsian locus, corresponding to the zero section of $T^{*}\T$ has zero Schwarzian and thus the period associated is also zero. We also note that the smoothness of this submanifold is by virtue of the map $q^{\f}$ being real analytic when restricted to arational measured foliations. The dimension of the submanifolds being clearly $\text{dim}(\mathcal{T}(\partial^{\pm}_{\infty}M))+\text{dim}(\mathbb{R}_{>0})=\text{dim}(\T)+1$. \\
    Consider now the blow-up $\widetilde{Q_{0}(\partial^{+}_{\infty}M)}$ which is $(Q_{0}(S)\setminus\T)\cup Q^{1}_{0}(S)$ with the $C^{1}$ structure described in \S\S \ref{sectionblowup}. Since $(\overline{v_{[c],q^{\f}}})\in UN\F$ gets mapped to $\mathfrak{R}([q])\in UT\T$ under the isomorphism $UN\F\cong UT\T$ and $\mathfrak{R}(q)$ again corresponds to $[q]\in Q^{1}(S)$ by Weil-Petersson duality, we have an open embedding $\widetilde{\QF}\hookrightarrow \widetilde{Q_{0}(\partial^{+}_{\infty}M)}$. 
  	Recall that $[c]\in \F\subset \Wf^{+}$ is associated to the unit normal vector $\overline{v_{([c],q^{\f}_{[c]})}}$ in $UN_{[c]}\F\subset \partial\widetilde{ \Wf^{+}}$ since $v_{([c],q^{\f}_{[c]})}$ is realised by the path $\beta_{([c],q^{\f}_{[c]})}(t^{2})$, and $\sigma^{t}_{+}$ for this path is indeed $t^{2}q^{\f}_{[c]}$ at first order at $\F$ by Lemma \red{\ref{firstorderschwarz}} and Lemma \red{\ref{firtorderfoliations}}. So $\partial\widetilde{\Wf^{+}}$ is contained in $Q^{1}_{0}(S)$, the boundary of $\widetilde{Q_{0}(S)}$.  \\
	For modifying the argument for $\widetilde{ \Wf^{-}}$ we need to consider the vector $q^{-\f}_{[c]}$ at $[c]\in \F\subset\Wf^{-}$, since the foliation at negative end at infinity for the path $\beta_{([c],q)}(t^{2})$ is given by $\mathsf{hor}_{[c]}(-q)$ at first order at $\F$ by Proposition \red{\ref{fluckit}} and Lemma \red{\ref{firtorderfoliations}}. The rest of the argument follows as it is and we have our claim. \qed}\\

We can now claim the following:

{\proposition\label{transbound} When $(\fp,\fm)\in \mathcal{FMF}_{0}(\Sig)$, $\partial\widetilde{\Wfp}$ and $\partial\widetilde{\Wfm}$ intersect transversely in $\partial\widetilde{\QF}$. Moreover, their intersection is equal to $(\mathsf{P}(\fp,\fm),\overline{v_{([c_{t}],q_{t})}})\ni UN\F $ where $[c_{t}]\in\T$ is the unique critical point of the function  $\ext(\sqrt{t}\fp)+\ext(\frac{1}{\sqrt{t}}\fm): \T \rightarrow \mathbb{R}$ and $q_{t}\in Q(S,c_{t})$ is the unique holomorphic quadratic differential realising them.}
	
{\proof At $[c]\in \F$, $\partial\widetilde{\Wfp}$ is equal to $\overline{v_{([c],q^{\fp}_{[c]})}}\in UN_{[c]}\F$ and $\partial\widetilde{\Wfm}$ is given by the vector $\overline{v_{([c],q^{\fm}_{[c]})}}\in UN_{[c]}\F$. We know from \S \S \ref{sectionblowup} that $dB(J_{PSL_{2}(\mathbb{C})}\overline{v_{([c],q^{\fpm}_{[c]})}})\in UT_{[c]}\QF$ is given by $(\mathfrak{R}(\textit{\textbf{i}}[q^{\fpm}_{[c]}]),\mathfrak{R}(\textit{\textbf{i}}[q^{\fpm}_{[c]}]))\in UT_{[c]}\F\times UT_{[c]}\F$. So, $\partial\widetilde{\Wfp}$ intersects $\partial\widetilde{\Wfm}$ uniquely  at $[c]\in \F$ if and only if the sections $\mathfrak{R}([q^{\fp}])$ and $\mathfrak{R}([q^{\fm}])$ do in $UT\T$. Again, for some $[c]\in \T$, $\mathfrak{R}([q^{\fp}_{[c]}])=\mathfrak{R}([q^{\fm}_{[c]}])$ if and only if $[q^{\fp}_{[c]}]=[q^{\fm}_{[c]}]$ via the duality between $ T^{*}\T$ and $T\T$. So from Proposition \red{\ref{line}} we have that $[c]$ is the unique critical point $ \p(\sqrt{t}\fp,\frac{1}{\sqrt{t}}\fm)$ of the function $\ext(\sqrt{t}\fp)+\ext(\frac{1}{\sqrt{t}}\fm):\T\rightarrow \mathbb{R}$ for some $t>0$. From proposition \ref{line} we see that the boundaries $\partial \widetilde{ \Wfp}$ and $\partial\widetilde{ \Wfm}$ intersect along the image of the Teichm{\"u}ller geodesic line $t \mapsto \p(\sqrt{t}\fp,\frac{1}{\sqrt{t}}\fm)$ under the section map $\F\rightarrow UN\F$. The transversality of their intersection follows from that of the submanifolds $[q^{\fp}]$ and $[q^{-\fm}]$ shown in Lemma \ref{line}.
	 \qed}}\\\\
Define now the map 
\begin{align*}
	\pi : \Wfp \cap \Wfm \rightarrow \mathbb{R}\times\mathbb{R}
\end{align*}
which sends $g\in \Wfp \cap \Wfm$ to the pair $(a,b)$ such that
$\mathfrak{F}(g)=(a\fp,b\fm)$ by definition of $\mathfrak{F}$. Observe that under this map, $\F$ gets mapped to $\left\lbrace 0\right\rbrace :=(0,0)$ and for $(\fp,\fm)\in \mathcal{FMF}_{0}(S)$ the map $\pi$ is smooth. So, if $\widetilde{\mathbb{R}^{2}}$ be the blow-up of $\mathbb{R}^{2}$ at the origin, then $\pi$ lifts to a smooth map 
\begin{align*}
	\widetilde{\pi}:\widetilde{\Wfp}\cap\widetilde{\Wfm}\rightarrow \widetilde{\mathbb{R}^{2}}
\end{align*}
Here $\widetilde{\mathbb{R}^{2}}$ is the set $\mathbb{R}^{2}\setminus\left\lbrace 0\right\rbrace \cup UT_{\left\lbrace 0\right\rbrace }\mathbb{R}^{2}$ where $UT_{\left\lbrace 0\right\rbrace }\mathbb{R}^{2}$ is the quotient of the tangent space at origin under the action of $\mathbb{R}_{>0}$. 
So we have 
	{\proposition \label{localdiffeo} 
	For a pair $(\fp,\fm)\in \mathcal{FMF}_{0}(S)$, the map $\widetilde{\pi}$ is a local diffeomorphism near $\partial\widetilde{\Wfp}\cap \partial\widetilde{\Wfm}$ onto its image.}

{\proof 
 We want to show that the map $\widetilde{\pi}$ has a solution at the Fuchsian locus, is invertible at that point and subsequently apply implicit function theorem. For this we show that $\widetilde{\pi}$ is a local immersion and local submersion at $p\in\partial\widetilde{ \Wfp}\cap\partial\widetilde{\Wfm}$, i.e to prove that \begin{align*}
		d_{p}\widetilde{\pi}: T_{p}\widetilde{\Wfp} \cap T_{p}\widetilde{\Wfm} \rightarrow T_{\widetilde{\pi}(p)}\widetilde{\mathbb{R}^{2}}
	\end{align*} is injective and surjective. Note that when restricting to $(\fp,\fm)$ in arational pairs, this map is indeed smooth as the submanifolds $\widetilde{\mathcal{W}}^{\pm}_{\f_{\pm}}$ are for $\f\in \mathcal{MF}_{0}(S)$.\\
	If for some $v \in  T_{p}\widetilde{\Wfp} \cap T_{p}\widetilde{\Wfm} $ we have that $d_{p}\widetilde{\pi}(v)=\left\lbrace 0\right\rbrace $ then we want to show first that $v$ is in the intersection of the tangent spaces to the boundary $T_{p}\partial \widetilde{\Wfp}\cap T_{p}\partial \widetilde{\Wfm}$. Let $m_{+}:\Wfp\rightarrow [0,\infty)$ be the map such that 
   $m_{+}(g)= t$ for some $g\in \Wfp$ which has measured foliation at the boundary at infinity given by $t\fp$. This induces a map $\widetilde{m_{+}}: \widetilde{\Wfp}\rightarrow [0,\infty)$ in the blow-up space as well. Observe that if we analogously define a map $\widetilde{m_{-}}:\widetilde{\Wfm} \rightarrow [0,\infty)$ then $\widetilde{\pi}:=(\widetilde{m_{+}},\widetilde{m_{-}})$. So if $v\in \mathsf{Ker}(d\widetilde{\pi})$ then $v\in \mathsf{Ker}(d\widetilde{m_{+}})\cap\mathsf{Ker}(d\widetilde{m_{-}})$ then this implies that $v\in T_{p}\partial \widetilde{\Wfp} \cap T_{p} \partial\widetilde{\Wfm}$.

\begin{figure}
		\centering
		\includegraphics[width=0.7\linewidth]{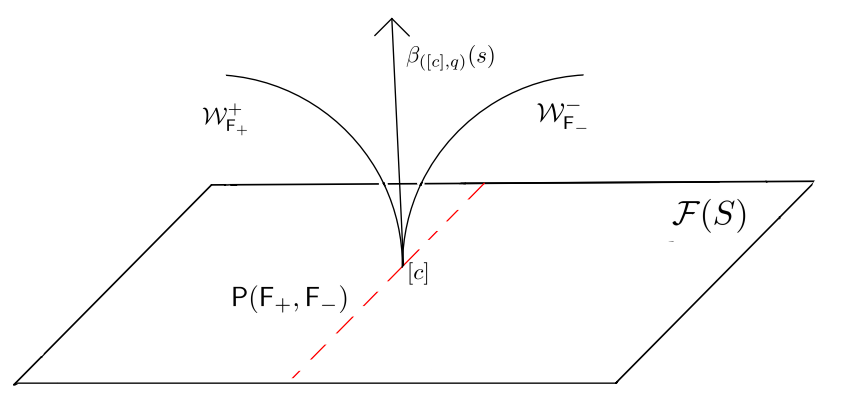}
		\caption{Schematic diagram of the path $\beta_{([c],q)}(s)$ leaving the Fuchsian locus $\F$ from the point $[c]$ along the direction of the normal vector $v_{([c],q)}$ with $\Wfp$ and $\Wfm$ intersecting at $\F\subset \QF$ prior to blow-up procedure for an arational pair $(\fp,\fm)$ which fills $S$ and the dashed line representing the Teichm{\"u}ller geodesic $\pp(\fp,\fm)$.}
		\label{fig:pathintersect}
	\end{figure}
	
Now let $\kappa: \widetilde{\mathbb{R}^{2}_{\geq 0}}\rightarrow \mathbb{R}_{>0}\times\mathbb{R}_{\geq 0}  $ be the chart which sends $(x,y) \mapsto (\frac{x}{y},y)$. Recall from Proposition \red{\ref{transbound}} that $ \partial\widetilde{\Wfp}\cap \partial\widetilde{\Wfm} $ has been shown to be the lift of the line $\mathsf{P}(\fp,\fm)$ in $\T\cong\F$ by the section $\F\rightarrow UN\F$ sending $\mathsf{P}(\fp,\fm) \ni [c]\mapsto ([c],\overline{v_{([c],q^{\fp}_{[c]})}})=([c],\overline{v_{([c],q^{-\fm}_{[c]})}}) \in UN_{[c]}\F$.\\ For a fixed $t>0$, let $[c]$ be the critical point of $\ext(\sqrt{t}\fp)+\ext(\frac{1}{\sqrt{t}}\fm)$ which is equivalent to being the critical point of $\ext(t\fp)+\ext(\fm)$. Let $s\mapsto \widetilde{g_{s}} \in \widetilde{\Wfp}\cap \widetilde{\Wfm}$ be a differentiable path such that $\widetilde{g_{0}}$ is the  $([c],\overline{v_{([c],q^{\fp}_{[c]})}})=([c],\overline{v_{([c],q^{-\fm}_{[c]})}})\in \partial\widetilde{ \Wfp}\cap \partial \widetilde{ \Wfm}$ and suppose $\widetilde{g_{s}}$ in turn descends to a curve $g_{s} \in \QF$ with $g_{0}=[c]$ under the projection $\widetilde{\QF} \rightarrow \QF$. As $\widetilde{g_{s}}\in\widetilde{\Wfp} \cap \widetilde{\Wfm}$ we have that $\pi(\mathfrak{F}(\widetilde{g_{s}}))=(\widetilde{a}(s)\fp,\widetilde{b}(s)\fm)$ which again descend to two smooth functions $1$-parameter functions $a(s),b(s)\in \mathbb{R}_{\geq 0}$ such that $\mathfrak{F}(g_{s})=(a(s)\fp,b(s)\fm)$ and $a(0)=b(0)=0$. Moreover by definition $g_{s}$ is normal to $\F$ at $g_{0}$ and along the direction ${v_{([c],q^{\fp}_{[c]})}}={v_{([c],q^{-\fm}_{[c]})}}$\\
	This brings us back to the case of Proposition \red{\ref{neccesary}} where we have a path starting from $\F$, normal to $\F$ and with specified first order behavior of the measured foliations at infinity given by a pair that fills $S$. Thus $g_{s}$ is a path of the type $\beta_{([c],q)}(t^{2})$ where $g_{0}=[c]=\mathsf{p}(a'(0)\fp,b'(0)\fm)$, the critical point for the function $\ext(a'(0)\fp)+\ext(b'(0)\fm)$. Again by assumption $[c]=\mathsf{p}(t\fp,\fm)$; so $\frac{a'(0)}{b'(0)}=t$, as $\mathsf{p}(\fp,\fm) $ is unique for a filling pair $(\fp,\fm)$ up to  scaling by $t$ (see Remark \red{\ref{scalediff}}). \\
	So we see that \begin{align*}
		\kappa \circ \widetilde{\pi}([c],\overline{{v_{([c],q^{\fp}_{[c]})}}})=\lim_{s \rightarrow 0}\kappa \circ \widetilde{\pi}(\tilde{g_{s}})=\lim_{s \rightarrow 0} \kappa \circ \pi(g_{s})=\lim_{s \rightarrow 0} \kappa \circ (a(s),b(s))=\lim_{s \rightarrow 0}(\frac{a(s)}{b(s)},b(s))=(t,0)
	\end{align*} \\
	This shows that if $v \in  T_{p}\widetilde{\Wfp} \cap T_{p}\widetilde{\Wfm}$ with $d_{p}\widetilde{\pi}(v)=0$ then $v$ is zero. Hence $d_{p}\widetilde{\pi}$ is injective at $\partial\widetilde{ \Wfp}\cap \partial\widetilde{ \Wfm}$.\\
	So the map $\widetilde{\pi}$ is a local immersion into $\widetilde{\mathbb{R}^{2}}$ at the points $p \in \partial\widetilde{\Wfp}\cap \partial\widetilde{\Wfm} $. Also $d_{p}\widetilde{\pi}$ is surjective at ${\partial\widetilde{\Wfp}\cap \partial\widetilde{\Wfm} }$ because the domain is a $2$ dimensional real manifold being the boundary of $\widetilde{ \Wf^{+}}\cap\widetilde{ \Wf^{-}}$ and the image is the boundary of the $2$ dimensional real manifold $\widetilde{\mathbb{R}^{2}}$ being $UT_{\left\lbrace 0\right\rbrace }\mathbb{R}^{2}$. \\
So we proved that $\widetilde{\pi}$ is a local diffemorphism in a neighbourhood of $\partial \widetilde{ \Wfp}\cap\partial{\widetilde \Wfm}$. \qed   }\\

We can now address the main proposition of this section which proves Theorem \ref{thm1.1}: 
{\proposition\label{pathexist}
	Let $(\fp,\fm)$ be a pair of arational measured foliations that fill $S$ and let $\mathsf{p}(\fp,\fm)\in \F$ be the critical point of the function $\ext(\fp)+\ext(\fm)$. Then for $ t\in [0,\epsilon) $ there exists a unique smooth curve $t \mapsto g_{t}\in \QF$, with $g_{0}=\mathsf{p}(\fp,\fm)$, such that the $\mathfrak{F}(g_{t})=(t\fp,t\fm)$ for all $t\in[0,\epsilon)$.   
}
{\proof By the preceding Proposition there exists a $ t\rightarrow\widetilde{g_{t}} \in \widetilde{\Wfp}\cap\widetilde{\Wfm}$ smooth curve in an open neighbourhood of $\partial\widetilde{ \Wfp}\cap\partial\widetilde{ \Wfm}$ such $b \circ \widetilde{\pi} \circ \mathfrak{F}(\widetilde{g_{t}})=(t,t)$ for $t\in (0,\epsilon)$ with $b$ being the blow-up map $b:\widetilde{\mathbb{R}^{2}} \rightarrow \mathbb{R}^{2}$ mapping $UT_{\left\lbrace 0\right\rbrace }(\mathbb{R}^{2})$ to the origin and identity on the rest. The result then follows as $\widetilde{g_{t}}$ descends to $g_{t} \in \QF$ with $\mathfrak{F}(g_{t})=(t\fp,t\fm)$ for $t\in[0,\epsilon)$. \qed   }\\ 
.

\section{Interpretation in Half-Pipe Geometry}\label{halfpipe}

	We will now give an interpretation of our result in quasi-Fuchsian {Half pipe} $3$ manifolds that we describe following\red{\cite{Danciger2013}}. To describe the space $\mathbb{HP}^{3}$, we will switch our viewpoint to the projective model for $\mathbb{H}^{3}$ in this section. Consider $\mathbb{RP}^{3}\subset \mathbb{R}^{4}$ with the group $PGL_{4}(\mathbb{R})$ being its isometry group.
	Consider now $\mathbb{H}^{3}$ as a subset of $\mathbb{RP}^{3}$. To be precise, consider $\mathbb{R}^{4}$ with the diagonal form given by the matrix 
	\begin{align*}
		\eta_{t}=\begin{bmatrix}
			-1&  &  & \\ 
			&  1&  & \\ 
			&  &  1& \\ 
			&  &  & t^{2}
		\end{bmatrix}
	\end{align*}	
	where $t \geq 0$. Each form $\eta_{t}$ define a convex region in $\mathbb{X}_{t}\subset \mathbb{RP}^{3}$ given by the relation 
	\begin{align*}
		x^{T}\eta_{t}x= -x_{1}^{2}+x_{2}^{2}+x_{3}^{2}+t^{2}x_{4}^{2}<0
	\end{align*}
	For each $t,$ $\mathbb{X}_{t}$ is a homogeneous subspace of $\mathbb{RP}^{3}$ which is preserved by the group $G_{t}$ of linear transformations that preserve $\eta_{t}$. With these notations, $\mathbb{H}^{3}=\mathbb{X}_{+1}$ and $G_{+1}=PO(3,1)\cong PSL_{2}(\mathbb{C})$. \\\\
	Moreover, define $\g_{t}: \mathbb{X}_{+1}\rightarrow \mathbb{X}_{t}$ as 
	\begin{align*}
	 \begin{bmatrix}
			1&  &  & \\ 
			&  1&  & \\ 
			&  &  1& \\ 
			&  &  & t^{-1}
		\end{bmatrix}
	\end{align*}
	and this gives an isomorphism between $\mathbb{X}_{+1}$ and $ \mathbb{X}_{t}$ . Moreover $\g_{t}$ conjugates $PO(3,1)$ to $G_{t}$. 
	Notice further the co-dimension $1$ space $\mathbb{P}^{3}$ defined by $x_{4}=0$ and $-x_{1}^{2}+x_{2}^{2}+x_{3}^{2}<0$ is a totally geodesic copy of $\mathbb{H}^{2}$ and is contained in $\mathbb{X}_{t}$ for all $t$ and $g_{t}$ fixes $\mathbb{P}^{3}$ pointwise. \\\hspace*{.5 cm}
	For $t>0$ we now consider a $1$-parameter family of quasi-Fuchsian structures on $M\cong \Sig \times \mathbb{R}$. So, we have a family of developing maps and holonomy representations given by:
	\begin{align*}
		\mathcal{D}_{t}: \widetilde{\Sig} \rightarrow \mathbb{H}^{3}\cong X_{+1} \\
		\rho_{t}: \pi_{1}(\Sig)\rightarrow PO(3,1) \cong PSL_{2}(\mathbb{C})
	\end{align*}  
	Assume further that for $t=0$, $\mathcal{D}_{0}$ gives us a submersion of $\widetilde{\Sig}$ onto $\mathbb{P}^{3}=\mathbb{H}^{2}$. That is the coordinate $x_{4}$ converges to a zero function. $\rho_{t}$ then converges to $\rho_{0}$ whose image lies in the subgroup $PO(2,1)\cong SL(2,\mathbb{R}). $ \\\hspace*{.5 cm}
		Apply now the rescaling map to obtain the developing map $g_{t}\mathcal{D}_{t}:\widetilde{\Sig}\rightarrow\mathbb{X}_{t}$, so that the holonomy representation is given by $g_{t}\rho_{t}g_{t}^{-1}$. Suppose that $t \rightarrow 0$ then $g_{t}\mathcal{D}_{t}$ converges to a local diffeomorphism $\mathcal{D}:\widetilde{\Sig}\rightarrow \mathbb{X}_{0}$ and if $\rho_{\mathcal{D}}:\pi_{1}(\Sig)\rightarrow PGL_{4}(\mathbb{R})$ is the limit of the holonomy $\rho_{t}$ as $t \rightarrow 0$ then $\mathcal{D}$ is equivariant with respect to $\rho_{\mathcal{D}}$ . To be precise, for $\gamma \in \pi_{1}(\Sig)$ if $\rho_{t}$ is of the form: 
	
	\begin{align}\rho_{t}=\begin{pmatrix}
			A(t) &w(t) \\ 
			v(t)& a(t)
	\end{pmatrix}\end{align}
	where $A \in PO(2,1)\cong PSL(2,\mathbb{R})$ and $w(t),v(t)^{T} \in \mathbb{R}^{3} $, then we have
	\begin{align}\label{halfpipething}
		\lim_{t \to 0}g_{t}\rho_{t}(\gamma)g_{t}^{-1}=\lim_{t \to 0} \begin{pmatrix}
			A(t) &tw(t) \\ 
			v(t)/t& a(t)
		\end{pmatrix}=\begin{pmatrix}
			A(0) &0 \\ 
			v'(0)& 1
		\end{pmatrix}=\rho_{\mathcal{D}}
	\end{align}
	So we have the following: 
	
	{\definition\label{quasifuchsianhalfpipe} A half-pipe structure on $\Sig\times \mathbb{R}$ is a $(G_{HP^{3}},\mathbb{HP}^{3})$ structure where $\mathbb{HP}^{3}=\mathbb{X}_{0}$ and $G_{HP}$ is the subgroup of $PGL_{4}(\mathbb{R})$ of matrices with the form
		$\begin{pmatrix}
			A &0 \\ 
			v& \pm 1
		\end{pmatrix}$
		where $A \in O(2,1)$ and $v^{T}\in \mathbb{R}^{3}$.}\\
	We also define:
	 {\definition\label{mcqhp} Any path $\rho_{t}$ of representations into $PSL_{2}(\mathbb{C})$ satisfying Equation (\ref{halfpipething}) is said to be compatible at first order at $t=0$ with $\rho_{\mathcal{D}}$}.
	 \\ As observed in\red{\cite{Danciger2013}}, $G_{HP^{3}}\cong \mathbb{R}^{2,1} \rtimes O(2,1) $, where an element of the form $\begin{pmatrix}
		A &0 \\ 
		v& \pm 1
	\end{pmatrix}$ can be interpreted as an infinitesimal deformations of the the hyperbolic structure given by $A\in PO(2,1)$ and along the direction $v$ normal to $PO(2,1)$ into $PO(3,1)$. Passing onto quotients, we see that quasi-Fuchsian half-pipe $3$-manifolds are precisely obtained by infinitesimal deformations in $\QF$ starting from the point $[c]\in \F$ along a direction $v_{([c],q)} \in N_{[c]}\mathcal{F}(\Sig)$. So we define:
			{\definition\label{hpqf} $M^{HP}_{c,q}$ is the half-pipe quasi-Fuchsian structures whose holonomy representation into $PGL_{4}(\mathbb{R})$ is compatible at first order at $t=0$ with the holonomy $\rho_{t}$ associated to quasi-Fuchsian metrics in $\beta_{([c],q)}(t)$ in the sense of Definition \red{\ref{mcqhp}}.}
				\subsection{Half-pipe Schwarzians and their measured  foliations}\label{hpschwarzandfoli}\hfill\break
		Recall again that $v_{([c],q)}$ is the tangent vector to the path $\beta_{([c],q)}(t)\in \QF$ associated to which we have unique minimal immersions of $\Sig$ for each $t<\epsilon$ with immersion data $I_{t}\in [c]$ and $\II_{t}=t\mathfrak{R}(q)$. So for each $t$ we have a $\mathcal{D}_{t}$ and $\rho_{t}$ in the sense above and a half-pipe structure as the limit when $t$ goes to $0$. There is also an analogous notion for half-pipe for the second fundamental form and shape operator in half-pipe geometry that follows from \cite{Fillastre2019}. So we want to study the limit of these immersion data as $t\rightarrow 0$ and use the following lemma: 
	{\lemma [\red{\cite{Fillastre2019}}]Let $\sigma_{t}$ be a $C^{2}$ family of minimal immersions of $\mathbb{H}^{2}$ into $\mathbb{H}^{3}$, such that $\sigma_{0}$ is an embedding of $\mathbb{H}^{2}$. Let
		$\sigma = \lim_{t \to 0}G_{t}\circ\sigma_{t}$
		be the rescaled immersion in $\mathbb{HP}^{3}$. Then:
		\begin{itemize}
			\item  The first fundamental form of $\sigma$ coincides with the first fundamental form of $\sigma_{0}$:
			\begin{align*}
				I(v,w)=\lim_{t \to 0}I_{t}(v,w)
			\end{align*}
			\item The second fundamental form of $\sigma$ is the first derivative of the second fundamental form of $\sigma_{t}$:
			\begin{align*}
				\II(v,w)=\lim_{t \to 0}\frac{\II_{t}(v,w)}{t}
			\end{align*}
			\item The shape operator $B$ of $\sigma$ is the first derivative of the shape operator $B_{t}$ of $\sigma_{t}$:
			\begin{align*}
				B(v)= \lim_{t \to 0}\frac{B_{t}}{t}
			\end{align*}
	\end{itemize} }\hfill\break
	 We immediately have the following for $M^{HP}_{c,q}$. 
	
	{\proposition\label{8,7}  The half-pipe manifold $M^{HP}_{c,q}$ contains a smooth minimal surface with immersion data uniquely given by $I\in[c]$ and $\II=\mathfrak{R}(q)$.  }
	{\proof Since $M^{HP}_{c,q}$ appears as a limit of the quasi-Fuchsian structures defines by the path $\beta_{([c],q)}(t)$ we apply the lemma directly. The induced metric of the minimal immersion of $\Sig$ into $M^{HP}_{c,q}$ is given by $I= \lim_{t \to 0}I_{t}\in [c]$. So we see that $\II=\lim_{t \to 0}\frac{\II_{t}}{t}=\mathfrak{R}(q)$.}\\\\
		So we will introduce an analogous notion for Schwarzian at infinity for half-pipe manifolds that is quite natural with the tools we have developed so far and with our definition of $M^{HP}_{cq}$.
	
	{\definition\label{definitionofhpsch}  The positive (resp. negative) half pipe Schwarzian at infinity associated to $M^{HP}_{cq}$ is defined as the derivative at $\F$ of the Schwarzian derivatives at the positive (resp. negative) end at infinity for quasi-Fuchsian metrics in the path $\beta_{([c],q)}(t)$ for $t$ small enough.}\\\\\
	From Lemma \red{\ref{firstorderschwarz}} So we have 
	{\proposition The positive and negative half-pipe Schwarzians at infinity for $M^{HP}_{c,q}$ are $q$ and $-q$.  }\\\\
		We can now again consider the horizontal measured foliation $\pm \f$ associated to $\pm q$ on $[c]$ and obtain our Theorem \ref{thm1.3} by an application of Theorem \red{\ref{GM}}:
	
	{\theorem\label{hpschth} Any pair $(\fp,\fm)\in \FMF$ can be uniquely realised as the horizontal foliations of the positive and negative half pipe Schwarzians at infinity associated to quasi-Fuchsian half-pipe manifold. Moreover, $M^{HP}_{cq}$ defined before is the unique one realising $(\fp,\fm)$, where $([c],q)\in T^{*}\T$ is the unique point realising $(\fp,\fm)$. }\\\\
	This can be seen as a {first-order interpretation of Theorem of Gardiner-Masur}  as half-pipe quasi-Fuchsian manifolds correspond to points in $T^{*}\T$ by Proposition \red{\ref{8,7}} via the minimal surface they contain. \\

\chapter{Quasi-Fuchsian manifolds close to the Fuchsian locus are foliated by CMC surfaces:}
One of the main idea of this work is to consider Epstein surfaces inside quasi-Fuchsian manifolds with respect to the data at the boundary at infinity. First we set the convention of the normal to an immersed surface in a quasi-Fuchsian manifold.
{\remark{\label{signconv}} When $S$ is an embedded surface in $\mathbb{H}^{3}/\Gamma$ homotopic to$S\times \left\lbrace \star \right\rbrace $
	we will refer to the (unit) normal vector to $S$ as the one chosen according to the following
	convention. We lift $S$ to a surface $\widetilde{S}$ in the universal cover $\mathbb{H}^{3}$
	, whose asymptotic boundary
	is the limit set $\Lambda_{\Gamma}$. Then $\widetilde{S}$ disconnects $\mathbb{H}^{3}$
	in two components. We declare that the unit
	normal vector to $S$ lifts to the unit normal vector to $\widetilde{S}$ pointing towards the component
	whose closure contains $\Omega_{+}$, the positive end of the domain of discontinuity.} \section{Epstein surfaces}
In this subsection we describe a construction due to Epstein in \cite{epstein}, which naturally associates to certain conformal metrics on a domain of $\mathbb{C}P^{1}\cong \partial_{\infty}\mathbb{H}^{3}$ an immersion into $\mathbb{H}^{3}$, that we will call the Epstein surface. 

\subsection{The Epstein map} Given any point $p\in \mathbb{H}^{3}$, we define a map  $G_p:T^1_p\mathbb{H}^3\to\mathbb{C}P^{1}$, by sending $(x,v)$ to the endpoint at infinity of the unique geodesic of $\Hyp^3$ starting at $x$ with tangent vector $v$. Then we define the \emph{visual metric} $V_p$ as the metric obtained by pushforward via $G_p$ of the canonical spherical metric of $T^1_p\Hyp^3$. One can easily check that the metric $V_p$ is conformal, namely compatible with the Riemann surface structure of $\mathbb{C}P^{1}$. Indeed, if $o$ is the origin in the unit ball model, then $V_o$ is just the usual spherical metric on the unit sphere. For the general case, if $M$ is an isometry of $\Hyp^3$ sending $o$ to $p$, then $V_p=M_*V_o$ and is therefore in the same conformal class, since $M$ extends to a biholomorphism of $\mathbb{C}P^{1}$.

The fundamental result is the following:

{\proposition[\cite{epstein,Dumas2017}]\label{prop:eps map}
	Let $\Omega$ be a connected open domain in $\mathbb{C}P^{1}$ and let $\varphi:\Omega\to \mathbb{C}P^{1}$ be a locally injective holomorphic map. If $\sigma$ is a $C^1$ conformal metric on $\Omega$, then there exists a unique continuous map $\Eps_{(\varphi,\sigma)}:\Omega\to\Hyp^3$ such that
	$$(\varphi^*V_{\Eps_{(\varphi,\sigma)}(z)})(z)=\sigma(z)$$
	for all $z\in \Omega$.  Moreover, if $\sigma$ is $C^k$, then $\Eps_{(\varphi,\sigma)}$ is $C^{k-1}$.
}

%we can define a visual metric on $z\in \mathbb{C}P^{1}$ in the following way: Chose the unique geodesic in $\mathbb{H}^{3}$ starting from $x$ and limiting to $z$ at the boundary. There is unique geodesic copy of $\mathbb{H}^{2}$ orthogonal to this geodesic at the point $x$ that bounds a disc $D$ at $\mathbb{C}P^{1}$ containing $z$. The visual metric at infinity $V_{x}(z)$ is then defined to be the $g_{D}(z)$ where $g_{D}$ is the hyperbolic metric on $D$. For example, if one considers $x$ to be the origin, then the visual metric $V_{0}(z)$ for $z\in \mathbb{C}P^{1}$ is nothing but the usual spherical metric. Moreover, for any other point $p\in \mathbb{H}^{3}$, $V_{p}$ can be defined to $M_{*}V_{0}$ where $M$ is the unique M{\"o}bius map joining $0$ to $x$. We can now define :

%{\defi Given a conformal metric $\sigma$ on a domain $\Omega \subset \mathbb{C}P^{1}$, there exists an unique map $Ep_{\sigma}:\Omega \rightarrow \mathbb{H}^{3}$ such that for all $z\in \Omega$ we have that $V_{Ep_{\sigma}}(z)= \sigma(z)$. The image $Ep_{\sigma}(\Omega)$ is defined to be the Epstein surface for the conformal metric $\sigma$. } \\\\

We remark that $\Eps_{(\varphi,\sigma)}$ is in general not an immersion. As an example, if $\sigma$ is the standard spherical metric on the unit sphere, then the associated Epstein map is constantly equal to the origin $o$ in the unit ball model.

In \cite[Section 3]{Dumas2017} Dumas introduced an explicit formula for $\Eps_{(\varphi,\sigma)}$ in the upper half-space model of $\Hyp^3$, which will be useful for our purposes. Let $p$ be the point in the geodesic joining $0$ and $\infty$ in the upper half-space model such that the visual metric $V_p$ at $0$ equals $|dz|^2$. Concretely, $p=(0,0,2)$. If we write the conformal metric as $\sigma=e^{2\eta}|dz|^2$, and  to simplify the notation we let $\Omega$ be a connected open subset of $\C$ so as to take $\varphi=\mathrm{id}$, then the expression for $\Eps_{(\mathrm{id},\sigma)}:D\to\Hyp^3$ is the following:

\begin{equation}\label{eq:SLframe}
	\Eps_{(\mathrm{id},\sigma)}(z)=\begin{pmatrix}
		1 & z \\
		0& 1 \\
	\end{pmatrix}\begin{pmatrix}
		1 & 0 \\
		\eta_{z}& 1 \\
	\end{pmatrix}\begin{pmatrix}
		e^{-\frac{\eta}{2}} & 0 \\
		0& e^{\frac{\eta}{2}} \\
	\end{pmatrix}\cdot p \end{equation}

\subsection{Schwarzian tensors}
The last fundamental preliminary step that we will need in our paper is an expression for the mean curvature of Epstein maps. For this purpose, we first need to introduce the notion of Schwarzian tensor, due to Osgood and Stowe \cite{zbMATH00078553}. Given two conformal metrics $\sigma_{1}=e^{2\eta_{1}}|dz|^{2}$ and $\sigma_{2}=e^{2\eta_{2}}|dz|^{2}$ on a domain $\Omega \subset \mathbb{C}P^{1}$,  the \emph{Schwarzian tensor} of $\sigma_{1}$ with respect to $\sigma_{2}$  is the  quadratic differential (which is not necessarily holomorphic, in general) defined as  
\begin{equation}\label{eq:schw tensor}
	B(\sigma_{1},\sigma_{2})= ((\eta_{2})_{zz}-{(\eta_{2})_{z}}^{2}- (\eta_{1})_{zz}+{(\eta_{1})_{z}}^{2})dz^{2}
\end{equation}
This definition generalizes the classical Schwarzian derivative, in the sense that, if $f:\Omega\to \C$ is a locally injective holomorphic function, then
\begin{equation}\label{eq:schw tensor and schwarzian}
	S(f) =2B(|dz|^{2},f^{*}|dz|^{2})~.
\end{equation}  
Clearly $B(\sigma_{2},\sigma_{1})=-B(\sigma_{1},\sigma_{2})$. Similarly to the Schwarzian derivative, the Schwarzian tensor has a number of naturality  properties. For any metrics $\sigma_{1},\sigma_{2},\sigma_{3}$ on $\Omega\subset \mathbb{C}P^{1}$, 
\begin{itemize}
	\item Given a locally injective holomorphic map $f$, 
	\begin{equation} \label{eq:swt1} f^{*}B(\sigma_{1},\sigma_{2})=B(f^{*}\sigma_{1},f^{*}\sigma_{2})~.\end{equation}
	\item The cocycle property holds: 
	\begin{equation} \label{eq:swt2}B(\sigma_{1},\sigma_{3})=B(\sigma_{1},\sigma_{2})+B(\sigma_{2},\sigma_{3})~.\end{equation}
\end{itemize}
In particular, \eqref{eq:swt1} implies that if $\sigma_1$ and $\sigma_2$ are invariant by an automorphism of $\Omega$, then so is the quadratic differential $B(\sigma_{1},\sigma_{2})$. If a group $\Gamma$ acts on $\Omega$ by biholomorphisms with $\Omega/\Gamma\cong\Sigma$, thus inducing in quotient surface $\Sigma$ a Riemann surface structure, and $\sigma_1,\sigma_2$ are $\Gamma$-invariant conformal metrics, then $B(\sigma_{1},\sigma_{2})$ induces a well-defined quadratic differential in the quotient.

\subsubsection{M{\"o}bius flat metrics}\label{subsec:mob flat}
A conformal metric $\sigma$ is said to be \emph{M{\"o}bius flat} if $B(\sigma,|dz|^{2})=0$. From \eqref{eq:schw tensor and schwarzian}, for example, when $f$ is itself a M{\"o}bius transformation, then the pull-back metric $f^{*}|dz|^{2}$ is always M{\"o}bius flat. This is not the only case. Indeed, one can show that 
$B(\sigma,|dz|^2)=0$ if and only if $\sigma$ is the pull-back by a M\"obius transformation of one of the following metrics:
\begin{itemize}
	\item the flat metric $|dz|^2$ on $\C$,
	\item a positive multiple of the Poincar\'e metric on $\D$,
	\item a positive multiple of the spherical metric on $\CP^{1}$.
\end{itemize}

Now, given a metric $\sigma$, we will denote by 
$$B(\sigma)=B(g_{\CP^1},\sigma)$$ the Schwarzian tensor of $\sigma$ with respect to a M{\"o}bius flat metric $g_{\CP^{1}}$. By the definition of  M{\"o}bius flat and the cocycle property \eqref{eq:swt2}, $B(\sigma)$ is independent of the chosen  M{\"o}bius flat metric $g_{\CP^{1}}$. Hence if $f$ is a M\"obius transformation, then 
\begin{equation}\label{eq invariance B}
	B(f^*\sigma)=f^*B(\sigma)
\end{equation}
by \eqref{eq:swt1}. As another consequence of the independence of the definition of $B(\sigma)$ from the choice of $g_{\CP^{1}}$, together with the definition \eqref{eq:schw tensor} applied to $B(e^{2t}\sigma)=B(|dz|^2,e^{2t}\sigma)$, we have that if $e^{2t}$ is any positive constant then 
\begin{equation}\label{eq:scale invariance}
	B(e^{2t}\sigma)=B(\sigma)~.
\end{equation}

Finally, given a quadratic differential $\phi=\lambda(z)dz^2$ and a conformal metric $\sigma=e^{2\eta}|dz|^2$, we define the norm of $\phi$ with respect to $\sigma$ as:
$$\|\phi\|_\sigma(z):=e^{-2\eta(z)}|\lambda(z)|~.$$
Since both $|\phi|$ and $\sigma$ follow the same transformation rule under a biholomorphic change of coordinates, $\|\phi\|_\sigma$ is as well-defined function, meaning that if $f$ is a locally injective holomorphic function, then 
\begin{equation}\label{eq invariance norm}
	\|f^*\phi\|_{f^*\sigma}=\|\phi\|_\sigma\circ f~.
\end{equation}
In particular, if $\sigma=e^{2u}h_0$ is a conformal metric on  $(\Sigma,h)$ and $\phi$ is a quadratic differential on  $(\Sigma,h)$, then $\|\phi\|_\sigma$ is a function on $\Sigma$. From \eqref{eq:scale invariance}, we also obtain:
\begin{equation}\label{eq:scale invariance2}
	\|\phi\|_{e^{2t}\sigma}=e^{-2t}\|\phi\|_\sigma~,
\end{equation}
for any constant $t\in\R$.

\subsection{Mean curvature}\label{subsec:mean}
We are now ready to provide the formula for the mean curvature of  Epstein maps. Let $\sigma$ be a $C^2$ conformal metric on an open set $\Omega$. To simplify the notation, we first suppose $\varphi=\mathrm{id}$. Assume moreover that $\Eps_{(\mathrm{id},\sigma)}$ is an immersion. In this case, it turns out that $\Eps_{(\mathrm{id},\sigma)}$ at $z$ is tangent to the unique horosphere through $\Eps_{(\mathrm{id},\sigma)}$ with point at infinity $z$. Then, the mean curvature of $\Eps_{(\mathrm{id},\sigma)}$ equals the function

\begin{equation}\label{meancurveq}
	\mathcal H(\Eps_{(\mathrm{id},\sigma)}) = \frac{K(\sigma)^{2}-1-16\|B(\sigma)\|_\sigma^2}{(K(\sigma)-1)^{2}-16\|B(\sigma)\|_\sigma^2}~,
\end{equation}
where $K(\sigma)$ denotes the curvature of $\sigma$. See \cite[Equations 3.2, 3.3]{Dumas} and \cite[Lemma 3.4]{quinn}.
Here the mean curvature is defined as one half the trace of the second fundamental form with respect to the first fundamental form. It is computed with respect to the unit normal vector pointing towards $\Omega$. We will then apply the formula \eqref{meancurveq} when the Epstein map induces an embedded surface in $\Hyp^3/\Gamma$ for $\Gamma$ a quasi-Fuchsian group, and for $\Omega=\Omega^+$. Hence the convention of the mean curvature here is consistent with Remark \ref{signconv}.

To write the general formula for $\Eps_{(\varphi,\sigma)}$, since the computation is local, we may restrict to an open subset $\Omega$ on which $\varphi$ is a biholomorphism onto its image. Let $\sigma$ be a metric on $\Omega$ and $\hat\sigma$ be such that $\varphi^*\hat\sigma=\sigma$. Then we observe that $K(\hat\sigma)\circ \varphi=K(\sigma)$, whereas by \eqref{eq:schw tensor and schwarzian}, \eqref{eq:swt1} and \eqref{eq:swt2},
$$\varphi^*B(\hat\sigma)=B(\varphi^*g_{\CP^1},\sigma)=B(\varphi^*g_{\CP^1},g_{\CP^1})+B(g_{\CP^1},\sigma)=B(\sigma)-\frac{1}{2}S(\varphi)~.$$
Hence we can deduce the expression:
\begin{equation}\label{meancurveq2}
	\mathcal H(\Eps_{(\varphi,\sigma)}) = \frac{K(\sigma)^{2}-1-16\|B(\sigma)-S(\varphi)/2\|^2_\sigma}{(K(\sigma)-1)^{2}-16\|B(\sigma)-S(\varphi)/2\|^2_\sigma}
\end{equation}

\subsection{A technical point}
The rough idea to prove the existence of CMC surfaces using the implicit function theorem is the following. Consider quasi-Fuchsian manifolds $\Hyp^3/\Gamma$, where $\Omega^\pm$ are the connected components of the complement of the limit set $\Lambda_\Gamma$. We would like to write the solutions of the CMC condition $\mathcal H=c$, for $c\in (-1,1)$, as the level sets of a function $G$ which depends on the hyperbolic metric $h$ on $S$ in the conformal class of $\Omega^+/\Gamma$ (that is, it represents the first Bers parameter $h^+$ of $M$), on a holomorphic quadratic differential $\phi$ on $(S,h)$ which is (the quotient of) the Schwarzian derivative of the conformal isomorphism between $\D$ and $\Omega_+$, and finally on the conformal factor of a metric of the form $e^{2u}h$ on $S$. This last function $u$ is  an element of the infinite-dimensional functional space $C^\infty(S,\R)$. A priori the pair $(h,\phi)$ varies in an infinite-dimensional space as well, since $h$ varies in the space $\mathcal M_{-1}(S)$ of hyperbolic metrics. Although this is not really necessary, it will be convenient to use the action of $\mathrm{Diff}_0(S)$  to reduce ourselves to representatives of pairs $(h,\phi)$, now varying in the finite-dimensional space $\Q(S)$. The following lemma will serve to formalize this approach.

\begin{lemma}\label{lem:section}
	Let $\pi:\mathcal M_{-1}(S)\to\T$ be the quotient map by the action of $\mathrm{Diff}_0(S)$ on $\mathcal M_{-1}(S)$. There exists a smooth section $s:\T\to \mathcal M_{-1}(S)$ of $\pi$.
\end{lemma}

We remark that the section $s$ that we are looking for is not ``canonical" in any manner. There are actually several ways to achieve this; we will sketch one relying on the theory of harmonic maps of hyperbolic surfaces, see \cite{Wolf1989}.

\begin{proof}[Sketch of proof of Lemma \ref{lem:section}]
	Fix a hyperbolic metric $h_0$ on $S$, and consider the vector space $H^0((S,h_0),K^2)$ of holomorphic quadratic differentials on $(S,h_0)$. Then for every $q\in H^0((S,h_0),K^2)$ there exists a unique hyperbolic metric $h_q$ such that $\mathrm{id}:(S,h_0)\to (S,h_q)$ is harmonic, with $h_q$ depending smoothly on $q$. The correspondence $q\mapsto h_q$ therefore gives a map $H^0((S,h_0),K^2) \rightarrow \mathcal M_{-1}(S)$ that, when post-composed with $\pi$, provides a homeomorphism from $H^0((S,h_0),K^2)$ to $\T$. This proves the existence of the desired section.
\end{proof}

\begin{remark}\label{rmk:wolf}
	Wolf's approach via harmonic maps actually led to the construction of a global parameterization of $\T$ by means of the space $H^0((S,h_0),K^2)$, once the metric $h_0$ is fixed. This allows us to identify the space $\Q(S)$ with a very concrete finite-dimensional manifold of real dimension $12g-12$, namely the total space of the smooth vector bundle $\mathcal E$ over $H^0((S,h_0),K^2)$ whose fiber over a quadratic differential $q$ is equal to $H^0((S,h_q),K^2)$, the space of holomorphic quadratic differentials of the hyperbolic surface $(S,h_q)$. In rest of our exposition we will identify with abuse any pair $(h,\phi)$ with its corresponding point in the total space of $\mathcal E$. (Notice that the identification with $\Q(S)$ heavily depends on the choice of the section $s$ from Lemma \ref{lem:section}.)
\end{remark}

	\section{Existence of CMC surfaces}
	
	The purpose of this section is to prove two existence results for CMC surfaces, morally one (Theorem \ref{thm:foliation_ends}) ``in the ends" and the other (Theorem \ref{thm:existence_compact}) ``in the compact part". Then in Theorem \ref{thm:existence} we combine them to obtain the existence of CMC surfaces for $h\in(-1,1)$ for quasi-Fuchsian manifolds close to the Fuchsian locus, which is for the moment  weaker than our main result, Theorem \ref{thm:foliation}.
	
	\subsection{Existence in the ends}\label{sec:foliation_ends}
	
	It has been proved in \cite{MP} that the ends of every quasi-Fuchsian manifold are monotonically foliated by CMC surfaces; another proof has been provided recently in \cite{quinn}. Here we will need an improved statement, so as to have a local (in $\QF$) uniform control on the value of the mean curvature along the leaves of the foliation. 
	
	\begin{theorem}\label{thm:foliation_ends}
		Let $S$ be a closed oriented surface of genus $\geq 2$ and $m\in\QF$. Then there exists a neighbourhood $U_0$ of $m$ in $\QF$ and a constant $\epsilon=\epsilon(m,U_0)$ such that the ends of every quasi-Fuchsian manifold in $U_0$ are smoothly monotonically foliated by CMC surfaces whose mean curvature ranges in $(-1,-1+\epsilon)$ and in $(1-\epsilon,1)$.
	\end{theorem}
	
	We say that the ends of $M\cong S\times\R$ are the connected components of the complement of a compact submanifold with boundary in $M$ homeomorphic to $S\times I$ for $I$ a closed interval.

	\subsubsection{Outline of the CMC existence for a fixed manifold}\label{subsec:outline1}
	We now quickly review, using our notation and set-up, the proof given in \cite{quinn} and later we will explain how it adapts in order to prove Theorem \ref{thm:foliation_ends}. Roughly speaking, the proof of \cite{quinn} is  an application of the implicit function theorem to the equation of constant mean curvature from the mean curvature formula \eqref{meancurveq}, with respect to a conformal metric at infinity.

	More precisely, the idea of Quinn's  proof is to consider Epstein maps defined on $\Omega_+$, with $\varphi=\mathrm{id}$, associated to  a conformal metric of the form $\sigma(u)=e^{2u}h_0$ for $h_0$ the conformal complete hyperbolic metric, and to study the following equation  in $u$:
	$$\mathcal H(\Eps_{(\mathrm{id},\sigma(u))})=H$$
	for $H\in (-1,1)$ close to $\pm 1$. From \eqref{meancurveq}, this gives  the equation: 
	\begin{equation}\label{eq:quinn}
		\mathcal H(\Eps_{(\mathrm{id},\sigma(u))}) = \frac{K(\sigma(u))^{2}-1-16\|B(\sigma(u))\|^{2}_{\sigma(u)}}{(K(\sigma(u))-1)^{2}-16\|B(\sigma(u))\|^{2}_{\sigma(u)}}= H
	\end{equation}
	
	\begin{remark}\label{rmk:quinn1}
		If we choose a  metric $\sigma$ invariant under the quasi-Fuchsian group $\Gamma$ acting on $\Omega_+$ by biholomorphisms, then $ \|B(\sigma(u))\|^{2}_{\sigma(u)}$ is a well-defined invariant function on the quotient $S$, by \eqref{eq invariance B} and \eqref{eq invariance norm}. This shows that the equation \eqref{eq:quinn} can be really thought as an equation for a function $u$ on the quotient surface $S$, where $\sigma(u)=e^{2u}h_0$ is a metric on $S$.
	\end{remark}
	
	\begin{remark}\label{rmk:quinn2}
		In the situation of Remark \ref{rmk:quinn1}, the uniqueness property of the Epstein map as in Proposition \ref{prop:eps map} implies that the Epstein surface is invariant under the quasi-Fuchsian group $\Gamma$. More precisely, for any $\gamma\in \Gamma$, we have
		\begin{equation}\label{eq:invariance}
			\Eps_{(\mathrm{id},\sigma(u))}\circ\gamma=\gamma\circ \Eps_{(\mathrm{id},\sigma(u))}~.
		\end{equation}
		Therefore $\Eps_{(\mathrm{id},\sigma(u))}$ induces a map from $\Omega^+/\Gamma$ to the quasi-Fuchsian manifold $\Hyp^3/\Gamma$.
	\end{remark}
	
	Now the trick consists in performing a renormalization to Equation \eqref{eq:quinn}, so as to obtain an equivalent equation, for which we can find an explicit solution for $H=-1$. This consists in the change of variables from $(H,u)$ to $(H,v)$, where
	\begin{equation}\label{eq:change variables}
		v:=u+\frac{1}{2}\log\left(\frac{1+H}{1-H}\right)~.
	\end{equation}
	Let us now set $\tau(v)=e^{2v}h_0$, so that we have the identity:
	\begin{equation}\label{eq:change variables2}
		\tau(v)=\frac{1+H}{1-H}\sigma(u)~.
	\end{equation}
	A direct computation from \eqref{eq:quinn} and \eqref{eq:change variables2} (and using also \eqref{eq:scale invariance2}) shows that $u$ solves \eqref{eq:quinn} for $H\in (-1,1)$ if and only if $v$ solves the equation:
	
	\begin{equation}\label{eq:quinn2}
		G(H,v):= 1-H -2HK(\tau(v) ) + (-1-H)\left(K(\tau(v))^{2}- 16\|B(\tau(v))\|^{2}_{\tau(v)}\right)=0~.
	\end{equation}
	
	The big advantage is that now the choice $v_0\equiv 0$ satisfies $G(-1,v_0)=0$, since $\tau(v_0)=h_0$  and $K(h_0)=-1$.
	Hence we are in the right setting to apply the implicit function theorem near this solution $(-1,v_0)$ of the equation $G=0$ (see e.g. \cite[\S I.5]{lang}). One must show that the derivative of $G$ with respect of $u$ is an invertible operator between suitable function spaces (see details below), and achieves a family of solutions $\mathsf v=\mathsf v(H)$ of \eqref{eq:quinn2} depending smoothly on $H$, for $H\in [-1,-1+\epsilon)$. This will provide CMC surfaces with mean curvature $H$ close to $-1$ via the Epstein maps $\Eps_{(\mathrm{id},\sigma(\mathsf u(H)))}$, where
	$$\mathsf u(H)=\mathsf v(H)-\frac{1}{2}\log\left(\frac{1+H}{1-H}\right)~.$$

	\subsubsection{Adaptation for Theorem \ref{thm:foliation_ends}}\label{subsec:adapt}
	
	We will now describe the extension of this strategy in our setting. The difference is that we need to allow the quasi-Fuchsian manifold to vary as well, represented by a variation of a pair $(h,\phi)$, and thus of the holomorphic map $f=f_{\tilde\phi}$ which gives a biholomorphism between $\D$ and the domain $\Omega^+$. Let us explain this in detail.
	
	To make explicit the dependence on the hyperbolic metric $h$, we now denote $\sigma_h(u):=e^{2u}h$. We need to replace Equation \eqref{eq:quinn} by the condition that the mean curvature of the Epstein map $\Eps_{(f_{\tilde\phi},\sigma_h(u))}$ equals $H$. From Equation \eqref{meancurveq2}, we see that such identity reads:
	
	\begin{equation}\label{eq:us}
		\frac{K(\sigma_h(u))^{2}-1-16\|B(\sigma_h(u))-\phi/2\|^{2}_{\sigma_h(u)}}{(K(\sigma_h(u))-1)^{2}-16\|B(\sigma_h(u))-\phi/2\|^{2}_{\sigma_h(u)}}= H
	\end{equation}
	where we have used that the holomorphic quadratic differential induced in the quotient by $S(f_{\tilde\phi})$ equals $\phi$ by construction. This is again an equation on the closed surface $S$, and the same change of variables as in \eqref{eq:change variables} leads to the equation:
	
	\begin{equation}\label{eq:us2}
		G(H,h,\phi,v):= 1-H -2HK(\tau_h(v) ) + (-1-H)\left(K(\tau_h(v))^{2}- 16\|B(\tau_h(v))-\phi/2\|^{2}_{\tau_h(v)}\right)=0~,
	\end{equation}
	where now $\tau_h(v)=e^{2v}h$. 
	
	Now, fix a hyperbolic metric $h_0$ on $S$ and a holomorphic quadratic differential $\phi_0$ on $(S,h)$. Similarly to Section \ref{subsec:outline1}, a solution to Equation \eqref{eq:us2} is given by $(-1,h_0,\phi_0,v_0)$ where $v_0$ denotes the constant null function, since $\tau_{h_0}(v_0)=h_0$ has curvature $-1$. 
	To apply the implicit function theorem, let us describe carefully the domain of definition of $G$. Recall from Remark \ref{rmk:wolf} that the choice of a section as in Lemma \ref{lem:section} provides us with a diffeomorphism between $\Q(S)$ and  $\R^{12g-12}$. We consider thus the open subset $W$ of $\R^{12g-12}$ that corresponds to the image of $\QF$ under the map $\mathcal S$ introduced in  \eqref{eq:schwarzian map}. By a small abuse of notation, we will denote the elements of $W$ as a pair $(h,\phi)$, where $h$ is a hyperbolic metric and $\phi$ a holomorphic quadratic differential on $(S,h)$. Then we consider $G$ as a map
	$$G:\R\times W\times W^{2,s}(S, h_0)\to W^{2,s - 2}(S, h_0)$$
	for $s\geq 2$, where $W^{2,s}(S, h_0)$ denotes the Sobolev space of real-valued functions on $S$ that admit $L^2$-integrable weak derivatives of order $\leq s$ (with respect to the standard Riemannian measure of $h_0$), and $W^{2,0}(S,h_0) : = L^2(S, h_0)$. By direct inspection, $G$ depends smoothly on all variables. 
	We now need to show that the derivative $d_{v}G_{(-1,h_0,\phi_0,v_0)}$ is a bounded invertible operator, for any $s\geq 2$. A simple computation gives: 
	
	\begin{equation}\label{eq:differential}
		\begin{split}
			d_{v}G_{(-1,h_0,\phi_0,v_0)}(\dot v) &= 2 \left.\frac{d}{dv}\right|_{v=v_0} (K(e^{2v}h)) \\
			& = 2 \left.\frac{d}{dv}\right|_{v=v_0}(e^{-2v}(-\Delta_{h_0} v+K(h_0))  \\
			&= 2(2\dot v-\Delta_{h_0}\dot v)
		\end{split}
	\end{equation} 
	It is well-known that such an operator is a continuous linear isomorphism; we provide here a sketch of proof for convenience of the reader.
	
	\begin{lemma}\label{lem filippo}
		Let $f$ be a smooth and strictly positive function, and let $h$ be any Riemannian metric on a compact surface $S$. Then the operator $u \mapsto f u - \Delta_h u$ is a positive definite and continuous linear isomorphism from $W^{2,s}(S, h)$ to $W^{2,s-2}(S, h)$ for any $s\geq 2$. In particular, for any smooth function $\lambda$ on $S$, there exists a unique smooth function $u$ satisfying $\Delta_h u - f u = \lambda$.
	\end{lemma}
	\begin{proof}
		Let $T$ denote the continuous linear operator
		\[
		T : = f \, \mathit{id} - \Delta_h : W^{2,s}(S, h) \to W^{2, s - 2}(S, h) ,
		\]
		for some $s \geq 2$. A simple integration by parts shows that $T$ is a positive definite symmetric operator with respect to the $L^2$-scalar product: indeed, for any $v, w \in W^{2,s}(S, h)$, we have
		\[
		\langle v, T w \rangle_{L^2} = \int_{S} v \, T w \, \mathrm{d}{a}_h = \int_{S} (f v w + h(\nabla v, \nabla w) ) \, \mathrm{d}{a}_h ,
		\]
		where $\nabla v$ denotes the (weak) gradient of $v$ with respect to the metric $h$, and $\mathrm{d}{a}_h$ is the standard Riemannian volume form. Since $f$ is a strictly positive function, $T$ satisfies $\langle v , T v \rangle_{L^2} \geq 0$ for any $v \in W^{2,s}(S, h)$, with equality if and only if $v = 0$. To prove that $T$ is surjective, let $\lambda \in W^{2,s-2}(S, h)$, and define the linear functional
		\[
		\varphi(v) : = \int_{S v\lambda \, \mathrm{d}}{a}_h .
		\]
		Notice that $\varphi$ is continuous with respect to the $L^2$-norm, and hence with respect to the Sobolev norm $\| \cdot \|_{W^{2,s}}$ for any $s \geq 0$. We now introduce the following bilinear symmetric form on $W^{2,1}(S,h)$:
		\[
		a(v,w) : = \int_{S} (f vw + h(\nabla v, \nabla w)) \, \mathrm{d}{a}_h ,
		\]
		If $C \geq 1$ is some positive constant satisfying $C^{-1} \leq f \leq C$, then we have
		\[
		C^{-1} \| v \|^2_{W^{2,1}} \leq a(v,v) \leq C \| v \|^2_{W^{2,1}} .
		\]
		for any $v \in W^{2,1}(S,h)$. Therefore the bilinear form $a$ is equivalent to the standard Hilbert scalar product of the Sobolev space $W^{2,1}(S,h)$, and therefore $\varphi$ is continuous with respect to $a$ as well. By Riesz representation theorem, we conclude that there exists a unique $u \in W^{2,1}(S,h)$ satisfying $a(u,v) = \varphi(v)$ for any $v \in W^{2,1}(S,h)$. This proves the existence of a weak solution $u \in W^{2,1}(S,h)$ of the equation $f u - \Delta_h u = \lambda$. 
		
		A more delicate analysis is then required to show that the regularity of $\lambda \in W^{2,s - 2}(S, h)$  is sufficient to "promote" $u$ to a genuine element in $W^{2,s}(S, h)$ satisfying $T u = \lambda$. This is the part of the argument where elliptic regularity theory is required, leading to controls of the form
		\[
		\| u \|_{W^{2,s}} \leq M ( \| u \|_{L^2} + \| \lambda \|_{W^{2,s - 2}} ) ,
		\]
		with the multiplicative constant $M > 0$ that depends only on $s \geq 2$, the function $f$, and the compact Riemannian surface $(S, h)$. We refer to \cite[\S 10.3.2]{nicolaescu} (see in particular \cite[Theorem 10.3.12]{nicolaescu}) for a detailed exposition of elliptic regularity results on smooth manifolds.
	\end{proof}
	
	We have thus shown that $d_{v}G:W^{2,s}(S,h_0)\to W^{2,s-2}(S,h_0)$ is a linear isomorphism at the point $(-1,h,\phi,v_0)$. We can now apply the implicit function theorem for Banach spaces, and deduce that there exist $\epsilon>0$, a neighbourhood $U_0$ of $(h_0,\phi_0)$ and a function 
	$$\mathsf v:[-1,1+\epsilon)\times U_0\to W^{2,s}(S,h_0)$$
	such that all solutions of $G=0$ in a neighbourhood of $(-1,h_0,\phi_0,v_0)$ are of the form $G(H,h,\phi,\mathsf v(H,h,\phi))=0$. Exactly as in \cite{quinn}, one can then apply elliptic regularity to show that the functions $\mathsf v(H,h,\phi)$ are smooth and depend smoothly on $(H,h,\phi)$ (see e.g. \cite[Lemma 17.16]{trudinger}).
	
	Using \eqref{eq:change variables}, we then define the function $\mathsf u:[-1,1+\epsilon)\times U_0\to W^{2,s}(S,h)$ by 
	\begin{equation}\label{eq solutions u}
		\mathsf u(H,h,\phi):=\mathsf v(H,h,\phi)-\frac{1}{2}\log\left(\frac{1+H}{1-H}\right)~.
	\end{equation}
	By construction, as $H$ varies in $[-1,-1+\epsilon)$, the Epstein maps 
	$$\Eps_{(f_{\tilde\phi},e^{2\mathsf u(H,h,\phi)})}:\D\to\Hyp^3$$ then induce (smooth) CMC immersions of mean curvature $H$. We will  see in Section \ref{sec:fol ends} below that, up to choosing smaller $\epsilon$ and $U_0$, these maps are immersions. Moreover, as observed in Remark \ref{rmk:quinn2}, they induce CMC immersions in the quasi-Fuchsian manifold whose image via the map $\mathcal S$ is the point $(h,\phi)$.
	
	Of course the same argument can be applied to the other end, namely for the component $\Omega^-$ of the domain of discontinuity, and for $H$ close to $1$. This concludes the existence part in Theorem \ref{thm:foliation_ends}.

	\subsection{Foliations of the ends}\label{sec:fol ends}
	
	We now discuss the foliation part of Theorem \ref{thm:foliation_ends}.
	For this purpose, let us first outline the proof given in \cite{quinn}, to show that the ends of a given quasi-Fuchsian manifold $M$ are foliated by CMC surfaces; we will then adapt this proof in order to complete the proof of Theorem \ref{thm:foliation_ends}.
	
	\subsubsection{Outline of the foliation statement for a fixed manifold}
	
	In our notation from the previous section, Quinn's idea is to consider, for $h_0$ and $\phi_0$ fixed, the map 
	$$\widehat\Psi:S\times [-1,-1+\epsilon)\to M\cup\partial_\infty^+M$$
	which is induced in the quotient by the map $\Psi:\Omega^+\times [-1,-1+\epsilon)\to \Hyp^3\cup\Omega^+$:
	$$\Psi(z,H)=\begin{cases}
		z & \textrm{if }H=-1 \\
		\Eps_{(\mathrm{id},e^{2\mathsf u(H)})}(z) & \textrm{if }H>-1
	\end{cases}$$
	Then one would like to show that $\Psi$ is a local diffeomorphism at every $(z,-1)$, and use a compactness argument to deduce that $\widehat\Psi$ is a diffeomorphism from $S\times [-1,-1+\epsilon')$ onto its image, up to choosing $\epsilon'<\epsilon$ sufficiently small.
	
	Unfortunately, the differential of the map $\Psi$ written above is not a injective at the points $(z,-1)$. However, this is easily fixed by a reparameterization of the  parameter $H$. Set $t(H)=\sqrt{1+H}$, and write $H(t)=-1+t^2$ for $t>0$. Then we modify the map $\Psi$ above to a new map, that we call again $\Psi:\Omega^+\times [0,\delta)\to \Hyp^3\cup\Omega^+$ with an abuse of notation, for $\delta=\sqrt{1+\epsilon}$. It is defined by:
	\begin{equation}\label{eq:map corrected}
		\Psi(z,t)=\begin{cases}
			z & \textrm{if }t=0 \\
			\Eps_{(\mathrm{id},e^{2\mathsf u(H(t))})}(z) & \textrm{if }t>0
		\end{cases}
	\end{equation}
	
	The map in \eqref{eq:map corrected} is now the expression that we would like to differentiate at points $(z,t=0)$. This is easily done using the following explicit expression for the Epstein map when $\varphi=\mathrm{id}$ and $\sigma=e^{2\eta}|dz|^2$, which is a consequence of the formula \eqref{eq:SLframe}:
	
	\begin{align*}
		\Eps_{(\mathrm{id},\sigma)}(z)= (z,0)+ \frac{2}{e^{2\eta}+4|\eta_{z}|^{2}}(2\eta_{\bar z},e^{\eta})~.
	\end{align*}
	We must apply this formula to the metric $\sigma(u)=e^{2\eta}|dz|^2=e^{2u}h_0$, for 
	$$u=\mathsf u(H(t))=\mathsf v(H(t))-\frac{1}{2}\log\left(\frac{1+H(t)}{1-H(t)}\right)$$
	as in \eqref{eq solutions u}. Writing $v=\mathsf v(H(t))$ and  $\tau(v)=e^{2v}h_0=e^{2\lambda}|dz|^2$, we have 
	$$\eta=\lambda-\frac{1}{2}\log\left(\frac{1+H(t)}{1-H(t)}\right)$$
	and therefore we obtain the expression:
	
	\begin{equation}\label{eq:long}
		\begin{split}
			\Eps_{(\mathrm{id},\sigma(\mathsf u(H(t)))}(z)&= (z,0)+ \frac{2}{e^{2\lambda}+4\frac{1+H(t)}{1-H(t)}|\lambda_{\bar z}|^{2}}\left(2\frac{1+H(t)}{1-H(t)}\lambda_{z},\sqrt{\frac{1+H(t)}{1-H(t)}}e^{\lambda}\right) \\
			&=(z,0)+ \frac{2}{e^{2\lambda}+\frac{4t^2}{2-t^2}|\lambda_{z}|^{2}}\left(\frac{2t^2}{2-t^2}\lambda_{\bar z},\sqrt{\frac{t^2}{2-t^2}}e^{\lambda}\right)~.
		\end{split}
	\end{equation}
	From here, one sees that the limit as $t\to 0^+$ (that is, as $H\to -1^+$) of  $\Eps_{(\mathrm{id},\sigma(\mathsf u(H(t))))}(z)$ equals $z$. Moreover, the derivative of $\Eps_{(\mathrm{id},\sigma(\mathsf u(H(t)))}$ with respect to $t$ at $t=0$ equals $(0,\sqrt 2e^{-\varrho})$ where $\varrho$ is the density of the hyperbolic metric on $\Omega^+$ with respect to $|dz|^2$. Therefore we have
	(in real coordinates on the upper half-space): 
	
	\begin{equation}\label{eq:invertible matrix 1}
		d\Psi_{(z,0)}=\begin{pmatrix}
			1
			& 0 & 0 \\
			0 & 1 &
			0  \\
			
			0 & 0 &	\sqrt{2}e^{-\varrho}
			
		\end{pmatrix}
	\end{equation}
	which is clearly invertible.

	\subsubsection{Adaptation for Theorem \ref{thm:foliation_ends}}\label{subsec:adapt2}
	
	The above construction by Quinn is analogue to the one that we apply here, up to a modification in order to be able to choose $\epsilon'$ uniformly when
	the pair $(h,\phi)$ varies in a small neighbourhood of $(h_0,\phi_0)$. For this purpose, we modify the maps above (which we denote with the same symbol by a small abuse of notation) to:

	$$\Psi:\D\times [0,\delta)\times U_0\to \left(\Hyp^3\cup\partial_\infty\Hyp^3\right)\times U_0$$
	defined by (recall the definition of $u(H,h,\phi)$ in \eqref{eq solutions u}):
	\begin{equation}\label{eq:defi new Psi}
		\Psi(z,H,h,\phi)=\begin{cases}
			\left(f_{\tilde\phi}(z),h,\phi\right) & \textrm{if }t=0 \\
			\left(\Eps_{(f_{\tilde\phi},e^{2\mathsf u(H(t),h,\phi)})},h,\phi\right) & \textrm{if }t>0
		\end{cases}
	\end{equation}
	The map $\Psi$ therefore  induces a continuous map
	$$\widehat\Psi:S\times [0,\delta)\times U_0\to \left(M\cup\partial_\infty^+M\right)\times U_0~.$$
	The first step consists in showing that the differential of $\Psi$ (and therefore of $\widehat \Psi)$ is invertible at the points $(z,t=0)$.
	
	\begin{lemma}
		For every $z\in\D$ and every pair $(h,\phi)\in U_0$, the differential at $(z,0,h,\phi)$ of the map $\Psi:\D\times [0,\delta)\times U_0\to \left(\Hyp^3\cup\partial_\infty\Hyp^3\right)\times U_0$ defined in \eqref{eq:defi new Psi} is invertible.
	\end{lemma}
	\begin{proof}
		We clearly have that the differential of $\Psi$ is of the form
		\[
		d\Psi_{(z,0,h_0,\phi_0)}=\left(\begin{array}{c|c}
			d\Psi_{(z,0)}(\cdot,\cdot,h_0,\phi_0)
			& \star \\
			\hline
			0 & {1}
		\end{array}\right)
		\]
		Hence it suffices to check that the differential of $\Psi(\cdot,\cdot,h_0,\phi_0)$ is invertible, namely, to compute the derivatives with respect to $z$ and $t$ keeping $h$ and $\phi$ fixed.	For this, we can actually reduce to the computation we performed to obtain \eqref{eq:invertible matrix 1}. Indeed, since $f_{\tilde\phi_0}$ is a locally injective holomorphic function, we can change variables from $z$ to $w:=f_{\tilde\phi_0}(z)$ in a small open set on which $f_{\tilde\phi_0}$ is a biholomorphism onto its image. We can then consider $u$, $v$, $\eta$ and $\lambda$ as functions of $w$ instead of $z$, up to composing with a local inverse of $f_{\tilde\phi_0}$. (Of course here $\mathsf u$ and $\mathsf v$ are functions not only of $(z,H)$ but also of $(h,\phi)$, but since we are differentiating with $(h,\phi)$ fixed, the result will remain exactly the same.)
		
		We then obtain, as in \eqref{eq:long}, 
		$$\Eps_{(f_{\tilde\phi_0},\sigma(\mathsf u(H(t),h_0,\phi_0))}(w)=(w,0)+ \frac{2}{e^{2\lambda}+\frac{4t^2}{2-t^2}|\lambda_{w}|^{2}}\left(\frac{2t^2}{2-t^2}\lambda_{\bar w},\sqrt{\frac{t^2}{2-t^2}}e^{\lambda}\right)~.$$
		
		Differentiating as above, we obtain the same expression as in \eqref{eq:invertible matrix 1}, which is invertible. Since $w$ is a local coordinate and the choice of $(h_0,\phi_0)$ is arbitrary, the differential of $\Psi$ is invertible at the point $(z,0,h,\phi)$ for any $z,h,\phi$.
	\end{proof}
	
	Therefore, $\widehat \Psi$ is a local diffeomorphism in a neighbourhood of every point $(z,0,h,\phi)$. We now prove an easy topological lemma.
	
	\begin{lemma}\label{topologylemma}
		Let $X$ be a metrizable compact topological space, $Y$ any topological space and $V$ an open subset of $\mathbb{R}^{n}$ containing the origin. Let $F: X \times V \rightarrow Y$ be a continuous map such that\begin{itemize} 
			\item  $F|_{X\times \left\lbrace 0\right\rbrace }$ is injective and 
			\item $F$ is locally injective at any $(x,0)\in X\times\{0\}$. \end{itemize}Then there exists a neighbourhood $V' \subset V $ of the origin such that $F|_{X\times V'}$ is injective.
	\end{lemma}
	\begin{proof}
		Assume that there exists no such neighbourhood $V'$ where $F|_{X\times V'}$ is injective. Then there exist sequences $(x_{n},t_{n})_{n\in\mathbb N}$ and $(x'_{n},t'_{n})_{n\in\mathbb N}$ with $t_{n},t'_{n}\rightarrow 0$ such that $(x_n,t_n)\neq (x_n',t_n')$ and $F(x_{n},t_{n})=F(x'_{n},t'_{n})$. Since $X$ is metrizable and compact, it is sequentially compact, and we can extract a convergent subsequence from both $(x_{n})_{n\in\mathbb N}$ and $(x'_{n})_{n\in\mathbb N}$. Let the respective limit points be $x_{\infty}$ and $x'_{\infty}$.  By continuity of $F$ we have that $F(x_{\infty},0)=F(x'_{\infty},0)$ which  implies that $x_{\infty}=x'_{\infty}$ since $F|_{X\times \{0\}}$ is injective. But $F$ is assumed to be locally injective in a neighbourhood of $(x_{\infty},0)$, which means that for $n$ large enough,  $(x_{n},t_{n})=(x'_{n},t'_{n})$. This gives a contradiction. 
	\end{proof}
	
	We are now ready to conclude the proof of Theorem \ref{thm:foliation_ends}. Indeed by Lemma \ref{topologylemma} the map $\widehat \Psi$ is an injective local diffeomorphism, if we restrict further its domain of definition, choosing smaller $\delta$ and $U_0$. Hence it is a diffeomorphism onto its image. In particular, composing with the projection to the first factor $M\cup\partial_\infty^+M$ gives a diffeomorphism from $S\times[0,\delta)$ to its image for all $(h,\phi)$ in $U_0.$ 
	Since $H(t)=-1+t^2$ is a diffeomorphism between $(0,\delta)$ and $(-1,-1+\epsilon)$ for $\epsilon=-1+\delta^2$, we have that for every $(h,\phi)$ in $U_0$ and every $H\in(-1,-1+\epsilon)$ the Epstein maps $\Eps_{(f_{\tilde\phi},\sigma(\mathsf u(H,h,\phi))}$ induce a smooth family of embeddings in the quasi-Fuchsian manifold $M$ corresponding to $(h,\phi)$ of constant mean curvature $H$. 
	
	Of course, the same argument can be repeated for $H$ close to $-1$, obtaining a monotone CMC foliation of a neighbourhood of $\partial_\infty^-M$. Clearly, up to choosing a smaller $\epsilon$ and a smaller $U_0$, we can assume that the regions of $m\in U_0$ foliated by surfaces with CMC in $(-1,-1+\epsilon)$ and in $(1-\epsilon,1)$ are disjoint. This means that for every $m\in U_0$, these CMC surfaces foliate the complement of a compact set homeomorphic to $S\times I$. 
	This concludes Theorem \ref{thm:foliation_ends}.

	\subsection{Existence in the compact part}\label{sec:existence_compact}
	
	We now prove the existence of CMC surfaces, with mean curvature in $(-1,1)$, in a neighbourhood of any \emph{Fuchsian} manifold. Again, we will need to have some (although very weak) local uniform control on the value of the mean curvature, as in the following statement.
	
	\begin{theorem}\label{thm:existence_compact}
		Let $S$ be a closed oriented surface of genus $\geq 2$, $H_0\in (-1,1)$ and $m\in\F(S)$. Then there exists a neighbourhood $U_{H_0}$ of $m$ in $\QF$ and a constant $\epsilon=\epsilon(m,U_{H_0},H_0)$ such that, for every $H\in (H_0-\epsilon,H_0+\epsilon)$, every quasi-Fuchsian manifold in $U_{H_0}$ contains CMC surfaces with mean curvature $H$, which vary smoothly with respect to $H$. Moreover, we can assume that all such CMC surfaces have principal curvatures in $(-1,1)$. 
	\end{theorem}
	
	To prove Theorem \ref{thm:existence_compact}, we will use a similar setting as in Section \ref{sec:foliation_ends}. Roughly, the main idea is to use the implicit function theorem in order to deform the solutions to the CMC problem in a Fuchsian manifold, which are given by umbilical surfaces equidistant from the totally geodesic surface, to solutions to the CMC problem in nearby manifolds and for nearby values of the mean curvature.
	
	\begin{proof}
		The proof is very similar to Section \ref{subsec:adapt}.  After the change of variables from $(H,h,\phi,u)$ to $(H,h,\phi,v)$, where $v$ is defined in Equation \eqref{eq:change variables}, the equation of constant mean curvature equal to $H$ for the Epstein map $\Eps_{(f_{\tilde\phi},\sigma_{h}(u))}$ is equivalent to Equation \eqref{eq:us2}, which we rewrite here for the sake of convenience:
		$$
		G(H,h,\phi,v):= 1-H -2HK(\tau_h(v) ) + (-1-H)\left(K(\tau_h(v))^{2}- 16\|B(\tau_h(v))-\phi/2\|^{2}_{\tau_h(v)}\right)=0~,
		$$
		for $\tau_h(v)=e^{2v}h$. We consider again $G$ as a map from $\R\times W\times W^{2,s}(S,h_0)$ to $W^{2,s - 2}(S,h_0)$, where $h_0$ is some fixed hyperbolic metric on $S
		$. One checks directly that, for any $H_0\in (-1,1)$, the point $(H_0,h_0,\phi_0,v_0)$ is a solution, where $v_0\equiv 0$ and $\phi_0\equiv 0$. This uses that $B(h_0)=0$ because $h_0$ lifts to the Poincar\'e metric on $\D$, which is M\"obius flat, as discussed in Section \ref{subsec:mob flat}. Of course this solution corresponds geometrically to the umbilical CMC surface in the Fuchsian manifold, obtained as an equidistant surface from the totally geodesic surface. 
		
		Hence to apply the implicit function theorem for Banach spaces, we differentiate $G$ with respect to $v$. The differential of the term $\|B(\tau_{h}(v))-\phi/2\|^{2}_{\tau_h(v)}$ vanishes because 
		$$B(\tau_{h_0}(v_0))-\phi_0/2=0~,$$ for the same reason as above. We therefore have, similarly to the proof of Theorem \ref{thm:foliation_ends} (see Equation \eqref{eq:differential}):
		
		\begin{align*}
			d_{v}G_{(H_0,h_0,\phi_0,v_0)} &= (-2H_0+1+H_0) \left.\frac{d}{dv}\right|_{v=v_0} (K(e^{2v}h)) \\
			&= (1-H_0)(2\dot v-\Delta_{h_0}\dot v)
		\end{align*} 
		Since $H_0\neq 1$, $d_{v}G_{(H_0,h_0,\phi_0,v_0)}$ is invertible by Lemma \ref{lem filippo}, and we therefore obtain a family $\mathsf v:[-1,1+\epsilon)\times U_0\to W^{2,s}(S,h_0)$ of smooth solutions, depending smoothly on $H$. 
		
		Define $\mathsf u:[-1,1+\epsilon)\times U_0\to W^{2,s}(S,h_0)$ as in \eqref{eq solutions u}. We claim that the Epstein map $\Eps_{(\tilde f_{\phi_0},e^{2\mathsf u(H_0,h_0,\phi_0)}h_0)}=\Eps_{(\mathrm{id},e^{2u_0}h_0)}$, where 
		$$u_0=\mathsf u(H_0,h_0,\phi_0)=-\frac{1}{2}\log\frac{1+H_0}{1-H_0}~,$$ 
		is an immersion with first fundamental form equal to a multiple of the hyperbolic metric $h_0$. This is of course what we expect since the geometric meaning of the solution $(H_0,h_0,\phi_0,v_0)$ is the umbilical CMC surface that descends to an equidistant surface from the totally geodesic surface in the Fuchsian manifold. The claim can actually be checked without any computation, because the Poincar\'e metric on $\D$, the vanishing quadratic differential $\phi_0$ and the constant function $u_0$ are all invariant under the group of biholomorphisms of $\D$. Hence one can use the uniqueness property in Proposition \ref{prop:eps map} to deduce that there exists a surface $S$ in $\Hyp^3$, equidistant from the totally geodesic plane whose boundary  coincides with $\partial\D$, such that Epstein map $\Eps_{(\mathrm{id},e^{2u_0}h_0)}$ is the unique embedding $\iota:\D\to S\subset\Hyp^3$ satisfying 
		$$\iota\circ \zeta=\zeta\circ\iota$$
		for every biholomorphism $\zeta$ of $\D$.
		
		Since being an immersion is an open condition, up to restricting the neighbourhood $U_{H_0}$ and taking a smaller $\epsilon$, we can therefore assume that all Epstein maps
		$$\Eps_{(f_{\tilde\phi},e^{2\mathsf u(H,h,\phi)})}:\D\to\Hyp^3$$
		are immersions, which have constant mean curvature equal to $H$ by construction.
		Hence these Epstein maps induce CMC surfaces in the quotient quasi-Fuchsian manifolds corresponding to the points $(h,\phi)$ in a neighbourhood of $(h_0,\phi_0)$.
		
		The ``moreover'' part of the statement follows again by continuity, up to restricting the neighbourhood $U_{H_0}$ and taking a smaller $\epsilon$, since the principal curvatures of the umbilical CMC surface with mean curvature $H$ are both equal to $H$, and therefore smaller than one in absolute value.
	\end{proof}

	\subsection{Conclusion of existence in a small neighbourhood}\label{sec:existence}
	
	Based on Theorems \ref{thm:foliation_ends} and \ref{thm:existence_compact}, we are now ready to prove the existence of CMC surfaces for each value of the mean curvature in $(-1,1)$, in a suitable neighbourhood of the Fuchsian locus.
	
	\begin{theorem}\label{thm:existence}
		Let $S$ be a closed oriented surface of genus $\geq 2$. Then there exists a neighbourhood $U$ of the Fuchsian locus in quasi-Fuchsian space $\QF$ such that, for every $H\in (-1,1)$, every quasi-Fuchsian manifold in $U$ contains an embedded CMC surface of mean curvature $H$.
	\end{theorem}
	\begin{proof}
		We will show that, for every $m\in\F(S)$, there exists a neighbourhood $V=V(m)$ of $m$ in $\QF$ such that every $m'$ in $V$ contains embedded CMC surfaces for all $H\in (-1,1)$. Taking the union of $V(m)$ as   $m$ varies in $\F(S)$ clearly provides the claimed neighbourhood of the Fuchsian locus.
		
		Let us fix a convenient notation. For the sake of simplicity, we fix $m$ in $\F(S)$, and we will omit every dependence on $m$. Theorems \ref{thm:foliation_ends} and \ref{thm:existence_compact} provide us with:
		\begin{enumerate}
			\item A neighbourhood $\widehat U$ of $m$ and a constant $\widehat \epsilon$ such that all quasi-Fuchsian manifolds in $\widehat U$ contain embedded CMC surfaces with mean curvature $H$ ranging in $(-1,-1+\widehat\epsilon)\cup (1-\widehat\epsilon,1)$, and
			\item For every $H_0\in(-1,1)$, a neighbourhood $U_{H_0}$ of $m$ and a constant $\epsilon_{H_0}$ such that   all quasi-Fuchsian manifolds in $\widehat U$ contain immersed CMC surfaces with mean curvature $H$ ranging in $(H_0-\epsilon_{H_0},H_0+\epsilon_{H_0})$ (clearly, $\epsilon_{H_0}$ will be small enough so that $(H_0-\epsilon_{H_0},H_0+\epsilon_{H_0})\subset (-1,1)$).
		\end{enumerate}
		
		Actually, in item (2), we can assume that the immersed CMC surfaces have  principal curvatures in $(-1,1)$. This implies automatically that they are embedded, see item \ref{item1} of Proposition \ref{propsmall} below.
		
		Now, the family of intervals 
		$$\mathcal F:=\left\{[-1,-1+\widehat\epsilon)\right\}\cup \left\{(1-\widehat\epsilon,1]\right\}\cup\{(H_0-\epsilon_{H_0},H_0+\epsilon_{H_0})\,|\,H_0\in(-1,1)\}$$
		is an open covering of the compact interval $[-1,1]$, hence it admits a finite subcover
		$$\mathcal F':=\left\{[-1,-1+\widehat\epsilon)\right\}\cup \left\{(1-\widehat\epsilon,1]\right\}\cup\{(H_0-\epsilon_{H_0},H_0+\epsilon_{H_0})\,|\,H_0\in\{c_1,\ldots,c_N\}\}~.$$
		Therefore the intersection
		$$U:=\widehat U\cap U_{c_1}\cap\ldots\cap U_{c_N}$$
		is an open neighbourhood of $m$ in $\QF$ with the property that for every $H\in (-1,1)$ and for every $m'$ in $U$ there exists an embedded CMC surface with constant mean curvature $H$. This concludes the proof.
	\end{proof}
	
	In the next section, we will improve the proof of Theorem \ref{thm:existence} in order to prove that the neighbourhood $U$ can be taken so as to have the property that the embedded CMC surfaces of each quasi-Fuchsian manifold $M$ in $U$ constitute a smooth monotone foliation of $M$.

	\section{Foliations of quasi-Fuchsian manifolds}\label{sec:finish}
	
	Having established the existence of embedded CMC surfaces, for $H\in (-1,1)$, in a quasi-Fuchsian manifold in a suitably small neighbourhood of the Fuchsian locus, we now refine the construction to show that, in a possibly smaller neighbourhood, there is a monotone smooth foliation by CMC surfaces.
	
	\subsection{Small principal curvatures and equidistant foliations}
	
	We will say that a $C^2$ immersion of a surface in $\Hyp^3$ has small principal curvatures if its principal curvatures are in $(-1,1)$. The following statement contains the fundamental properties that we will use on surfaces with small principal curvatures.
	
	{\proposition {\label{propsmall}}
		Let $S$ be a closed surface and let $\iota:S\to M$ be an immersion with small principal curvatures in a quasi-Fuchsian manifold $M$ homeomorphic to $S\times\R$. Then:
		\begin{enumerate}
			\myitem{$i)$}  \label{item1} The immersion $\iota$ is an embedding and a homotopy equivalence.
			\myitem{$ii)$} \label{item2}There is a diffeomorphism $\zeta:S\times\R\to M$ such that $\zeta(\cdot,0)=\iota$, $\zeta(p,\cdot)$ is the unit speed geodesic intersecting $\iota(S)$ orthogonally at $\iota(p)$, and 
			\begin{equation}\label{eq:distance}
				d_M(\zeta(p,r_1),\zeta(p,r_2))=d_M(\zeta(S\times\{r_1\}),\zeta(p,r_2))=|r_2-r_1|~.
			\end{equation}
		\end{enumerate}
		Let us choose such $\zeta$ so that $\zeta(\cdot,r)$ approaches $\partial^-_\infty M$ as $r\to+\infty$. If moreover $\iota$ has constant mean curvature $H$, then
		\begin{enumerate}[resume]
			\myitem{$iii)$} \label{item3} The mean curvature of the surface $\zeta(S\times\{r\})$ is strictly larger than $H$ if $r>0$ and strictly smaller than $H$ if $r<0$.
			\myitem{$iv)$} \label{item4} There exist differentiable functions $f_-,f_+:\R \to \R$ satisfying $f_\pm(0)=H$ and $f'_\pm(r)>0$ for all $r$, such that the mean curvature of $\zeta(S\times\{r\})$ is between $f_-(r)$ and $f_+(r)$.
		\end{enumerate}
	}
	We will refer to the function $r:M\to\R$ as the \emph{signed distance} from the embedded surface $S=\iota(S)$. 
	\begin{proof}
		Points  \ref{item1} and  \ref{item2} are well known. For point \ref{item1}, see \cite{epstein} or \cite[Proposition 4.15, Remark 4.22]{elemamseppi}. Let $\widetilde S$ be the lift of $S=\iota(S)$ to the universal cover $\Hyp^3$. To show point \ref{item2}, the fundamental property is that $\widetilde S$ stays in the concave side of any tangent horosphere (see \cite[Lemma 4.11]{elemamseppi}), hence \emph{a fortiori} on the concave side of any tangent metric ball centered at a point $P$ outside $\widetilde S$. This implies that the geodesics orthogonal to $S$ are pairwise disjoint and form a global foliation in lines of $M$. Moreover, the distance from $S$ is realized along the orthogonal geodesic through $P$. Observe that if $S=\iota(S)$ has small principal curvatures, then all equidistant surfaces $\zeta(S\times\{r\})$ also have small principal curvatures (\cite[Chapter 3]{epstein} or \cite[Corollary 4.4]{elemamseppi}). Hence one can repeat the above argument replacing $S$ with $\zeta(\Sigma\times\{r\})$, and conclude \eqref{eq:distance} for all $r_1,r_2$.
		
		To prove points \ref{item3} and \ref{item4}, observe that, with our convention on the mean curvature (see Section \ref{subsec:mean}), the principal curvatures $\lambda_1(r),\lambda_2(r)$ of the embedding $\iota_r:=\zeta(\cdot,r):S\to M$ at the point $p$ satisfy the formula:
		\begin{equation}\label{eq:formula mean r}
			\lambda_i(p,r)=\tanh(\mu_i(p)+r)~,
		\end{equation}
		which is monotone increasing in $r$, where $\lambda_i(p,0)=\tanh\mu_i(p)\in (-1,1)$. Since the mean curvature of $\iota_r$ at $p$ equals $(\lambda_1(p,r)+\lambda_2(p,r))/2$, it follows that it is larger than $H=(\lambda_1(p,0)+\lambda_2(p,0))/2$ if $r>0$ and smaller than $H$ if $r<0$, as claimed in point \ref{item3}.
		
		More precisely, by a direct computation from Equation \eqref{eq:formula mean r} one checks that the derivative of the mean curvature function 
		$$r\mapsto H_p(r)=\frac{1}{2}\left( \lambda_1(p,r)+ \lambda_2(p,r)\right)$$ takes value in $(0,1)$ for all $r$. If we fix $r$, using compactness of $S$ we can define the  functions
		$$g_-(r_0)=\min_{p\in S}\left.\frac{d}{dr}\right|_{r=r_0}H_p(r)\qquad g_+(r_0)=\max_{p\in S}\left.\frac{d}{dr}\right|_{r=r_0}H_p(r)~.$$
		Integrating $g_-$ and $g_+$, which are both positive everywhere, from $0$ to $r$, one obtains the functions $f_-$ and $f_+$ as in point \ref{item4}. We remark that $g_\pm$ are continuous, hence integrable: indeed, using continuity in $p$ and $r$ of the $r$-derivative of $H_{p}(r)$, we see that if $r_n\to r_\infty$, then a sequence $p_n\in S$ of minimum points of $(d/dr)H_\bullet(r_n)$ converges up to a subsequence to $p_\infty$, which is necessarily a minimum point of $(d/dr)H_\bullet(r_\infty)$. Hence $g_-(r_\infty)=\lim_n g_-(r_n)$, and analogously for $g_+$ by replacing minimum by maximum.
	\end{proof}

	\subsection{Maximum principle for CMC surfaces}
	
	In this section we apply Proposition \ref{propsmall} and the geometric maximum principle for mean curvature to achieve two properties which will play a fundamental role in the proof of the foliation result, Theorem \ref{thm:foliation}.
	
	{	\proposition{\label{prop:unique}}
		Let $M\cong S\times\R$ be a quasi-Fuchsian manifold and let $S_H$ and $S_H'$ be closed embedded CMC surfaces in $M$ homotopic to $S\times\{*\}$ with the same mean curvature $H\in (-1,1)$. If $S_H$ has small principal curvatures, then $S_H=S_H'$.}
	
	\begin{proof}
		Let $r$ be the signed distance function from $S_H$, given by the diffeomorphism $\zeta$ as in Proposition \ref{propsmall}, applied to the inclusion $\iota$ of $S$ with image $S_H$. Since $S_H'$ is compact, the restriction of $r$ to $S_H'$ has a maximum $r_{\max}=r(p_{\max})$ and a minimum $r_{\min}=r(p_{\min})$. By Remark \ref{signconv}, the normal vector to $S_H'$ coincides with minus the gradient of the function $r$ at the points $p_{\min}$ and $p_{\max}$.
		
		This implies that $S_H'$ is tangent to the equidistant surface $\zeta(S\times\{r_{\max}\})$, and entirely contained in the side $\{r\leq r_{\max}\}$, towards which the normal vector is pointing by our convention. By the geometric maximum principle, the mean curvature of $S_H'$, which equals $H$, is larger than the mean curvature of  $\zeta(S\times\{r_{\max}\})$ at $p_{\max}$. By item \ref{item3} of Proposition \ref{propsmall}, $r_{\max}\leq 0$. Repeating the argument for the minimum point, one obtains $r_{\min}\geq 0$. Hence $r\equiv 0$ on $S_H'$. Since both $S_H$ and $S_H'$ are closed embedded surfaces, they must coincide. 
	\end{proof}
	
	Let us now consider the case of two CMC surfaces with different values of the mean curvature.
	
	\begin{lemma}\label{lemma:disjoint}
		Let $M\cong S\times\R$ be a quasi-Fuchsian manifold and let $S_H$ and $S_{H'}$ be closed embedded CMC surfaces in $M$ homotopic to $S\times\{*\}$, with mean curvature $H$ and $H'$ respectively, for $H\neq H'$. If $S_H$ has small principal curvatures, then $S_H$ and $S_{H'}$ are disjoint, and moreover the signed distance of every point of $S_{H'}$ from $S_H$ is between $f_+^{-1}(H')$ and $f_-^{-1}(H')$, where $f_\pm$ are the increasing functions introduced in Proposition \ref{propsmall}.
	\end{lemma}
	\begin{proof}
		The proof is very similar to Proposition \ref{prop:unique}. Suppose $H'>H$, the other case being analogous. Consider the restriction to $S_{H'}$ of the signed distance function $r$ with respect to $S_H$. This functions admits a minimum $r_{\min}=r(p_{\min})$ and a maximum $r_{\max}=r(p_{\max})$. Hence $S_{H'}$ is tangent to $\zeta(S\times\{r_{\min}\})$ at $p_{\min}$ and to $\zeta(S\times\{r_{\max}\})$ at $p_{\max}$, and contained in the region $\{r_{\min}\leq r\leq r_{\max}\}$ between the two. The geometric maximum principle together with item \ref{item4} of Proposition \ref{propsmall} then implies that 
		$$f_-(r_{\max})\leq H'\leq f_+(r_{\min})~.$$
		This implies that the restriction of $r$ to $S_{H'}$ is at least $r_{\min}\geq f_+^{-1}(H')$, and at most $r_{\max}\leq f_-^{-1}(H')$, as claimed.
	\end{proof}
	
	\subsection{Proof of Theorem \ref{thm:foliation}}
	Let us now conclude the proof of the smooth monotone foliation result, by putting together all the ingredients. The aim is showing that, for $M$ a quasi-Fuchsian manifold in a suitable neighbourhood $U$ of the Fuchsian locus $\F(S)$, there exists a diffeomorphism between $S\times (-1,1)$ and $M$ such that, restricted to each slice $S\times \{H\}$, is an embedding of constant mean curvature $H$. The existence of such CMC surfaces has been proved in Theorem \ref{thm:existence}, so now the goal (up to choosing a smaller neighbourhood $U$ of $\F(S)$) is achieving the diffeomorphism, thus proving the smooth foliation part.
	
	\begin{proof}[Proof of Theorem \ref{thm:foliation}]
		Recall that the proof of Theorem \ref{thm:existence} produces, for every $m$ in $\F(S)$, a neighbourhood $U$ in $\QF$ as the intersection 
		\begin{equation}\label{eq:finite subcover}
			U:=\widehat U\cap U_{c_1}\cap\ldots\cap U_{c_N}~,
		\end{equation}
		where $\widehat U$ is a neighbourhood of $m$ in which the ends are monotonically foliated by CMC surfaces with mean curvature ranging in $(-1,-1+\widehat\epsilon)\cup (1-\widehat\epsilon,1)$, and the $U_{c_i}$ are neighbourhoods of $m$ obtained from the family $U_{H_0}$ (by extracting a finite cover of the interval $[-1,1]$). Hence for every $i$, in every quasi-Fuchsian manifold inside $U_{c_i}$ we have existence of CMC surfaces of mean curvature ranging in $(c_i-\epsilon_{c_i},c_i+\epsilon_{c_i})$.
		
		Now, let us provide a couple of preliminary observations. First, from Theorem \ref{thm:existence_compact}, we can assume that the $U_{H_0}$ have the property that the CMC surfaces of mean curvature $(H_0-\epsilon_{H_0},H_0+\epsilon_{H_0})$ have small principal curvatures. (In particular, they are embedded by \ref{item1} of Proposition \ref{propsmall}.) Hence in \eqref{eq:finite subcover}, we can assume that all the $U_{c_i}$ have this property. Second, it is harmless to assume that $c_1<\ldots<c_N$ and that the corresponding intervals, namely $(-1,-1+\widehat\epsilon),(c_1-\epsilon_{c_1},c_1+\epsilon_{c_1}), \ldots, (c_N-\epsilon_{c_N},c_N+\epsilon_{c_N}), (1-\widehat\epsilon,1)$ only intersect in pairs (that is, each interval intersects the previous and the next one, and no other), up to choosing smaller $\epsilon$'s.  
		
		Having made these assumptions, using Theorem \ref{thm:existence_compact} we can construct, for any quasi-Fuchsian manifold $M$ in $U$, smooth maps 
		$$\xi_{i}:S\times (c_i-\epsilon_{c_i},c_i+\epsilon_{c_i})\to M$$ 
		having the property that $\xi_{c_i}(S\times\{H\})$ is an embedded CMC surface of mean curvature $H$. Similarly in the ends, from Theorem  \ref{thm:foliation_ends} we get smooth maps 
		$$\xi_0:S\times (-1,-1+\widehat\epsilon)\to M\qquad\text{and}\qquad \xi_{N+1}:S\times (1-\widehat\epsilon,1)\to M$$
		satisfying the analogous property.
		
		By our previous assumption, all the $\xi_{i}(S\times\{H\})$ have small principal curvatures, if $i\in\{1,\ldots,N\}$. Hence by Proposition \ref{prop:unique}, we have $\xi_{i}(S\times\{H\})=\xi_{i'}(S\times\{H\})$ for every $i,i'\in\{0,\ldots,N+1\}$. Using our other assumption, namely that only consecutive intervals overlap, we can iteratively precompose each $\xi_i$, starting from $\xi_1$, with smooth diffeomorphisms of the source that preserve each slice $S\times\{H\}$, so that $\xi_i(\cdot,H)=\xi_{i+1}(\cdot,H)$ as long as $H$ is in the intersection of the corresponding intervals. Hence we can glue together the $\xi_i$'s to obtain a smooth map 
		$$\xi
		:S\times(-1,1)\to M$$
		such that $\xi(S\times\{H\})$ is an embedding of a CMC surface with mean curvature $H$, which we denote by $S_H$. 
		
		We claim that $\xi$ is injective. Indeed it is injective on every slice $S\times\{H\}$, hence it suffices to show that the images of different slices are disjoint. We distinguish three cases. If $H$ is in one of the intervals $(c_i-\epsilon_{c_i},c_i+\epsilon_{c_i})$, then $S_H$ is disjoint by any $S_{H'}$ for $H'\neq H$ by  Lemma \ref{lemma:disjoint}. If $H\in (-1,-1+\widehat\epsilon)$ and $H'\in (1-\widehat\epsilon,1)$, then $S_H$ and $S_{H'}$ are disjoint because the two neighbourhoods of the ends are disjoint. Finally, if both $H$ and $H'$ are in $(-1,-1+\widehat\epsilon)$ or in $(1-\widehat\epsilon,1)$, then $S_H$ and $S_{H'}$ are disjoint by Theorem  \ref{thm:foliation_ends}.
		
		Moreover $\xi$ is surjective by the intermediate value theorem, because it is a diffeomorphism onto a neighbourhood of the ends when restricted to $S\times (-1,-1+\widehat\epsilon)$ and $S\times (1-\widehat\epsilon,1)$ by Theorem \ref{thm:foliation_ends}. Hence $\xi$ is a homeomorphism. By the inverse function theorem, to prove that it is a diffeomorphism, and thus conclude  Theorem  \ref{thm:foliation}, it suffices to show that its differential is injective at every $(p,H)$ with $H$ in one of the intervals $(c_i-\epsilon_{c_i},c_i+\epsilon_{c_i})$. 
		
		For this purpose, we know already that the differential of $\xi$ is injective when restricted to $T_p S \subset T_{(p,H)}(S\times(-1,1))$, and $d\xi(T_p S)$ is the tangent space to the CMC surface which we will call $S_H$. Hence it suffices to show that $d\xi(\partial/\partial H)$ is a nonzero vector transverse to $d\xi_{(p,H)}(T_p S)=T_{\xi(p,H)}S_H$. Here we use that $S_H$ has small principal curvatures, and the equidistant foliation from Proposition \ref{propsmall}. Indeed, it is sufficient to show that $d(r\circ\xi)(\partial/\partial H)$ does not vanish, where $r$ is the signed distance from $S_H$ provided by Proposition \ref{propsmall}. But the last part of Lemma \ref{lemma:disjoint} tells us that  $r\circ\xi$ (which is a differentiable function) is larger than the function $f_+^{-1}$, whose derivative is positive. Hence 
		$$\left.\frac{d}{dt}\right|_{t=H}(r\circ\xi)(p,t)>0~.$$
		This concludes the proof.
	\end{proof}

\chapter{ Future research}
A continuation of the project in \cite{cmcdifian} is to consider the flow in $T^{*}\T$ that we obtain by considering the path $(-1,1)\rightarrow T^{*}\T$ given by $t\mapsto ([c_{t}],q_{t})$ where the first fundamental form of the CMC surface with mean curvature $t$ is in the conformal class $[c_{t}]$  and the traceless part of the second fundamental form is given by the real part $\mathfrak{R}(q_{t})$ (see\cite{Tromba1992-oa} for example). The statement we then prove is that:
{\theorem The flow $(-1,1): t\mapsto ([c_{t}],q_{t})$ are orbits of a Hamiltonian flow on $T^{*}\T$ with respect to its cotangent symplectic structure and the Hamiltonian function is given by $-\frac{1}{2}$ times area of the CMC surfaces.}\\\\
Note that similar results for CMC foliations in Lorentzian setting and their relations with Teichmüller theory have already been studied in \cite{moncrief}. Also, as an extension of my Ph.D project, there is a related question that I will like to consider regarding {prescribing Schwarzians at infinity} of quasi-Fuchsian manifolds. For a given measured foliation $\f$ we can consider the set $O(\f)\in \MF$ which consists of all measured foliations $\g'$ such that $(\f,\g')$ fill for any $\g'\in O(\f)$. We can then ask:
{\question Given a measured foliations $\f$, does there exist a unique quasi-Fuchsian manifold with Schwarzians at infinity $([c_{+}],tq^{\f}_{[c_{+}]})$ and $([c_{-}],tq^{\g'}_{[c_{-}]})$ for some $t>0$ where $[c_{+}],[c_{-}]\in \T$ are the equivalence classes of complex structures appearing at the boundary at infinity and for any $g'\in O(\f)$ ?}. \\\\
For $t$ small enough and restricted to arational case, Theorem \ref{thm1.1} provides an answer for quasi-Fuchsian manifolds near the Fuchsian locus although. If one consider the case of almost-Fuchsian manifolds, then there are again some well established results regarding the description of the complex structures at the ends of an almost-Fuchsian manifold in terms of the immersion data of the minimal surface which maybe useful for this purpose, see for instance\cite{Trautwein2019} (Proposition $5.6$) and also\cite{Krasnov2007}. In fact, we have shown in\cite{dip} that for a quasi-Fuchsian manifold near the Fuchsian locus, the Schwarzians at infinity for the paths $\beta_{([c],q)}(t)$ are in fact determined at first order by the holomorphic quadratic differential $q\in T^{*}\T$ such that the real part $\mathfrak{R}(q)$ is equal to the second fundamental form $\II$ of the unique immersed minimal surface it contains.\\\\  
Another aspect can be to consider a well known result (see\cite{sullivan},\cite{Epstein2005},\cite{yarmola}) that the hyperbolic metric on the positive boundary component of $\mathcal{CC}(M)$ and the conformal class at boundary at positive infinity of a quasi-Fuchsian manifold are uniformly close in $\T$, in fact quasi-conformal to each other by a factor $\leq 2.1$ and ask a similar comparative question for measured foliations at infinity and measured bending lamination at infinity. We will denote the space of equivalence classes of measured laminations on $S$ as $\ML$  where we also have that $\MF\cong \ML$ (see\cite{mosher}). So we can ask:

{\question\label{qw} Let $\lambda_{+}$ and $\fp$ be the measured bending lamination and the measured foliation at infinity at the positive ends of the boundary of $\mathcal{CC}(M)$ and the boundary at infinity respectively. Then are they uniformly close in some sense in the space of equivalence class of measured geodesic lamination $\ML$?}\\\\
Moreover investigating the existence of measured foliations at infinity or prescribing Schwarzians at infinity in convex-cocompact hyperbolic $3$-manifolds or higher dimensional hyperbolic manifolds is something I aspire to do in the future.\\\\
One more question we can ask is about the limiting behaviour of a sequence of quasi-Fuchsian manifolds $g_{n}$ based on their measured foliations at infinity $(\f_{n},\g_{n})$. Recall here that there is a well-defined notion of convergence in the space of measured foliation where the limit of a sequence of measured foliations in $\MF$ converge to a point in the space of projective measured foliations $P\MF$. In particular we want to ask the following question:
{\question Let $(\f_{n},\g_{n})$ converge to the pair of filling projective measured foliations $([\f],[\g])$. Then does $g_{n}$ have a converging subsequence? If so, is the limit the ending lamination for the sequence $g_{n}$?  }
Here by ending lamination we mean the limit of the measured bending lamination on the boundary of the convex core of $g_{n}$ and has been a very well-studied entity in recent times. \\ \\
To finish, there is one more related problem I am interested in at the moment which concerns the intersection of the sections $q^{\f},q^{-\g}:\T \rightarrow T^{*}\T$ for a filling pair $(\f,\g)$ and asking if the intersection of the sections is transverse in $T^{*}\T$? It is clear that they intersect uniquely at the point $([7]c],q)\in T^{*}\T$ which realise $\f$ and $\g$ as its horizontal and vertical measured foliations as par Gardiner-Masur theorem. The transversality of their intersection on the other hand can be reformulated into the following question:

{\question Let $(\f,\g)$ be a filling pair and $([c],q)\in T^{*}\T$ be the unique holomorphic quadratic differential realising them. Consider a first-order deformation given by $t\mapsto ([c_{t}],q_{t}), t\geq 0$ such that $[c_{0}]=[c]$ and $q_{0}=q$ and which maintains $(\f,\g)$ as its measured foliation at first-order at $t=0$. Then, is this deformation necessarily trivial?  } \\

Here, we say $t\mapsto \f_{t}\in \MF$ with $\f_{0}=\f$ is said to be equal to $\f$ at first order at $t=0$ if $\ddt i(\gamma,\f_{t})=0$ for any simple closed curve $\gamma$ on $S$ and where $i(\gamma,.):\MF\rightarrow \mathbb{R}_{\geq 0}$ is the intersection number of $\gamma$ with a given measured foliation. The answer is positive when we restrict to holomorphic quadratic differentials in the dense generic stratum, i.e, when all the zeroes of $q_{t}$ are simple. Moreover, for genus $2$ surface it can be shown that the question above has a positive answer. The main difficulty lies in analysing deformations that collapse or join the zeroes of $q$ for arbitrary genus.\\\\

		\bibliographystyle{abbrv}
	\bibliography{versionfinal2,main,cmcbiblio}

\end{document}